\documentclass[12pt]{amsart}
\usepackage{latexsym}
\usepackage{indentfirst}
\usepackage{amssymb,amsfonts,amsmath,amsthm}
\usepackage{graphicx}
\usepackage{mathrsfs}
\usepackage{array}
\usepackage{color}
\usepackage{epstopdf}
\usepackage[all]{xy}
\usepackage{verbatim}
\usepackage{url}
\usepackage{tikz-cd}
\usepackage{relsize}

\usepackage[linktocpage=true, hidelinks, colorlinks, linkcolor=blue, citecolor=red, bookmarksopen=true, urlcolor=brown]{hyperref}

\usepackage[a4paper,bindingoffset=0.2in,           left=0.5in,right=0.51in,top=0.6in,bottom=0.5in,            footskip=.25in]{geometry}

\usepackage[bbgreekl]{mathbbol}
\DeclareSymbolFontAlphabet{\mathbb}{AMSb} %to ensure that the meaning of \mathbb does not change
\DeclareSymbolFontAlphabet{\mathbbl}{bbold}

\newtheorem{thm}{Theorem}[section]
\newtheorem{prop}[thm]{Proposition}
\newtheorem{lem}[thm]{Lemma}
\newtheorem{conj}[thm]{Conjecture}
\newtheorem{cor}[thm]{Corollary}

\numberwithin{equation}{section}

\theoremstyle{definition}

\newtheorem{construction}[thm]{Construction}

\newtheorem{convention}[thm]{Convention}

\newtheorem{dfn}[thm]{Definition}
\newtheorem{exam}[thm]{Example}

\newtheorem{rmk}[thm]{Remark}

\begin{document}

\newcommand{\fraka}{{\mathfrak a}}
\newcommand{\frakb}{{\mathfrak b}}
\newcommand{\frakc}{{\mathfrak c}}
\newcommand{\frakd}{{\mathfrak d}}
\newcommand{\frake}{{\mathfrak e}}
\newcommand{\frakf}{{\mathfrak f}}
\newcommand{\frakg}{{\mathfrak g}}
\newcommand{\frakh}{{\mathfrak h}}
\newcommand{\fraki}{{\mathfrak i}}
\newcommand{\frakj}{{\mathfrak j}}
\newcommand{\frakk}{{\mathfrak k}}
\newcommand{\frakl}{{\mathfrak l}}
\newcommand{\frakm}{{\mathfrak m}}
\newcommand{\frakn}{{\mathfrak n}}
\newcommand{\frako}{{\mathfrak o}}
\newcommand{\frakp}{{\mathfrak p}}
\newcommand{\frakq}{{\mathfrak q}}
\newcommand{\frakr}{{\mathfrak r}}
\newcommand{\fraks}{{\mathfrak s}}
\newcommand{\frakt}{{\mathfrak t}}
\newcommand{\fraku}{{\mathfrak u}}
\newcommand{\frakv}{{\mathfrak v}}
\newcommand{\frakw}{{\mathfrak w}}
\newcommand{\frakx}{{\mathfrak x}}
\newcommand{\fraky}{{\mathfrak y}}
\newcommand{\frakz}{{\mathfrak z}}

\newcommand{\frakA}{{\mathfrak A}}
\newcommand{\frakB}{{\mathfrak B}}
\newcommand{\frakC}{{\mathfrak C}}
\newcommand{\frakD}{{\mathfrak D}}
\newcommand{\frakE}{{\mathfrak E}}
\newcommand{\frakF}{{\mathfrak F}}
\newcommand{\frakG}{{\mathfrak G}}
\newcommand{\frakH}{{\mathfrak H}}
\newcommand{\frakI}{{\mathfrak I}}
\newcommand{\frakJ}{{\mathfrak J}}
\newcommand{\frakK}{{\mathfrak K}}
\newcommand{\frakL}{{\mathfrak L}}
\newcommand{\frakM}{{\mathfrak M}}
\newcommand{\frakN}{{\mathfrak N}}
\newcommand{\frakO}{{\mathfrak O}}
\newcommand{\frakP}{{\mathfrak P}}
\newcommand{\frakQ}{{\mathfrak Q}}
\newcommand{\frakR}{{\mathfrak R}}
\newcommand{\frakS}{{\mathfrak S}}
\newcommand{\frakT}{{\mathfrak T}}
\newcommand{\frakU}{{\mathfrak U}}
\newcommand{\frakV}{{\mathfrak V}}
\newcommand{\frakW}{{\mathfrak W}}
\newcommand{\frakX}{{\mathfrak X}}
\newcommand{\frakY}{{\mathfrak Y}}
\newcommand{\frakZ}{{\mathfrak Z}}

\newcommand{\bA}{{\mathbb A}}
\newcommand{\bB}{{\mathbb B}}
\newcommand{\bC}{{\mathbb C}}
\newcommand{\bD}{{\mathbb D}}
\newcommand{\bE}{{\mathbb E}}
\newcommand{\bF}{{\mathbb F}}
\newcommand{\bG}{{\mathbb G}}
\newcommand{\bH}{{\mathbb H}}
\newcommand{\bI}{{\mathbb I}}
\newcommand{\bJ}{{\mathbb J}}
\newcommand{\bK}{{\mathbb K}}
\newcommand{\bL}{{\mathbb L}}
\newcommand{\bM}{{\mathbb M}}
\newcommand{\bN}{{\mathbb N}}
\newcommand{\bO}{{\mathbb O}}
\newcommand{\bP}{{\mathbb P}}
\newcommand{\bQ}{{\mathbb Q}}
\newcommand{\bR}{{\mathbb R}}
\newcommand{\bS}{{\mathbb S}}
\newcommand{\bT}{{\mathbb T}}
\newcommand{\bU}{{\mathbb U}}
\newcommand{\bV}{{\mathbb V}}
\newcommand{\bW}{{\mathbb W}}
\newcommand{\bX}{{\mathbb X}}
\newcommand{\bY}{{\mathbb Y}}
\newcommand{\bZ}{{\mathbb Z}}

\newcommand{\bfA}{{\mathbf A}}
\newcommand{\bfB}{{\mathbf B}}
\newcommand{\bfC}{{\mathbf C}}
\newcommand{\bfD}{{\mathbf D}}
\newcommand{\bfE}{{\mathbf E}}
\newcommand{\bfF}{{\mathbf F}}
\newcommand{\bfG}{{\mathbf G}}
\newcommand{\bfH}{{\mathbf H}}
\newcommand{\bfI}{{\mathbf I}}
\newcommand{\bfJ}{{\mathbf J}}
\newcommand{\bfK}{{\mathbf K}}
\newcommand{\bfL}{{\mathbf L}}
\newcommand{\bfM}{{\mathbf M}}
\newcommand{\bfN}{{\mathbf N}}
\newcommand{\bfO}{{\mathbf O}}
\newcommand{\bfP}{{\mathbf P}}
\newcommand{\bfQ}{{\mathbf Q}}
\newcommand{\bfR}{{\mathbf R}}
\newcommand{\bfS}{{\mathbf S}}
\newcommand{\bfT}{{\mathbf T}}
\newcommand{\bfU}{{\mathbf U}}
\newcommand{\bfV}{{\mathbf V}}
\newcommand{\bfW}{{\mathbf W}}
\newcommand{\bfX}{{\mathbf X}}
\newcommand{\bfY}{{\mathbf Y}}
\newcommand{\bfZ}{{\mathbf Z}}

\newcommand{\calA}{{\mathcal A}}
\newcommand{\calB}{{\mathcal B}}
\newcommand{\calC}{{\mathcal C}}
\newcommand{\calD}{{\mathcal D}}
\newcommand{\calE}{{\mathcal E}}
\newcommand{\calF}{{\mathcal F}}
\newcommand{\calG}{{\mathcal G}}
\newcommand{\calH}{{\mathcal H}}
\newcommand{\calI}{{\mathcal I}}
\newcommand{\calJ}{{\mathcal J}}
\newcommand{\calK}{{\mathcal K}}
\newcommand{\calL}{{\mathcal L}}
\newcommand{\calM}{{\mathcal M}}
\newcommand{\calN}{{\mathcal N}}
\newcommand{\calO}{{\mathcal O}}
\newcommand{\calP}{{\mathcal P}}
\newcommand{\calQ}{{\mathcal Q}}
\newcommand{\calR}{{\mathcal R}}
\newcommand{\calS}{{\mathcal S}}
\newcommand{\calT}{{\mathcal T}}
\newcommand{\calU}{{\mathcal U}}
\newcommand{\calV}{{\mathcal V}}
\newcommand{\calW}{{\mathcal W}}
\newcommand{\calX}{{\mathcal X}}
\newcommand{\calY}{{\mathcal Y}}
\newcommand{\calZ}{{\mathcal Z}}

\newcommand{\rA}{{\mathrm A}}
\newcommand{\rB}{{\mathrm B}}
\newcommand{\rC}{{\mathrm C}}
\newcommand{\rD}{{\mathrm D}}
\newcommand{\rd}{{\mathrm d}}
\newcommand{\rE}{{\mathrm E}}
\newcommand{\rF}{{\mathrm F}}
\newcommand{\rG}{{\mathrm G}}
\newcommand{\rH}{{\mathrm H}}
\newcommand{\rI}{{\mathrm I}}
\newcommand{\rJ}{{\mathrm J}}
\newcommand{\rK}{{\mathrm K}}
\newcommand{\rL}{{\mathrm L}}
\newcommand{\rM}{{\mathrm M}}
\newcommand{\rN}{{\mathrm N}}
\newcommand{\rO}{{\mathrm O}}
\newcommand{\rP}{{\mathrm P}}
\newcommand{\rQ}{{\mathrm Q}}
\newcommand{\rR}{{\mathrm R}}
\newcommand{\rS}{{\mathrm S}}
\newcommand{\rT}{{\mathrm T}}
\newcommand{\rU}{{\mathrm U}}
\newcommand{\rV}{{\mathrm V}}
\newcommand{\rW}{{\mathrm W}}
\newcommand{\rX}{{\mathrm X}}
\newcommand{\rY}{{\mathrm Y}}
\newcommand{\rZ}{{\mathrm Z}}

%algebras
\newcommand{\Zp}{{\bZ_p}}
\newcommand{\Zl}{{\bZ_l}}
\newcommand{\Qp}{{\bQ_p}}
\newcommand{\Ql}{{\bQ_l}}
\newcommand{\Cp}{{\bC_p}}
\newcommand{\Fp}{{\bF_p}}
\newcommand{\Fl}{{\bF_l}}

\newcommand{\Ainf}{{\mathbf{A}_{\mathrm{inf}}}}
\newcommand{\AdR}{{\mathbf{A}_{\dR}}}
\newcommand{\BdRp}{{\mathbf{B}_{\mathrm{dR}}^+}}
\newcommand{\BdR}{{\mathbf{B}_{\mathrm{dR}}}}

%sheaves
\newcommand{\OXp}{{\widehat \calO_X^+}}
\newcommand{\OX}{{\widehat \calO_X}}
\newcommand{\AAinf}{{\bA_{\mathrm{inf}}}}            %absolutely A_inf
\newcommand{\AAinfK}{{\bA_{\mathrm{inf},K}}} 
\newcommand{\AAdR}{{\bA_{\dR}}}
%absolutely A_dR
\newcommand{\BBinf}{{\bB_{\mathrm{inf}}}}            %absolutely B_inf
\newcommand{\BBbd}{{\bB^{\mathrm{bd}}}}              %absolutely B^{bd}
\newcommand{\BBdRp}{{\bB_{\mathrm{dR}}^+}}           %absolutely B_dR+
\newcommand{\BBdRpn}{{\bB_{\mathrm{dR},n}^+}}     
\newcommand{\BBdR}{{\bB_{\mathrm{dR}}}}              %absolutely B_dR
\newcommand{\OAinf}{{\calO\bA_{\mathrm{inf}}}}         %OA_inf
\newcommand{\OBinf}{{\calO\bB_{\mathrm{inf}}}}         %OB_inf
\newcommand{\OBdRp}{{\calO\bB_{\mathrm{dR}}^+}}        %OB_dR+
\newcommand{\OBdR}{{\calO\bB_{\mathrm{dR}}}}           %OB_dR
\newcommand{\OC}{{\calO\bC}}                           %OC (=gr_0(OB_dR+) = Hyodo's generalised Cp)

%Symbols
\newcommand{\Aut}{{\mathrm{Aut}}}
\newcommand{\Bun}{{\mathrm{Bun}}}           %Bundle
\newcommand{\Cech}{{\check{H}}}             %Cech cohomology
\newcommand{\Coker}{{\mathrm{Coker}}}       %cokernal
\newcommand{\colim}{{\mathrm{colim}}}       %colimit
\newcommand{\dlog}{{\mathrm{dlog}}}         %log differential
\newcommand{\End}{{\mathrm{End}}}           %End-functor
\newcommand{\Exp}{{\mathrm{Exp}}}           %Big p-adic exponential
\newcommand{\Ext}{{\mathrm{Ext}}}           %Ext-functor
\newcommand{\Fil}{{\mathrm{Fil}}}           %Filtration
\newcommand{\Fitt}{{\mathrm{Fitt}}}         %Fitting ideal
\newcommand{\Frac}{{\mathrm{Frac}}}         %fractional fields
\newcommand{\Frob}{{\mathrm{Frob}}}         %Frobenius
\newcommand{\Gal}{{\mathrm{Gal}}}           %Galois groups
\newcommand{\Gr}{{\mathrm{Gr}}}             %Grading
\newcommand{\Hom}{{\mathrm{Hom}}}           %Hom-functor
\newcommand{\HIG}{{\mathrm{HIG}}}           %Higgs module
\newcommand{\id}{{\mathrm{id}}}             %Identity
\newcommand{\Ima}{{\mathrm{Im}}}            %image
\newcommand{\Isom}{{\mathrm{Isom}}}         %Isomorphic class
\newcommand{\Ker}{{\mathrm{Ker}}}           %kernel
\newcommand{\Lie}{{\mathrm{Lie}}}           %Lie algebra
\newcommand{\Log}{{\mathrm{Log}}}           %Big p-adic logarithm
\newcommand{\Mat}{{\mathrm{Mat}}}           % Matrix
\newcommand{\MIC}{{\mathrm{MIC}}}           %Module of integrable connection
\newcommand{\Mod}{{\mathrm{Mod}}}           %Module
\newcommand{\Perfd}{{\mathrm{Perfd}}}         %Perfectoid spaces
\newcommand{\pr}{{\mathrm{pr}}}             %Projection
\newcommand{\Proj}{\mathrm{Proj}}        %Projective space
\newcommand{\Rep}{{\mathrm{Rep}}}           %Representations
\newcommand{\Res}{{\mathrm{Res}}}           % Restriction
\newcommand{\RGamma}{{\mathrm{\rR\Gamma}}}    %right-derived Gamma
\newcommand{\rk}{{\mathrm{rk}}}             %rank
\newcommand{\Rlim}{{\mathrm{R}\underleftarrow{\lim}}} %derived projective limit
\newcommand{\sgn}{{\mathrm{sgn}}}           %Signature
\newcommand{\Sh}{{\mathrm{Shv}}}             %Shimura varieties
\newcommand{\Sht}{{\mathrm{Sht}}}           %Shtukas
\newcommand{\Spa}{{\mathrm{Spa}}}           %Affinoid adic spaces
\newcommand{\Spf}{{\mathrm{Spf}}}           %formal affine schemes
\newcommand{\Spec}{{\mathrm{Spec}}}         %affine schemes
\newcommand{\Strat}{{\mathrm{Strat}}}       %Stratification
\newcommand{\Sym}{{\mathrm{Sym}}}           %Symmetic algebra
\newcommand{\Tor}{{\mathrm{{Tor}}}}         %Tor-functor
\newcommand{\Tot}{{\mathrm{Tot}}}           %Total complex
\newcommand{\Vect}{{\mathrm{Vect}}}         %Vector bundle

%%%%%%%%%%%
%algebraic groups

\newcommand{\GL}{{\mathrm{GL}}}             % genetal linear group
\newcommand{\SL}{{\mathrm{SL}}}             % special linear group

%Notations

\newcommand{\aff}{{\mathrm{aff}}}          % affine or affinoid
\newcommand{\an}{{\mathrm{an}}}             %analytification
\newcommand{\can}{{\mathrm{can}}}           %canonical
\newcommand{\cl}{{\mathrm{cl}}}             % (co-)homological class
\newcommand{\cofib}{{\mathrm{cofib}}}        %cofiber
\newcommand{\cts}{{\mathrm{cts}}}           % continuous
\newcommand{\cris}{{\mathrm{cris}}}         % cristalline
\newcommand{\crys}{{\mathrm{crys}}}         % crystalline
\newcommand{\cyc}{{\mathrm{cyc}}}           % cyclotomic
\newcommand{\dR}{{\mathrm{dR}}}             % de Rham
\newcommand{\et}{{\mathrm{\acute{e}t}}}    % etale
\newcommand{\fib}{{\mathrm{fib}}}          %fiber
\newcommand{\fl}{{\mathrm{fl}}}             %flat
\newcommand{\fppf}{{\mathrm{fppf}}}         % finite presented faithfully flat
\newcommand{\ket}{{\mathrm{k\acute{e}t}}}    %Kummer etale
\newcommand{\geo}{{\mathrm{geo}}}           % geometry
\newcommand{\gp}{{\mathrm{gp}}}             % group
\newcommand{\la}{{\mathrm{la}}}             %local analytic
\newcommand{\nil}{{\mathrm{nil}}}           %nilpotent
\newcommand{\pd}{{\mathrm{pd}}}             %pd-power
\newcommand{\perf}{\mathrm{perf}}           % perfection
\newcommand{\proet}{{\mathrm{pro\acute{e}t}}}% pro-etale
\newcommand{\proket}{{prok\mathrm{\acute{e}t}}}     %pro-Kummer-etale
\newcommand{\rig}{{\mathrm{rig}}}            %rigid
\newcommand{\sm}{{\mathrm{sm}}}             %smooth
\newcommand{\st}{{\mathrm{st}}}             % semi stable
\newcommand{\tor}{{\mathrm{tor}}}           % torsion
\newcommand{\Zar}{{\mathrm{Zar}}}           % Zariski

%以上是一些常用的数学记号

%以下是Prismatic theory的常见符号
%\DeclareSymbolFontAlphabet{\mathbb}{AMSb}
%\DeclareSymbolFontAlphabet{\mathbbl}{bbold}

\newcommand{\Prism}{{\mathlarger{\mathbbl{\Delta}}}} % Prism
\newcommand{\OPrism}{{\calO_{\Prism}}}
\newcommand{\IPrism}{{\calI_{\Prism}}}

%左右尖括号
\newcommand{\ya}{{\rangle}}
\newcommand{\za}{{\langle}}

%matrix
\newcommand{\smat}[1]{\left( \begin{smallmatrix} #1 \end{smallmatrix} \right)}

%临时记号
\newcommand{\AAKK}{{\bA_{2,K}}}

%%%%%%%%%%%%%%%%%%%%%%%%%%%%%%%%%%%%%

%%%%%%%%%%%%%%%%%%%%%%%%%%%%%%%%%%%%%

%New symbols
\newcommand{\Ctilde}{\widetilde{\bfC}}
\newcommand{\CCtilde}{\widetilde{\bC}}
\newcommand{\Cpsi}{\widetilde{\mathbb{C}}_\psi^I}
\newcommand{\CCpsi}{\widetilde{\mathbb{C}}_{\tilde{\psi}}^I}
\newcommand{\Cppsi}{\widetilde{\mathbb{C}}_\psi^{I, +}}
\newcommand{\CCppsi}{\widetilde{\mathbb{C}}_{\tilde{\psi}}^{I,+}}
\newcommand{\Bnpsi}{\widetilde{\bfB}_{\psi,n}}
\newcommand{\Bpsi}{\widetilde{\mathbf{B}}_{\psi}}
\newcommand{\Bnppsi}{\widetilde{\mathbf{B}}_{\psi,\infty,n}}
\newcommand{\Bppsi}{\widetilde{\mathbf{B}}_{\psi,\infty}}
\newcommand{\norm}{{|\!|}}

\newcommand{\wtA}{\widetilde{A}}
\newcommand{\wtJ}{\widetilde{J}}
\newcommand{\wtS}{\widetilde{S}}
\newcommand{\wtR}{\widetilde{R}}
\newcommand{\wtX}{\widetilde{X}}
\newcommand{\wtY}{\widetilde{Y}}
\newcommand{\wtZ}{\widetilde{Z}}

\newcommand{\wtr}{\widetilde{\calR}}
\newcommand{\wtx}{\widetilde{\frakX}}
\newcommand{\wty}{\widetilde{\frakY}}
\newcommand{\wtz}{\widetilde{\frakZ}}

\newcommand{\Hsmall}{\text{H-sm}}

\newcommand{\tHsmall}{\text{t-H-sm}}

\title[]{The small $p$-adic Simpson correspondence in the semi-stable reduction case}

\author[]{Mao Sheng}
\address{M.S., Yau Mathematical Sciences Center, Tsinghua University, Beijing, 100084 China \&  Yanqi Lake Beijing Institute of Mathematical Sciences and Applications, Beijing, 101408, China}
\email{msheng@mail.tsinghua.edu.cn}

\author[]{Yupeng Wang}
\address{Y.W., Beijing International Center for Mathematical Research, Peking University, YiHeYuan Road 5, Beijing, 100190, China.}
\email{2306393435@pku.edu.cn}



\begin{abstract}
  We generalize several known results on small Simpson correspondence for smooth formal schemes over $\calO_C$ to the case for semi-stable formal schemes. More precisely, for a liftable semi-stable formal scheme $\frakX$ over $\calO_C$ with generic fiber $X$, we establish (1) an equivalence between the category of Hitchin-small integral $v$-bundles on $X_{v}$ and the category of Hitchin-small Higgs bundles on $\frakX_{\et}$, generalizing the previous work of Min--Wang, and (2) an equivalence between the moduli stack of $v$-bundles on $X_{v}$ and the moduli stack of rational Higgs bundles on $\frakX_{\et}$ (equivalently, moduli stack of Higgs bundles on $X_{\et}$), generalizing the previous work of Ansch\"utz--Heuer--Le Bras.
\end{abstract}

%\date{\today}

\maketitle
\setcounter{tocdepth}{1}

%\noindent{\textbf{Keywords:} Integral $p$-adic Simpson correspondence, Deligne--Illusie decomposition.}

%\noindent{\textbf{MSC2020:} Primary 14G22. Secondary 14F30, 14G45.}

\tableofcontents

%%%%%%第一章：简介
\section{Introduction}\label{sec:introduction}
\subsection{Overview}
  The small $p$-adic Simpson correspondence was firstly considered by Faltings \cite{Fal} and then systematically studied by Abbes--Gros--Tsuji \cite{AGT} and Tsuji \cite{Tsu}, which is known as the equivalence between categories of \emph{(Faltings-)small} generalized representations and \emph{(Faltings)-small} Higgs bundles on $X$, the generic fiber of a liftable (log-)scheme $\frakX$ over $\calO_C$ with nice singularities. (For example, $\frakX$ could have semi-stable special fiber.) Using this, for curves, Faltings gave an equivalence between the whole category of generalized representations and the whole category of Higgs bundles.

  According to the recent development of $p$-adic geometry, we have several improvements of the previous works on $p$-adic Simpson correspondence. In \cite{LZ}, for a smooth rigid variety $X$ over $K$, a discretely valued complete field over $\Qp$ with perfect residue field, Liu and Zhu assigned to each $\Qp$-local system on $X_{\et}$ to a $G_K$-equivairant Higgs bundle on $X_{\widehat{\overline K},\et}$, based on the previous work of Scholze \cite{Sch-Pi}. Using the decompletion theory in \cite{DLLZ}, Min and the second author generalized this assignment to an equivalence between the category of $v$-bundles on $X_{v}$ and the category of $G_K$-equivariant Higgs bundles on $X_{\widehat{\overline K},\et}$, as the \emph{arithmetic $p$-adic Simpson correspondence} \cite{MW-JEMS}. More generally, for any \emph{proper} smooth rigid variety $X$ over $C$, a complete algebraic closed field, Heuer established an equivalence between the whole category of $v$-bundles on $X_v$ and the whole category of Higgs bundles $X_{\et}$ \cite{Heu-Simpson}, generalising the previous work of Faltings \cite{Fal} in the curve case. Heuer also introduced the moduli stack of $v$-bundles on $X_v$ and the moduli stack of Higgs bundles on $X_{\et}$, and proved these two stacks are isomorphic after taking \'etale sheafifications \cite{Heu-Moduli}. Very recently, Heuer and Xu proved when $X$ is a smooth curve, these two stacks are equivalent to each other up to a choice of extra data \cite{HX}.
  
  On the other hand, the theory of $p$-adic Simpson correspondence is closely related to the prismatic theory introduced by Bhatt--Scholze \cite{BS22} and developed by Bhatt--Lurie \cite{BL22a,BL22b}. For a smooth formal scheme $\frakX$ over $\calO_C$, in the case $\frakX$ is affine and \'etale over a formal torus, Tian established an equivalence between the category of \emph{Hodge--Tate crystals} on the prismatic site associated to $\frakX$ and the category of topologically nilpotent Higgs bundles on $\frakX$ \cite{Tian}. In the same setting, Morrow and Tsuji \cite{MT} also obtained a ``$q$-deformation'' of Tian's result. When $\frakX$ is smooth over $\calO_K$, Min and the second author established an equivalence between the category of rational Hodge--Tate crystals on the absolute prismatic site associated to $\frakX$ and the category of \emph{enhanced Higgs bundles} on $X_{\et}$ \cite{MW-JEMS}. This work was generalized by Ansch\"utz--Heuer--Le Bras to the derived case and they also gave a pointwise criterion for a Higgs bundle being enhanced \cite{AHLB23,AHLB-small}. When $\frakX$ is smooth and liftable over $\calO_C$ with generic fiber $X$, Ansch\"utz--Heuer--Le Bras also gave an equivalence between the moduli stack of \emph{Hitchin-small} $v$-bundles on $X_v$ and the moduli stack of \emph{Hitchin-small} Higgs bundles $X_{\et}$. As (Faltings-)small objects are always Hitchin-small, their result generalized the previous works of Faltings \cite{Fal}, Abbes--Gros--Tsuji \cite{AGT} and the second author \cite{Wan23}. There is another generalisation of Faltings' Simpson correspondence: In \cite{MW-AIM}, Min and the second author obtain an equivalence between the category of (Faltings-)small \emph{integral} $v$-bundles $X_v$ and the category of (Faltings-)small \emph{integral} Higgs bundles on $\frakX_{\et}$, also generalizing partial results in \cite{BMS18}. 
  
\subsection{Main results}
  From now on, we freely use the notations in \S\ref{ssec:notation}. In particular, we always let $K$ be a complete discrete valuation field with the ring of integers $\calO_K$ and the perfect residue field $\kappa$, and let $C$ be the completion of a fixed algebraic closure of $K$ with the ring of integers $\calO_C$ and the maximal ideal $\frakm_C$. Let $\bfA_{\inf,K}:=\Ainf(\calO_C)\otimes_{\rW(\kappa)}\calO_K$ be the ramified infinitesimal period ring of Fontaine (with respect to $K$), $\xi_K$ be a generator of the natural surjection $\theta_K:\bfA_{\inf,K}\to \calO_C$, and $\bfA_{2,K}:=\bfA_{\inf,K}/\xi_K^2$. Equip $\bfA_{2,K}$ and $\calO_C$ with the canonical log structures; that is, the log-structures induced from $\calO_C^{\flat}\setminus\{0\}\xrightarrow{[\cdot]}\bfA_{\inf,K}$ via the corresponding quotients. In this paper, we always work with semi-stable $p$-adic formal schemes $\frakX$ (viewed as log-schemes) with the generic fiber $X$ in the sense of \cite{CK19}. 
  \subsubsection{An integral $p$-adic Simpson correspondence}
  Our first result is the following integral $p$-adic Simpson correspondence for liftable semi-stable formal schemes.
  \begin{thm}[Theorem \ref{thm:integral Simpson}]\label{thm:Intro-integral Simpson}
      Let $\frakX$ be a semi-stable formal scheme over $\calO_C$ of relative dimension $d$. 
      Suppose that it admits a flat lifting (as a formal log-scheme) $\widetilde \frakX$ over $\bfA_{2,K}$. Then there exists a period sheaf $(\calO\widehat \bC_{\pd}^+,\Theta)$ together with Higgs field $\Theta$ (depending on the given lifting $\widetilde \frakX$) inducing a rank-preserving equivalence
      \begin{equation}\label{intro-equ:integral Simpson}
          \rL\rS^{\Hsmall}(\frakX,\OXp)\simeq \HIG^{\tHsmall}(\frakX,\calO_{\frakX})
      \end{equation}
      between the category $\rL\rS^{\Hsmall}(\frakX,\OXp)$ of Hitchin-small integral $v$-bundles on $X_v$ (cf. Definition \ref{dfn:small O^+-representations}) and the category of twisted Hitchin-small Higgs bundles on $\frakX_{\et}$ (Definition \ref{dfn:small Higgs bundles}), which preserves tensor products and dualities. More precisely, let $\nu:X_v\to \frakX_{\et}$ be the natural morphism of sites, and then the following assertions are true:
      \begin{enumerate}
          \item[(1)] For any $\calM^+\in \rL\rS^{\Hsmall}(\frakX,\widehat \calO^+_{X})$ of rank $r$, the complex $\rR\nu_*(\calM^+\otimes\calO\widehat \bC_{\pd}^+)$ is concentrated in degree $[0,d]$ such that 
          \[\rL\eta_{\rho_K(\zeta_p-1)}\rR\nu_*(\calM^+\otimes\calO\widehat \bC_{\pd}^+)\simeq \left(\nu_*(\calM^+\otimes\calO\widehat \bC_{\pd}^+)\right)[0],\]
          where $\rL\eta_{\rho_K(\zeta_p-1)}$ denotes the d\'ecalage functor in \cite[\S6]{BMS18}.
          Moreover, the push-forward
          \[(\calH^+(\calM^+),\theta):=(\nu_*(\calM^+\otimes\calO\widehat \bC_{\pd}^+),\nu_*(\id_{\calM^+}\otimes\Theta))\]
          defines a twisted Hitchin-small Higgs bundle of rank $r$ on $\frakX_{\et}$.
        
          \item[(2)] For any $(\calH^+,\theta)\in \HIG^{\tHsmall}(\frakX,\calO_{\frakX})$ of rank $r$, the
          \[\calM^+(\calH^+,\theta):=(\calH^+\otimes\calO\widehat \bC_{\pd}^+)^{\theta\otimes\id+\id\otimes\Theta = 0}\]
          defines a Hitchin-small integral $v$-bundle of rank $r$ on $X_{v}$.

          \item[(3)] The equivalence \eqref{intro-equ:integral Simpson} is induced by the functors in Items (1) and (2).
      \end{enumerate}
  \end{thm}
  \begin{rmk}
      Theorem \ref{thm:Intro-integral Simpson} generalizes the main result \cite[Th. 1.1]{MW-AIM} in the following sense: In \emph{loc.cit.}, we always work with smooth formal schemes and (Faltings-)small integral $v$-bundles and Higgs bundles, but here, the result holds for semi-stable formal schemes and for Hitchin-small integral $v$-bundles and Higgs bundles. We remark that a Faltings-small integral $v$-bundle (resp. Higgs bundle) is always Hitchin-small but the converse is not true (for example, a nilpotent Higgs bundle is always Hitchin-small but not always Faltings-small). It also upgrades the rational Simpson correspondences of Faltings \cite{Fal}, Abbes--Gros--Tsuji \cite{AGT} and Wang \cite{Wan23} for Faltings-small objects and Ansch\"utz--Heuer--Le Bras \cite[Th. 1.2]{AHLB-small} for Hitchin-small objects to the integral level (see Theorem \ref{thm:Intro-rational Simpson} for its analogue in semi-stable reduction case).
  \end{rmk}
  Given a Hitchin-small integral $v$-bundle $\calM^+$ with the associated Hitchin-small Higgs bundle $(\calH^+,\theta)$ in the sense of Theorem \ref{thm:Intro-integral Simpson}, we also want to compare the push-forward $\rR\nu_*\calM^+$ with the Higgs complex $\rD\rR(\calH^+,\theta)$. Indeed, we are able to prove the following truncated cohomological comparison, generalizing \cite[Cor. 1.2 and Th. 1.4]{MW-AIM} to the semi-stable reduction case.
  \begin{thm}[Corollary \ref{cor:bounded torsion} and Theorem \ref{thm:truncated comparison}]\label{thm:Intro-cohomological comparison}
      Keep assumptions in Theorem \ref{thm:Intro-integral Simpson} and let $\calM^+$ be a Hitchin-small integral $v$-bundle with the corresponding Hitchin-small Higgs bundle $(\calH^+,\theta)$. Then there exists a canonical morphism
      \[\rD\rR(\calH^+,\theta)\to\rR\nu_*\calM^+\]
      whose cofiber is killed by $(\rho_K(\zeta_p-1))^{\max(d+1,2(d-1))}$, and this morphism induces a quasi-isomorphism
      \[\tau^{\leq 1}\rD\rR(\calH^+,\theta)\simeq \tau^{\leq 1}\rL\eta_{\rho_K(\zeta_p-1)}\rR\nu_*\calM^+.\]
      In particular, when $\frakX$ is a curve (i.e. $d=1$), we get a canonical quasi-isomorphism
      \[\rD\rR(\calH^+,\theta)\simeq \rL\eta_{\rho_K(\zeta_p-1)}\rR\nu_*\calM^+.\]
  \end{thm}
  In particular, by letting $\calM^+ = \OXp$ (or equivalently $(\calH^+,\theta)=(\calO_{\frakX},0)$), by a standard trick used in the proof of \cite[Th. 2.1]{DI} and \cite[Th. 4.1]{Min21}, we conclude the following analogue of Delighe--Illusie decomposition for semi-stable formal schemes:
  \begin{thm}[Theorem \ref{thm:DI}]\label{thm:Intro-DI}
      Keep assumptions in Theorem \ref{thm:Intro-integral Simpson}. Then there exists a quasi-isomorphism
      \[\oplus_{i=0}^{p-1}\Omega^{i,\log}_{\frakX}\{-i\}[-i]\to \tau^{\leq p-1}\rL\eta_{\rho_K(\zeta_p-1)}\rR\nu_*\OXp.\]
  \end{thm}
  \begin{rmk}
      To obtain the decomposition in Theorem \ref{thm:Intro-DI}, it seems that we must assume $\frakX$ admits a lifting as formal \emph{log-scheme} over $\bfA_{2,K}$ (endowed with the \emph{canonical log structure}). This phenomenon appears in the classical theory of Deligne--Illusie decomposition for log-schemes in positive characteristic, by a previous work of the first author \cite{SSt}. See Remark \ref{rmk:Dependence on log-lifting} for more discussion.
  \end{rmk}
\subsubsection{A stacky $p$-adic Simpson correspondence}
  Let $\frakX$ be a semi-stable $p$-adic formal scheme over $\calO_C$ with the generic fiber $X$ as before. Again we assume that $\frakX$ admits a lifting $\widetilde \frakX$ over $\bfA_{2,K}$. Denote by $\Perfd$ the $v$-site of affinoid perfectoid spaces over $C$ in the sense of \cite{Sch-Diamond}. For any $S = \Spa(A,A^+)$, let $\frakX_S$ and $X_S$ be the base-change of $\frakX$ abd $X$ to $A^+$ and $A$, respectively. Then for any $r\geq 0$, the following two functors
  \[\rL\rS_r(X,\OX): S\in \Perfd\mapsto \{\text{groupoid of $v$-bundles of rank $r$ on $X_{S,v}$}\}\]
  and 
  \[\HIG_r(X,\calO_X): S\in\Perfd\mapsto\{\text{groupoid Higgs bundles of rank $r$ on $X_{S,\et}$}\}\]
  are actually small $v$-stacks \cite[Th. 1.4]{Heu-Moduli}. Denote by $\calA_r$ the following $v$-sheaf
  \[\calA_r: S\in \Perfd\mapsto \oplus_{i=1}^r\rH^0(X_S,\Sym^i(\Omega^1_{X_S}\{-1\})).\]
  Then we have the Hitchin-fibrations
  \[\xymatrix@C=0.5cm{
     \rL\rS_r(X,\OX)\ar[rd]^{\widetilde h}&&\ar[ld]_{h}\HIG(X,\calO_X)\\ 
     &\calA_r,
  }\]
  where $h$ carries each Higgs bundle $(\calH,\theta)$ to the characteristic polynomial of $\theta$ while $\widetilde h$ denotes the composite of $h$ with the ``\rm{HTlog}'' in \cite[Def. 1.6]{Heu-Moduli}.
  The integral model $\frakX$ of $X$ induces a sub-sheaf 
  \[\calA_r^{\Hsmall}:S\in \Perfd\mapsto \oplus_{i=1}^rp^{<\frac{i}{p-1}}\rH^0(\frakX_S,\Sym^i(\Omega^{1,\log}_{\frakX_S}\{-1\}))\]
  of $\calA_r$, where $p^{<\frac{i}{p-1}}$ denotes the ideal $(\zeta_p-1)^i\frakm_C\subset \calO_C$. For any $Z\in\{\rL\rS_r(X,\OX),\HIG_r(X,\calO_X)\}$, denote by $Z^{\Hsmall}:=Z\times_{\calA_r}\calA_r^{\Hsmall}$ its \emph{Hitchin-small locus}. A Higgs bundle on $X_{S,\et}$ (resp. a $v$-bundle on $X_{S,v}$) of rank $r$ is called Hitchin-small if as a point, it belongs to the Hitchin-small locus of the corresponding moduli stack. Then our second main result is the following equivalence of stacks:
  \begin{thm}[Theorem \ref{thm:stacky Simpson}]\label{thm:Intro-stack Simpson}
      Let $\frakX$ be a semi-stable $p$-adic formal scheme over $\calO_C$ with the generic fiber $X$ which admits a lifting $\widetilde \frakX$ over $\bfA_{2,K}$. Then the lifting $\widetilde \frakX$ induces an equivalence of stacks
      \[\rho_{\widetilde \frakX}:\rL\rS_r(X,\OX)^{\Hsmall}\xrightarrow{\simeq}\HIG_r(X,\calO_X)^{\Hsmall}.\]
  \end{thm}
  \begin{rmk}
      When $\frakX$ is smooth over $\calO_C$, the above equivalence was obtained by Ansch\"utz--Heuer--Le Bras \cite[Th. 1.1]{AHLB-small}, based on their previous work \cite{AHLB23} on studying rational Hodge--Tate crystals on the prismatic site associated to $\frakX$. They first constructed a fully faithful functor $S_{\widetilde \frakX}:\HIG_r(X,\calO_X)\to\rL\rS_r(X,\OX)$ via prismatic theory and then showed the essential surjectivity by working locally on $\frakX$. Compared with their construction, we do not need any input of prismatic theory and can give an explicit description of $\rho_{\widetilde \frakX}^{-1}$ (cf. Theorem \ref{thm:Intro-rational Simpson}). 
  \end{rmk}
  \begin{rmk}
      It is still a question if there exists an equivalence of the whole stacks
      \[\rL\rS_r(X,\OX)\xrightarrow{\simeq}\HIG_r(X,\calO_X)\]
      for general smooth $X$ over $C$. Up to now, we only know a few on this question: We only have the desired equivalence when $X$ is either a curve \cite{HX} or the projective space $\bP^n$ \cite[Cor. 1.3]{AHLB-small}. It seems the recent announcement of Bhargav Bhatt and Mingjia Zhang on Simpson gerbe may help to solve this question, but up to now, we still do not know if such an equivalence always exists.
  \end{rmk}

  We now describe the strategy to prove Theorem \ref{thm:Intro-stack Simpson}. It suffices to show for any $S = \Spa(A,A^+)\in \Perfd$, there is a rank-preserving equivalence 
  \[\rL\rS(X,\OX)^{\Hsmall}(S)\xrightarrow{\simeq}\HIG(X,\calO_X)^{\Hsmall}(S)\]
  between the category of Hitchin-small $v$-bundles on $X_{v,\et}$ and the category of Hitchin-small Higgs bundles on $X_{S,\et}$ which is functorial in $S$. To do so, we may follow the same argument for the proof of Theorem \ref{thm:Intro-integral Simpson}: Given a lifting $\widetilde \frakX$ of $\frakX$ over $\bfA_{2,K}$, its base-change $\widetilde \frakX_S$ along $\bfA_{2,K}\to \AAKK(S) = \AAinf(S)\otimes_{\Ainf}\bfA_{2,K}$ is a lifting of $\frakX_S$ (the base-change of $\frakX$ along $\calO_C\to A^+$). Using this, one can still construct a period sheaf $(\calO\widehat \bC_{\pd,S},\Theta)$ with Higgs field whose Higgs complex induces a resolution of the structure sheaf $\widehat \calO_{X_S}$ on $X_{S,v}$. Then one can prove the following rational Simpson correspondence:
  \begin{thm}[Theorem \ref{thm:rational Simpson} and Lemma \ref{lem:Hitchin-small vs twisted Hitchin-small}]\label{thm:Intro-rational Simpson}
      Let $\nu: X_{S,v}\to X_{S,\et}$ be the natural morphism of sites.
      \begin{enumerate}
          \item[(1)] For any $\calM\in \rL\rS(X,\OX)^{\Hsmall}(S)$ of rank $r$, we have a quasi-isomorphism 
          \[\rR\nu_*(\calM\otimes\calO\widehat \bC_{\pd,S}^+)\simeq\nu_*(\calM\otimes\calO\widehat \bC_{\pd,S}^+)\]
          Moreover, the push-forward
          \[(\calH(\calM),\theta):=(\nu_*(\calM\otimes\calO\widehat \bC_{\pd,S}),(\zeta_p-1)\nu_*(\Theta_{\id_{\calM}\otimes\Theta}))\]
          defines a Hitchin-small Higgs bundle of rank $r$ on $X_{S,\et}$.
        
          \item[(2)] For any $(\calH,\theta)\in \HIG(X,\calO_X)^{\Hsmall}(S)$ of rank $r$, the
          \[\calM(\calH,\theta):=(\calH\otimes\calO\widehat \bC_{\pd,S})^{(\zeta_p-1)^{-1}\theta\otimes\id+\id_{\calH}\otimes\Theta = 0}\]
          defines a Hitchin-small $v$-bundle of rank $r$ on $X_{S,v}$.
          
          \item[(3)] The functors in Items (1) and (2) defines an equivalence of categories
          \[\rL\rS(X,\OX)^{\Hsmall}(S)\xrightarrow{\simeq}\HIG(X,\calO_X)^{\Hsmall}(S).\]
          which preserves ranks, tensor products and dualities. Moreover, for any Hitchin-small $v$-bundle $\calM$ with associated Hitchin-small Higgs bundle $(\calH,\theta)$, there exists a quasi-isomorphism
          \[\rR\nu_*\calM\simeq \rD\rR(\calH,\theta).\]
      \end{enumerate}
  \end{thm}
  Then Theorem \ref{thm:Intro-stack Simpson} follows immediately as the equivalence above is ovbiously functorial in $S$ by the construction.

  To obtain the desired period sheaf with Higgs field $(\calO\widehat \bC_{\pd,S},\Theta)$, we need to construct the \emph{integral Faltings' extension} corresponding to the lifting $\widetilde \frakX_S$ of $\frakX_S$ using the theory of log-cotangent complex of Olsson (and Gabber) \cite[\S8]{Ols} as we did in \cite[\S2]{Wan23}. To do so, we need to endow $A^+$ with a suitable log-structure, called the \emph{canonical log-structure}, such that the analogue of \cite[Lem. 3.14]{BMS18} holds true. The most natural log-structure on $A^+$ is $(A^{\times}\cap A^+\to A^+)$, and the difficult part is prove it is the correct one (cf. \S\ref{ssec:log structure on perfectoid Tate algebra}).

\subsection{Organization}
  The paper is organized as follows: In \S\ref{sec:set-up}, we introduce the canonical log-structure on perfectoid Tate algebra, and the basic set-up on semi-stable formal schemes we will work with. In \S\ref{sec:period sheaves}, we construct our period sheaf with Higgs field and prove the corresponding Poincar\'e's Lemma. In \S\ref{sec:local Simpson}, we include the key local calculations and give a local version of Simpson correspondence. Finally, in \S\ref{sec:global Simpson}, we prove the integral Simpson correspondence at first and then give the desired equivalence of moduli stacks on Hitchin-small $v$-bundles and Hitchin-small Higgs bundles for lifable semi-stable formal schemes $\frakX$.
\subsection{Notations}\label{ssec:notation}
  Throughout this paper, let $K$ be a complete discrete valuation field over $\Qp$ with the ring of integers $\calO_K$ and the residue field $\kappa$, which is required to be perfect. Put $\rW:=\rW(\kappa)$, 
  let $C = \widehat{\overline K}$ be the completion of a fixed algebraic closure $\overline K$ of $K$ with the ring of integers $\calO_C$ and the maximal ideal $\frakm_C$. Let $\Ainf$ and $\BdRp$ be the corresponding infinitesimal and de Rham period rings. Fix an embedding $p^{\bQ}\subset C^{\times}$, which induces an embedding $\varpi^{\bQ}\subset C^{\flat\times}$, where $\varpi = (p,p^{1/p},p^{1/p^2},\dots)\in C^{\flat}$. Fix a coherent system $\{\zeta_{p^n}\}_{n\geq 0}$ of primitive $p^n$-th roots of unity in $C$, and let $\epsilon:=(1,\zeta_p,\zeta_p^2,\dots)\in C^{\flat}$. Put $\bfA_{\inf,K}:=\Ainf\otimes_{\rW}\calO_K$ and then we have the canonical surjection $\theta_K:\bfA_{\inf,K}\to \calO_C$ whose kernel $\rI_K$ is principally generated and we fix an its generator $\xi_K$. Define $\bfA_{2,K}:=\bfA_{\inf,K}/\rI_K^2$.
  
  For any (sheaf of) $\bfA_{\inf,K}$-module $M$ and any $n\in \bZ$, denote by $M\{n\}$ its Breuil--Kisin--Fargues twist 
  \[M\{n\}:=M\otimes_{\Ainf}\rI_K^{\otimes n},\]
  which can be trivialized by $\xi_K^n$; that is, we have the identification $M\{n\} = M\cdot\xi_K^n$.
  Using this, we may regard $M$ as a sub-$\bfA_{\inf,K}$-module of $M\{-1\}$ via the identification $M = \xi_K M\{-1\}$.

  Let $\mu_{p^{\infty}}$ be the sub-group of $\calO_C^{\times}$ generated by $\{\zeta_{p^n}\}_{n\geq 1}$ and $\Zp(1):=\rT_p(\mu_{p^{\infty}})$ be its Tate module. For any (sheaf of) $\Zp$-module $M$ and any $n\in \bZ$, denote by $M(n)$ its Tate twist
  \[M(n):= M\otimes_{\Zp}\Zp(1)^{\otimes n}.\]
  Let $t = \log([\epsilon])\in \BdRp$ be the Fontaine's $p$-adic analogue of ``$2\pi i$''. Then $M(n)$ can be trivialized by $t^n$; that is, we have the identification $M(n) = M\cdot t^n$.

  The natural inclusion $\bfA_{\inf,K}\hookrightarrow\BdRp$ induces a natural inclusion
  \[\calO_C\{1\}\cong \xi_K\bfA_{\inf,K}/\xi_K^2\bfA_{\inf,K}\hookrightarrow t\BdRp/t^2\BdRp\cong C(1)\]
  identifying $\calO_C\{1\}$ with an $\calO_C$-lattice of $C(1)$.
  Let $\calO_C(1) = \calO_C\otimes_{\Zp}\Zp(1)\subset C\otimes_{\Zp}\Zp(1) = C(1)$ be the standard $\calO_C$-lattice of $C(1)$. Then there exists an element $\rho_K\in \calO_C$ with the $p$-adic valuation $\nu_p(\rho_K) = \nu_p(\calD_K)+\frac{1}{p-1}$ such that 
  \[\calO_C(1) = \rho_K\calO_C\{1\},\]
  where $\calD_K$ denotes the ideal of differentials of $\calO_K$. For example, when $\calO_K=\rW$, we have $\calD_K=\calO_K$ and can choose $\rho_K = \zeta_p-1$.

  Fix a ring $R$. If an element $x\in R$ admits arbitrary pd-powers, we denote by $x^{[n]}$ its $n$-th pd-power (i.e. analogue of $\frac{x^n}{n!}$) in $R$. Put $E_i = (0,\dots,1,\dots,0)\in \bN^d$ with $1$ appearing at exactly the $i$-th component. For any $J=(j_1,\dots,j_d)\in\bN^d$ and any $x_1,\dots,x_d\in R$, we put 
  \[\underline x^J:=x_1^{j_1}\cdots x_d^{j_d}\]
  and if moreover $x_i$ admits arbitrary pd-powers in $A$ for all $i$, we put
  \[\underline x^{[J]}:=x_1^{[j_1]}\cdots x_d^{[j_d]}.\]
  Define $|J|:=j_1+\cdots+j_d$.
  For any $\alpha\in\bN[1/p]\cap[0,1)$, we put
  \[\zeta^{\alpha} = \zeta_{p^n}^m\]
  if $\alpha = \frac{m}{p^n}$ such that $p$ and $m$ are coprime in $\bN$. If $x\in A$ admits compatible $p^n$-th roots $x^{\frac{1}{p^n}}$, we put
  \[x^{\alpha} = x^{\frac{m}{p^n}}\]
  for $\alpha = \frac{m}{p^n}$ as above. In general, for any $\underline \alpha:=(\alpha_1,\dots,\alpha_d)\in(\bN[1/p]\cap[0,1))^d$ and any $x_1,\dots,x_d$ admitting compatible $p^n$-th roots, we put
  \[\underline x^{\underline \alpha}:=x_1^{\alpha_1}\cdots x_d^{\alpha_d}.\]

  We always denote by $\Perfd$ the $v$-site of affinoid perfectoid spaces over $\Spa(C,\calO_C)$ in the sense of \cite{Sch-Diamond}.

\section*{Acknowledgement}
   The second author would like to thank Ruochuan Liu for his interest, Yu Min for valuable conversations, and Gabriel Dorfsman--Hopkins for the correspondence upon the proof of Proposition \ref{prop:log structure on perfectoids}. The project started at June 2020, when the senond was invited by the first author to visit the university of science and technique of china (USTC). The second author would like to thank the first author for the invitation and the USTC for the great research environment. The authors are partially supported by CAS Project for Young Scientists in Basic Research, Grant No. YSBR-032.

\section{Basic set-up}\label{sec:set-up}

\subsection{Canonical log structure on perfectoid affinoid algebras}\label{ssec:log structure on perfectoid Tate algebra}

    Let $(A,A^+)$ be a perfectoid affinoid algebra with tilting $(A^{\flat},A^{\flat+})$. Fix a $\underline{\pi}\in A^{\flat,+}$ such that $\pi = \underline{\pi}^{\sharp}$ and $\pi^p = pu$ for some unit $u\in A^{+,\times}$ (cf. \cite[Lem. 3.19]{BMS18}). Then $A^+$ is $\pi$-adically complete and $A = A^+[\frac{1}{\pi}]$ while the same holds for $(A^{\flat}, A^{\flat,+},\underline{\pi})$ (instead of $(A,A^+,\pi)$). Let 
    \[\sharp: A^{\flat}\to A,\quad \underline f = (f_0,f_1,\dots)\in A^{\flat} = \varprojlim_{x\mapsto x^p}A\mapsto f_0\] 
    denote the usual sharp map, which is a morphism of multiplicative monoids. Then $\sharp$ coincides with the composite
    \[A^{\flat,+}\xrightarrow{[\cdot]}\AAinfK(A,A^+)\xrightarrow{\theta_K}A^+.\]

    \begin{dfn}\label{dfn:canonical log structure on Ainf}
        For any affinoid perfectoid $U = \Spa(A,A^+)$ over $\Qp$, the \emph{canonical log-structure} on $\AAinfK(U) = \AAinfK(A,A^+)$ is the log-structure induced by the pre-log structure 
        \[A^{\flat,\times}\cap A^{\flat,+}\xrightarrow{[\cdot]}\AAinfK(U).\]
        For any $\AAinfK(U)$-algebra $B^+\in\{A^{\flat,+},A^+,\AAinfK(U)/\Ker(\theta_K)^n\}$, the \emph{canonical log-structure} on $B^+$ is the log-structure induced from the canonical log-structure on $\AAinfK(U)$ via the natural surjection $\AAinfK(U)\to B^+$. In particular, the canonical log-structure on $A^+$ is the log-structure associated to the pre-log structure $A^{\flat,\times}\cap A^{\flat,+}\xrightarrow{\sharp}A^+$.
    \end{dfn}

    Note that $A^+$ admits another log-structure $(A^{\times}\cap A^+\hookrightarrow A^+)$. Our purpose in this section is to prove the following result:
    \begin{prop}\label{prop:log structure on perfectoids}
        For any perfectoid affinoid algebra $(A,A^+)$, the canonical log-structure on $A^+$ coincides with the log-structure $(A^{\times}\cap A^+\hookrightarrow A^+)$.
    \end{prop}
    For example, as $C^{\flat,\times}\cap\calO_C^{\flat} = \varpi^{\bQ_{>0}}\cdot \calO_C^{\flat,\times}$ and $C^{\times}\cap\calO_C = p^{\bQ_{>0}}\cdot \calO_C^{\times}$, we know the canonical log-structure on $\calO_C$ is associated to the pre-log-structure $\bQ_{>0}\xrightarrow{a\mapsto \varpi^a}\calO_C^{\flat}\xrightarrow{\sharp}\calO_C$ while the log-structure $C^{\times}\cap\calO_C\hookrightarrow\calO_C$ is associated to the pre-log-structure $\bQ_{>0}\xrightarrow{a\mapsto p^a}\calO_C$, yielding Proposition \ref{prop:log structure on perfectoids} in this case because $(\varpi^a)^{\sharp} = p^a$ for any $a\in\bQ$.

    The key ingredient for proving Proposition \ref{prop:log structure on perfectoids} is the following result, whose proof is similar to that of \cite[Prop. 2.9]{DH}.

    \begin{prop}\label{prop:unit group for perfectoids}
        Let $(A,A^+)$ be a perfectoid affinoid algebra over $\Qp$ with tilting $(A^{\flat},A^{\flat,+})$. Then for any $f\in A^{\times}\cap A^+$, there exists a $\underline g\in A^{\flat,\times}\cap A^{\flat+}$ and an $h\in A^+$ such that $f = \underline g^{\sharp}\cdot(1+p h)$.
    \end{prop}
    \begin{proof}
        Put $U = \Spa(A,A^+)$ and $U^{\flat} = \Spa(A^{\flat},A^{\flat+})$. By tilting equivalence, we can an homeomorphism of underlying topological spaces of $|U|\cong|U^{\flat}|$. For any $x\in|U|$, we denote by $x^{\flat}\in |U^{\flat}|$ the image of $x$ via this identification.

        Fix an $f\in A^{\times}\cap A^+$. By \cite[Lem. 3.3(i)]{Hub93} and \cite[Lem. 1.4]{Hub94}, for any $x\in |U|$, we always have $0<|f(x)|\leq 1$. As $\pi$ is a pesudo-uniformizer and $|U|$ is quasi-compact, there exists an integer $N\geq 0$ such that for any $x\in |U|$,
        \begin{equation}\label{equ:UnitGroup-I}
            |\underline{\pi}(x^{\flat})|^{pN}=|\pi(x)|^{pN}\leq |f(x)|\leq 0.
        \end{equation}
        By the approximation lemma \cite[Lem. 2.3.1]{CS24}, there exists a $\underline g\in A^{\flat,+}$ such that for any $x\in|U|$, 
        \begin{equation}\label{equ:UnitGroup-II}
            |f(x)-\underline g^{\sharp}(x)|\leq |p(x)|\cdot\max(|\underline g(x^{\flat})|,|\underline{\pi}(x^{\flat})|^{pN}).
        \end{equation}
        As $|\underline g^{\sharp}(x)| = |\underline g(x^{\flat})|$ and $|p(x)|<1$, by strong triangular inequality, we deduce from (\ref{equ:UnitGroup-I}) and (\ref{equ:UnitGroup-II}) that for any $x\in |U|$, 
        \[|f(x)| = |\underline g^{\sharp}(x)| = |\underline g(x^{\flat})|.\] 
        In particular, using (\ref{equ:UnitGroup-I}) again, for any $x^{\flat}\in|U^{\flat}|$, we have
        \[0<|\underline{\pi}(x^{\flat})|^{pN}\leq |\underline g(x^{\flat})|\leq 1.\] 
        This forces that $\underline g\in A^{\flat\times}\cap A^{\flat,+}$ and that $\underline g^{-1}\underline{\pi}^{pN}\in A^{\flat,+}$, by \cite[Lem. 3.3(i)]{Hub93} and \cite[Lem. 1.4]{Hub94} again. A similar argument also shows that $(\underline g^{\sharp})^{-1}f\in A^{+,\times}$ is a unit in $A^+$. So we can rewrite (\ref{equ:UnitGroup-II}) as 
        \[|((\underline g^{\sharp})^{-1}f)(x)-1|\leq |p(x)|\cdot\max(1,|(\underline g^{-1}\underline{\pi}^{pN})(x^{\flat})|) = |p(x)|.\]
        As $p$ is invertible in $A$, the above argument implies that $h:=p^{-1}((\underline g^{\sharp})^{-1}f-1)$ is a well-defined element in $A^+$. By construction, we conclude that $f = \underline g^{\sharp}(1+ph)$ as desired.
    \end{proof}
    \begin{cor}\label{cor:unit group for perfectoids}
        Let $(A,A^+)$ be a perfectoid affinoid algebra over $\Qp$ with tilting $(A^{\flat},A^{\flat,+})$. Then for any $?\in\{+,\emptyset\}$, we have 
        \[A^{\times}\cap A^{?} = (A^{\flat,\times}\cap A^{\flat,?})^{\sharp}\cdot(1+pA^+).\]
    \end{cor}
    \begin{proof}
        For $? = \emptyset$, we have to show $A^{\times} = (A^{\flat,\times})^{\sharp}\cdot(1+pA^+)$. By recalling that $A = A^+[\frac{1}{\pi}]$ and $A^{\flat} = A^{\flat,+}[\frac{1}{\underline \pi}]$ with $\underline{\pi}^{\sharp} = \pi$, we are reduced to the case for $? = +$. That is, we have to show 
        \[A^{\times}\cap A^{+} = (A^{\flat,\times}\cap A^{\flat,+})^{\sharp}\cdot(1+pA^+).\] But this follows from Proposition \ref{prop:unit group for perfectoids} immediately.
    \end{proof}
    \begin{proof}[\textbf{Proof of Proposition \ref{prop:log structure on perfectoids}:}]
        Let $G \subset A^{\flat,\times}$ be the kernel of the homomorphism $\sharp:A^{\flat,\times}\to A^{\times}$. Keep the notations in the proof of Proposition \ref{prop:unit group for perfectoids}. As for any $g\in G$ and any $x\in |U|$, we have 
        \[|g(x^{\flat})| = |g^{\sharp}(x)| = 1,\]
        by \cite[Lem. 3.3(i)]{Hub93} and \cite[Lem. 1.4]{Hub94}, $G$ is a sub-group of $A^{\flat,+,\times}$ and hence the kernel of the homomorphism $\sharp:A^{\flat,+,\times}\to A^{+,\times}$. It is also the kernel of the homomorphism of monoids
        \[\sharp:A^{\flat,\times}\cap A^{\flat,+}\to A^{\times}\cap A^+.\]
        So the canonical log-structure on $A^+$ is the log-structure associated to the pre-log structure
        \[(A^{\flat,\times}\cap A^{\flat,+})^{\sharp}\hookrightarrow A^+.\]
        So the result follows from Corollary \ref{cor:unit group for perfectoids} because $1+pA^+\subset A^{+,\times}$.
    \end{proof} 
    \begin{rmk}[Surjectivity of $\sharp$ for sympathetic algebras]\label{Rmk-Sympathetic}
        Let $\Lambda$ be a sympathetic algebra over $C$ in the sense of \cite[\S5]{Col02}; that is, it is a connected, $p$-closed and spectral Banach $C$-algebra $\Lambda$ (with respect to the spectral norm $|\cdot|_{\Lambda}$). Put $\calO_{\Lambda} = \{\lambda\in\Lambda\mid|\lambda|_{\Lambda}\leq 1\}$ and then by \cite[Lem. 2.15(iii)]{Col02}, $(\Lambda,\calO_{\Lambda})$ is a perfectoid affinoid algebra. The $p$-closeness of $\Lambda$ \cite[\S2.8]{Col02} together with Corollary \ref{cor:unit group for perfectoids} implies that $\Lambda^{\times} = (\Lambda^{\flat,\times})^{\sharp}$, that $\Lambda^{\times}\cap \calO_{\Lambda} = (\Lambda^{\flat,\times}\cap \calO_{\Lambda}^{\flat})^{\sharp}$ and that $\calO_{\Lambda}^{\times} = (\calO_{\Lambda}^{\flat,\times})^{\sharp}$. We point out that the sympathetic algebras usually form a basis for the pro-\'etale topology of a rigid space (cf. \cite[Prop. 4.8]{Sch-Pi} and the proof therein).
  \end{rmk}

  As a consequence, we have an analogue of \cite[Lem. 3.14]{BMS18} in the logarithmic setting.

  \begin{dfn}\label{dfn:perfect log-structure}
      Let $A$ be a perfect $\Fp$-algebra (resp. a perfectoid algebra over $\Zp$). A log-structure $M_A\to A$ on $A$ is called \emph{perfect}, if it is associated to a pre-log-structure $N\to A$ with $N$ a uniquely $p$-divisible monoid; that is, the map $N\xrightarrow{n\mapsto pn}N$ is bijective. 
  \end{dfn}
  The following lemma is well-known to experts:
  \begin{lem}\label{lem:cotangent complex for perfect log structure}
    \begin{enumerate}
        \item Let $(M_A\to A)\to (M_B\to B)$ be a morphism of perfect $\Fp$-algebras with perfect log-structures. Then the corresponding cotangent complex $\rL_{(M_B\to B)/(M_A\to A)} = 0$.

        \item Let $(M_A\to A)\to (M_B\to B)$ be a morphism of perfectoid algebras over $\Zp$ with perfect log-structures. Then the corresponding $p$-complete cotangent complex $\widehat \rL_{(M_B\to B)/(M_A\to A)} = 0$.
    \end{enumerate}
  \end{lem}
  \begin{proof}
      For Item (1): Considering the morphisms of log-rings 
      \[(A^{\times}\hookrightarrow A)\to(M_A\to A)\to (M_B\to B)\]
      and the associated exact triangle (cf. \cite[Th. 8.18]{Ols})
      \[\rL_{(M_A\to A)/(A^{\times}\to A)}\otimes_A^{\rL}B\to\rL_{(M_B\to B)/(A^{\times}\to A)}\to\rL_{(M_B\to B)/(M_A\to A)},\]
      we are reduced to the case for $(M_A\to A)=(A^{\times}\hookrightarrow A)$. Assume the log-structure $M_B\to B$ is associated to the pre-log-structure $N\to B$ with $N$ uniquely $p$-divisible. Then the morphism $(A^{\times}\hookrightarrow A)\to(M_B\to B)$ of log-structures is associated to the morphism $(0\xrightarrow{0\mapsto 1}A)\to (N\to B)$ of pre-log-structures. By \cite[Th. 8.20]{Ols}, it is enough to show that 
      \[\rL_{(N\to B)/(0\to A)} = 0.\]
      As $\rL_{(0\to B)/(0\to A)}\simeq \rL_{B/A} = 0$ (cf. \cite[Lem. 8.22]{Ols}), by considering the exact triangle
      \[\rL_{(0\to B)/(0\to A)}\to\rL_{(N\to B)/(0\to A)}\to\rL_{(N\to B)/(0\to B)},\]
      we are reduced to showing that
      \[\rL_{(N\to B)/(0\to B)} = 0.\]
      By \cite[Lem. 8.28]{Ols} and \cite[Lem. 8.23(ii)]{Ols}, we have quasi-isomorphisms
      \[\rL_{(N\to B)/(0\to B)} \simeq \rL_{(N\to\bZ[N])/(0\to\bZ)}\otimes_{\bZ[N]}^{\rL}B\simeq N^{\gp}\otimes_{\bZ}^{\rL}B,\]
      where $N^{\gp}$ denotes the group associated to the monoid $N$. It remains to show that
      \[N^{\gp}\otimes_{\bZ}^{\rL}B = 0.\]
      As $N$ is uniquely $p$-divisible, we see that $p$ acts isomorphically on $N^{\gp}$ and thus on $N^{\gp}\otimes_{\bZ}^{\rL}B$. This forces $N^{\gp}\otimes_{\bZ}^{\rL}B = 0$ because $pB=0$.

      For Item (2): Assume the log-structure $M_B\to B$ is associated to the pre-log-structure $N\to B$ with $N$ uniquely $p$-divisible. Similar to the proof of Item (1), we are reduced to the case to show that the derived $p$-adic completion of $N^{\gp}\otimes_{\bZ}^{\rL}B$ vanishes. By derived Nakayama's Lemma, this amounts to that
      \[N^{\gp}\otimes_{\bZ}^{\rL}B\otimes_{\bZ}^{\rL}\Fp = 0.\]
      But this is trivial because $N$ is uniquely $p$-divisible.
  \end{proof}

  Combining Proposition \ref{prop:log structure on perfectoids} together with Lemma \ref{lem:cotangent complex for perfect log structure}(2), we can conclude the following analogue of \cite[Lem. 3.14]{BMS18} immediately.
  \begin{prop}\label{prop:BMS Lem 314}
      For any morphism $(A,A^+)\to (B,B^+)$ of perfectoid affinoid algebras, the $p$-complete cotangent complex $\widehat \rL_{(B^{\times}\cap B^+\to B^+)/(A^{\times}\cap A^+\to A^+)} = 0$.
  \end{prop}

  \begin{convention}\label{convention:log structure}
      From now on, without other clarification, for any $n\geq 1$, we always regard 
      \[\Spf(\AAinfK(A,A^+)/\Ker(\theta_K)^n)\]
      as a $p$-adic log formal scheme with the canonical log-structure.
  \end{convention}

  Although Proposition \ref{prop:log structure on perfectoids} is enough for our use, to complete the theory, we deal with a general case at the end of this subsection. In what follows, let $A^+$ be any perfectoid ring with $A = A^+[\frac{1}{p}]$. By \cite[Lem. 3.9]{BMS18}, one can still choose a $\pi\in A^+$ and a $\underline \pi\in A^{+,\flat}$ such that $\pi = \underline{\pi}^{\sharp}$ and $\pi^p = pu$ for some $u\in A^{+,\times}$. In this case, we also have $A = A^+[\frac{1}{\pi}]$ and $A^{\flat} = A^{+,\flat}[\frac{1}{\underline \pi}]$ as well. Also, for any $\alpha\in \bN[1/p]$, the $\pi^{\alpha}$ is well-defined.
  \begin{prop}\label{prop:unit group for arbitrary perfectoid}
      Let $A^+$ be a perfectoid ring with $A = A^+[\frac{1}{p}]$. Then for any $x\in A^+$ whose image in $A$ is invertible, it always factors as
      \[x = \underline{x}^{\sharp}(1+\pi^{p\alpha} y_{\alpha}),\]
      where $\underline x\in A^{+,\flat}$ whose image in $A^{\flat}$ is invertible, $\alpha\in \bN[\frac{1}{p}]\cap(0,1)$ and $y_{\alpha}\in A^+$.
  \end{prop}
  \begin{proof}
      Fix an $\alpha\in \bN[\frac{1}{p}]\cap(0,1)$. 
      We first assume $A^+$ is $p$-torsion free. In this case, we have $A^+\subset A^{\circ}$ (and $A^{+,\flat}\subset A^{\flat,\circ}$). As $(A,A^{\circ})$ is perfectoid affinoid, by Proposition \ref{prop:unit group for perfectoids}, there exists an $\underline x\in A^{\flat,\circ}$ and an $y\in A^{\circ}$ such that $x = \underline{x}^{\sharp}(1+py)$. As $\sqrt{pA^{\circ}}\subset A^+$, we have $y_{\alpha}:=\frac{p}{\pi^{p\alpha}}y = u^{-1}\pi^{p(1-\alpha)}y\in A^+$ such that
      \[x = \underline{x}^{\sharp}(1+\pi^{p\alpha} y_{\alpha}).\]
      As $1+\pi^{p\alpha} y_{\alpha}\in A^{+,\times}$, we have $\underline{x}^{\sharp}\in A^+$, yielding that $\underline x\in A^{+,\flat}$. Clearly, we have $\underline x\in A^{\flat,\times}$ as desired.

      Now we move to the general case. Put $B^+:=A^+/A^+[\sqrt{pA^+}]$, $\overline{A}^+:= A^+/\sqrt{pA^+}$ and $\overline B^+:=B^+/\sqrt{pB^+}$. Then we have a commutative diagram of morphisms of perfectoid rings
      \begin{equation}\label{diag:fiber-cofiber square}
          \xymatrix@C=0.5cm{ A^+\ar[d]\ar[rr]&& B^+\ar[d]\\
           \overline A^+\ar[rr]&& \overline B^+}
      \end{equation}
      which is both a fiber square and a cofiber square (cf. \cite[Prop. 3.2]{Bha-LectNote}). Note that $B^+$ is $p$-torsion free with $B^+[\frac{1}{p}] = A$. As $A^+\to B^+$ is surjective, so is $A^{+,\flat}\to B^{+,\flat}$. By what we have proved, there exists a $w_{\alpha}\in A^+$ such that the image of
      \[y:=(1+\pi^{p\alpha}w_{\alpha})^{-1}x\]
      in $B^+$ is of the form $\underline{z}^{\sharp}$ for some $\underline z\in B^{+,\flat}\cap B^{\flat,\times}$. It remains to show $y = \underline{y}^{\sharp}$ for some $\underline y\in A^{+,\flat}$. Granting this, the image of $\underline y$ in $A^{\flat} = B^{\flat}$ coincides with the image of $\underline z$, which is invertible as desired.
      
      Write $\underline z = (z_0,z_1,\dots)\in B^{+,\flat} = \varprojlim_{x\mapsto x^p}B^+$. Denote by $\overline y$ the image of $y$ in $\overline A^+$. As \eqref{diag:fiber-cofiber square} is commutative, $\overline y$ and $z_0 = \underline{z}^{\sharp}$ coincide in $\overline B^+$. As both $\overline A^+$ and $\overline B^+$ are perfectoid in characteristic $p$ (and thus perfect), we have $\overline{y}^{p^{-n}}$ is well-defined in $\overline A^+$ and coincides with $z_n$ in $\overline B^+$ for any $n\geq 0$. So the $y_n:=(\overline{y}^{p^{-n}},z_n)\in \overline A^+\times B^+$ defines an element in $A^+$. By construction, we have $y_0 = y$ and $y_{n+1}^p = y_n$ for any $n\geq 0$. Put 
      \[\underline y = (y_0,y_1,\dots)\in \varprojlim_{x\mapsto x^p}A^+ = A^{+,\flat}\]
      and then we have $y = \underline{y}^{\sharp}$ as desired.
  \end{proof}
  \begin{cor}\label{cor:unit group for arbitrary perfectoid}
      Let $A^+$ be any perfectoid ring with $A = A^+[\frac{1}{p}]$. For any $?\in\{\emptyset,\flat\}$, put 
      \[A^{+,?}\cap A^{?,\times}:=\{x\in A^{+,?}\mid \text{the image of $x$ in $A^?$ is invertible.}\}.\]
      Then for any $\alpha\in \bN[\frac{1}{p}]\cap(0,1)$, we have 
      \[A^{+}\cap A^{\times} = (A^{+,\flat}\cap A^{\flat,\times})^{\sharp}\cdot(1+\pi^{p\alpha}A^+).\]
  \end{cor}
  \begin{proof}
      This follows from Proposition \ref{prop:unit group for arbitrary perfectoid} immediately.
  \end{proof}

\subsection{Semi-stable formal schemes over $A^+$}\label{ssec:semistable schemes}
  This subsection is closely related with the work of Cesnavi\v{c}ius and Koshikawa \cite{CK19}.
  Fix an affinoid perfectoid $S= \Spa(A,A^+)\in \Perfd$.
  By a \emph{semi-stable} formal scheme $\frakX_S$ of relative dimension $d$ over $A^+$, we mean a $p$-adic formal scheme $\frakX_S$ together with the log-structure $\calM_{\frakX_S}$ over $A^+$ (cf. Convention \ref{convention:log structure}), which is \'etale locally of the form $\Spf(R_S^+)$ with $R_S^+$ small semi-stable of relative dimension $d$ over $A^+$ as defined below. We remark that the generic fiber $X_S$ of a semi-stable $\frakX_S$ over $A^+$ is always smooth over $S$.
  \begin{dfn}\label{dfn:small affine}
      Fix an affinoid perfectoid $S = \Spa(A,A^+)\in\Perfd$.
      A $p$-complete $A^+$-algebra $R_S^+$ is called \emph{small semi-stable} of relative dimension $d$ over $A^+$, if there exists an \'etale morphism of $p$-adic formal schemes
      \[\psi: \Spf(R_S^+)\to \Spf(A^+\za T_0,\dots,T_r,T_{r+1}^{\pm 1},\dots, T_d^{\pm 1}\ya/(T_0\cdots T_r-p^a))\]
      for some $a\in\bQ$ and $0\leq r\leq d$. 
      Equip $R_S^+$ with the log-structure associated to the pre-log-structure
      \begin{equation}\label{equ:log-structure on chart-I}
          M_{r,a}(A^+):=(\oplus_{i=0}^r\bN\cdot e_i)\oplus_{\bN\cdot(e_1+\cdots+e_r)}(A^{\times}\cap A^+)\xrightarrow{(\sum_{i=0}^rn_ie_i,x)\mapsto x\prod_{i=0}^rT_i^{n_i}} R_S^+,
      \end{equation}
      where $(\oplus_{i=0}^r\bN\cdot e_i)\oplus_{\bN\cdot(e_1+\cdots+e_r)}(A^{\times}\cap A^+)$ denotes the push-out of monoids
      \[\xymatrix@C=0.5cm{
        \bN\ar[d]^{1\mapsto p^a}\ar[rrr]^{1\mapsto e_0+\cdots+e_r\qquad}&&&\bN\cdot e_1\oplus\cdots\oplus\bN\cdot e_r\\
        (A^{\times}\cap A^+).
      }\]
      In this case, we also say the semi-stable formal scheme $\Spf(R_S^+)$ is \emph{small affine}. We call such a $\psi$ (resp. $T_0,\dots,T_d$) a \emph{chart} (resp. \emph{coordinates}) on $R_S^+$ or on $\Spf(R_S^+)$. The generic fiber $\Spa(R_S,R_S^+)$ of $\Spf(R_S^+)$ is then smooth over $S$ and endowed with the induced chart $\psi$.
  \end{dfn}

  Denote by $\Omega^{1,\log}_{\frakX_S}$ the module of (continuous) log-differentials of $\frakX_S$ over $A^+$ and for any $n\geq 1$, define $\Omega^{n,\log}_{\frakX_S} = \wedge^n\Omega^{1,\log}_{\frakX_S}$. Then $\Omega^{n,\log}_{\frakX_S}$ is a locally finite free $\calO_{\frakX}$-module for any $n\geq 0$. When $\frakX_S = \Spf(R_S^+)$ is small affine, the $\Omega^{1,\log}_{\frakX_S}$ is associated to a finite free $R_S^+$-module $\Omega^{1,\log}_{R_S^+}$ of rank $d$, and the chart $\psi$ on $R_S^+$ induces an identification
  \begin{equation}\label{equ:omega^1}
      \begin{split}
          \Omega^{1,\log}_{R_S^+} &= \left((\oplus_{i=0}^d\bZ\cdot e_i)/\bZ\cdot(e_0+\cdots+e_r)\right)\otimes_{\bZ}R_S^+\oplus \left(R_S^+\cdot\dlog T_{r+1}\oplus\cdots\oplus R^+\cdot\dlog T_d\right)\\
          &=\left((\oplus_{i=0}^d R_S^+\cdot e_i)/R_S^+\cdot(e_0+\cdots+e_r)\right)\oplus \left(R_S^+\cdot\dlog T_{r+1}\oplus\cdots\oplus R_S^+\cdot\dlog T_d\right).
      \end{split}
  \end{equation}

  A semi-stable $\frakX_S$ over $A^+$ is called \emph{liftable}, if there exists a flat log $p$-adic formal scheme $\widetilde{\frakX}_S$ with the log-structure $\calM_{\widetilde \frakX_S}$ over $\AAKK(S):=\AAinfK(S)/\Ker(\theta_K)^2$ (as a log formal scheme with the canonical log-structure, cf. Convention \ref{convention:log structure}), such that $\frakX_S$ is the reduction of $\widetilde{\frakX}_S$ modulo $\xi_K$; that is, the underlying scheme $\frakX_S$ is the base-change of $\widetilde \frakX_S$ along the canonical surjection $\AAKK(S)\to A^+$ while the log-structure $\calM_{\frakX_S}$ is induced by the composite $\calM_{\widetilde \frakX_S}\to \calO_{\widetilde \frakX_S}\to\calO_{\frakX_S}$. 
  
  Given a semi-stable $\frakX_S$ over $A^+$, its lifting $\widetilde \frakX_S$ over $\AAKK(S)$ may not always exist.
  However, when $\frakX_S = \Spf(R_S^+)$ is small affine, by (log-)smoothness of $\frakX_S$, the lifting $\widetilde \frakX_S$ always exists and is unique up to isomorphisms. More precisely, the 
  \[\AAKK(S)\za T_0,\dots,T_r,T_{r+1}^{\pm 1},\dots,T_d^{\pm 1}\ya/(T_0\cdots T_r-[\varpi^a])\]
  with the log-structure associated to the pre-log-structure
  \begin{equation}\label{equ:log-structure on chart-II}
      \begin{split}
          M_{r,a}(A^{\flat,+})&:=(\oplus_{i=0}^r\bN\cdot e_i)\oplus_{\bN\cdot(e_1+\cdots+e_r)}(A^{\flat,\times}\cap A^{\flat,+})\\
          &\xrightarrow{(\sum_{i=0}^rn_ie_i,x)\mapsto [x]\prod_{i=0}^rT_i^{n_i}} \AAKK(S)\za T_0,\dots,T_r,T_{r+1}^{\pm 1},\dots,T_d^{\pm 1}\ya/(T_0\cdots T_r-[\varpi^a]),
      \end{split}
    \end{equation}
  where $(\oplus_{i=0}^r\bN\cdot e_i)\oplus_{\bN\cdot(e_1+\cdots+e_r)}(A^{\flat,\times}\cap A^{\flat,+})$ denotes the push-out of monoids
  \[\xymatrix@C=0.5cm{
    \bN\ar[d]^{1\mapsto \varpi^a}\ar[rrr]^{1\mapsto e_0+\cdots+e_r\qquad}&&&\bN\cdot e_1\oplus\cdots\oplus\bN\cdot e_r\\
    (A^{\flat,\times}\cap A^{\flat,+}),
  }\]
  is a lifting of the log-structure associated to 
  \[(M_{r,a}(A^+)\to A^+\za T_0,\dots,T_r,T_{r+1}^{\pm 1},\dots,T_d^{\pm 1}\ya/(T_0\cdots T_r-p^a)).\]
  By the \'etaleness of the chart $\psi$, there exists a unique $\widetilde{R}_S^+$ together with a unique homomorphism of $\AAKK(S)$-algebras 
  \[\widetilde{\psi}:\AAKK(S)\za T_0,\dots,T_r,T_{r+1}^{\pm 1},\dots,T_d^{\pm 1}\ya/(T_0\cdots T_r-[\varpi^a])\to\widetilde{R}_S^+\]
  lifting $\psi$. Then we have $\widetilde{\frakX}_S = \Spf(\widetilde{R}^+_S)$ with the log-structure $\calM_{\widetilde \frakX_S}$ induced by the composite
  \begin{equation}\label{equ:log-structure on chart-III}
      M_{r,a}(A^{\flat,+})\to\AAKK(S)\za T_0,\dots,T_r,T_{r+1}^{\pm 1},\dots,T_d^{\pm 1}\ya/(T_0\cdots T_r-[\varpi^a])\xrightarrow{\widetilde \psi}\widetilde{R}_S^+.
  \end{equation}
  Using this, it is clear that $\widetilde \frakX_S$ is (log-)smooth over $\AAKK(S)$.

  By definition, any smooth $\frakX$ over $A^+$ endowed with the induced log-structure from $A^+$ is always semi-stable over $A^+$.
  We now give another typical example of semi-stable formal schemes over $A^+$.
  \begin{exam}\label{exam:semi-stable formal scheme after CK}
      The semi-stable formal scheme over $\calO_C$ defined above is exactly the semi-stable formal scheme $\frakX$ with the log-structure $(\calM_{\frakX} = \calO_X^{\times}\cap\calO_{\frakX}\hookrightarrow\calO_{\frakX})$ considered in \cite[\S1.5 and \S1.6]{CK19}, where $X$ is the generic fiber of $\frakX$. For any $S = \Spa(A,A^+)$, denote by $\frakX_S$ the base-change of $\frakX$ along $\Spf(A^+)\to\Spf(\calO_C)$ (viewed log formal schemes, Convention \ref{convention:log structure}) with the fiber product log-structure $\calM_{\frakX_S}$. Then $\frakX_S$ is a semi-stable formal scheme over $A^+$. 
      Moreover, if $\frakX$ is liftable, then so is $\frakX_S$. Indeed, let $\widetilde \frakX$ (with the log-structure $\calM_{\frakX}$) is a lifting of $\frakX$ over $\bfA_{2,K}$, then its base-change $\widetilde \frakX_S$ along $\bfA_{2,K}\to\AAKK(S)$ with the log-structure $\calM_{\widetilde \frakX_S}$ induced from the fiber product gives rise to a lifting of $\frakX_S$. This is the typical case we shall work in.
  \end{exam}

  Now, we are going to introduce some notations, which will be used in local calculations, as in \cite[\S3.2]{CK19} for small semi-stable $\frakX_S = \Spf(R_S^+)$ over $A^+$ with the chart $\psi$ as in Definition \ref{dfn:small affine}. Denote its generic fiber by $X_S = \Spa(R_S,R_S^+)$.

  For any $n\geq 0$, put 
  \[A_{r,a,n}^+:=A^+\za T_0^{\frac{1}{p^n}},\dots,T_r^{\frac{1}{p^n}},T_{r+1}^{\pm \frac{1}{p^n}},\dots, T_d^{\pm \frac{1}{p^n}}\ya/(T_0^{\frac{1}{p^n}}\cdots T_r^{\frac{1}{p^n}}-p^\frac{a}{p^n})\]
  and 
  \begin{equation*}
     M_{r,a,n}(A^+):=(\oplus_{i=0}^r\frac{1}{p^n}\bN\cdot e_i)\oplus_{\frac{1}{p^n}\bN\cdot(e_1+\cdots+e_r)}(A^{\times}\cap A^+)\xrightarrow{(\sum_{i=0}^r\frac{n_i}{p^n}e_i,x)\mapsto x\prod_{i=0}^rT_i^{n_i}} A_{r,a,n},
  \end{equation*}
  where $(\oplus_{i=0}^r\frac{1}{p^n}\bN\cdot e_i)\oplus_{\frac{1}{p^n}\bN\cdot(e_1+\cdots+e_r)}(A^{\times}\cap A^+)$ denotes the push-out of monoids
  \[\xymatrix@C=0.5cm{
    \frac{1}{p^n}\bN\ar[d]^{\frac{1}{p^n}\mapsto p^{\frac{a}{p^n}}}\ar[rrr]^{\frac{1}{p^n}\mapsto\sum_{i=0}^r\frac{1}{p^n} e_i\quad\qquad}&&&\frac{1}{p^n}\bN\cdot e_1\oplus\cdots\oplus\frac{1}{p^n}\bN\cdot e_r\\
    (A^{\times}\cap A^+).
  }\]
  Put $A^+_{r,a,\infty}:= (\colim_{n}A^+_{r,a,n})^{\wedge}$ and $M_{r,a,\infty}(A^+):=\colim_{n}M_{r,a,n}(A^+)$. Then $A^+_{r,a,\infty}$ is perfectoid over $A^+$ and the natural map $M_{r,a,\infty}\to A^+_{r,a,\infty}$ induces a perfect log-structure on $A^+_{r,a,\infty}$ (cf. Definition \ref{dfn:perfect log-structure}). Indeed, we have
  \begin{equation*}
     M_{r,a,\infty}(A^+):=(\oplus_{i=0}^r\bN[\frac{1}{p}]\cdot e_i)\oplus_{\bN[\frac{1}{p}]\cdot(e_1+\cdots+e_r)}(A^{\times}\cap A^+)\xrightarrow{(\sum_{i=0}^r\frac{n_i}{p^n}e_i,x)\mapsto x\prod_{i=0}^rT_i^{n_i}} A^+_{r,a,\infty},
  \end{equation*}
  where $(\oplus_{i=0}^r\frac{1}{p^n}\bN\cdot e_i)\oplus_{\frac{1}{p^n}\bN\cdot(e_1+\cdots+e_r)}(A^{\times}\cap A^+)$ denotes the push-out of monoids
  \[\xymatrix@C=0.5cm{
    \bN[\frac{1}{p}]\ar[d]^{\frac{1}{p^n}\mapsto p^{\frac{a}{p^n}}}\ar[rrr]^{\frac{1}{p^n}\mapsto\sum_{i=0}^r\frac{1}{p^n} e_i\quad\qquad}&&&\bN[\frac{1}{p}]\cdot e_1\oplus\cdots\oplus\bN[\frac{1}{p}]\cdot e_r\\
    (A^{\times}\cap A^+).
  }\]
  Note that $\Spa(A_{r,a,\infty}=A^+_{r,a,\infty}[\frac{1}{p}],A^+_{r,a,\infty})\to \Spa(A_{r,a,0}=A^+_{r,a,0}[\frac{1}{p}],A^+_{r,a,0})$ is a pro-\'etale Galois covering with the Galois group 
  \begin{equation}\label{equ:Gamma group}
      \Gamma \cong \{\delta=\delta_0^{n_0}\cdots\delta_r^{n_r}\delta_{r+1}^{n_{r+1}}\cdots\delta_d^{n_d}\mid n_i\in \Zp,~\forall~0\leq i\leq d, \text{ such that }n_0+\cdots+n_r = 0\}\cong \bZ_p^{\oplus d},
  \end{equation}
  where the action of $\Gamma$ on $A^+_{r,a,\infty}$ is uniquely determined such that for any $0\leq i\leq d$, and $n\geq 0$ and any $\delta =\delta_0^{n_0}\cdots\delta_d^{n_d}\in\Gamma$, we have
  \begin{equation}\label{equ:Gamma action}
      \delta(T_i^{\frac{1}{p^n}}) = \zeta_{p^n}^{n_i}T_i^{\frac{1}{p^n}}.
  \end{equation}
  Put $\gamma_i:=\delta_0^{-1}\delta_i$ when $1\leq i\leq r$ and $\gamma_j = \delta_j$ when $r+1\leq j\leq d$. Then we have an isomorphism 
  \begin{equation}\label{equ:Gamma group-II}
      \Gamma \cong \Zp\cdot\gamma_1\oplus\cdots\oplus\Zp\cdot\gamma_d.
  \end{equation}
  This is useful in some calculations.

  Let $X_{\infty,S} = \Spa(\widehat R_{\infty,S},\widehat R_{\infty,S}^+)$ be the base-change of $X_S$ along $\Spa(A_{r,a,\infty},A^+_{r,a,\infty})\to \Spa(A_{r,a,0},A^+_{r,a,0})$ with respect to the chart $\psi:X_S\to \Spa(A_{r,a,0},A^+_{r,a,0})$. Then $X_{\infty,S}$ is affinoid perfectoid such that $X_{\infty,S}\to X_S$ is a pro-\'etale Galois covering with Galois group $\Gamma$ whose action on $\widehat R_{\infty,S}^+$ is determined by (\ref{equ:Gamma action}). Consider the set of indices
  \begin{equation}\label{equ:index set}
      J_r:=
      \{\underline \alpha = (\alpha_0,\dots,\alpha_r,\alpha_{r+1},\dots,\alpha_d)\in (\bN[\frac{1}{p}]\cap[0,1))^{d+1}\mid \prod_{i=0}^r\alpha_i = 0\}.
  \end{equation}
  Then $\widehat R_{\infty,S}^+$ admits a $\Gamma$-equivariant decomposition
  \begin{equation}\label{equ:Gamma decomposition}
      \widehat R_{\infty,S}^+ = \widehat \bigoplus_{\underline \alpha\in J_r}R^+_S\cdot \underline T^{\underline \alpha},
  \end{equation}
  where ``$\widehat \oplus$'' denotes the $p$-adic topological direct sum. Note that the composite
  \[M_{r,a,\infty}(A^+)\to A^+_{r,a,\infty}\to \widehat R_{\infty,S}^+\]
  defines a perfect log-structure on $\widehat R_{\infty,S}^+$, which factors through the canonical log-structure on $\widehat R_{\infty,S}^+$ because $T_i\in \widehat R_{\infty,S}^{\times}\cap \widehat R_{\infty,S}^+$ for any $0\leq i\leq r$.

  For any $0\leq i\leq d$, let $T_i^{\flat}:=(T_i,T_i^{\frac{1}{p}},\dots)\in \widehat R_{\infty,S}^{\flat,\times}\cap \widehat R_{\infty,S}^{\flat,+}$. Then the map
  \[\iota_{\psi}:\AAKK(S)\za T_0,\dots,T_r,T_{r+1}^{\pm 1},\dots,T_d^{\pm 1}\ya/(T_0\cdots T_r-[\varpi^a])\to\AAKK(X_{\infty})\]
  carrying each $T_i$ to $[T_i]^{\flat}$ is a well-defined morphism of $\AAKK(S)$-algebras compatible with log-structures on the source and target. By the \'etaleness of $\widetilde \psi$ above, the $\iota_{\psi}$ uniquely extends to a morphism of $\AAKK(S)$-algebras (still denoted by)
  \begin{equation}\label{equ:iota_psi}
      \iota_{\psi}:(M_{r,a}(A^{\flat,+})\to \widetilde R_S^+)\to(\widehat R_{\infty,S}^{\flat,\times}\cap \widehat R_{\infty,S}^{\flat,+}\xrightarrow{[\cdot]}\AAKK(X_{\infty})),
  \end{equation}
  lifting the natural morphism $(M_{r,a}(A^{+})\to R_S^+)\to(\widehat R_{\infty,S}^{\times}\cap \widehat R_{\infty,S}^{+}\to\widehat R_{\infty,S}^+)$. The $\Gamma$-action on $X_{\infty}$ induces a $\Gamma$-action on $\widehat R_{\infty,S}^{\flat}$ and thus on $\AAKK(X_{\infty})$ such that for any $0\leq i\leq d$, any $n\geq 0$ and any $\delta =\delta_0^{n_0}\cdots\delta_d^{n_d}\in\Gamma$, we have
  \begin{equation}\label{equ:Gamma action-II}
      \delta((T_i^{\flat})^{\frac{1}{p^n}}) = \epsilon^{\frac{n_i}{p^n}}(T_i^{\flat})^{\frac{1}{p^n}}.
  \end{equation}
  Note that the $\iota_{\psi}$ above is \emph{not} $\Gamma$-equivariant!

\section{Period sheaves}\label{sec:period sheaves}

  Fix an $S = \Spa(A,A^+)\in \Perfd$. Throughout this section, we always assume $\frakX_S$ is a liftable semi-stable formal scheme over $A^+$ with the generic fiber $X_S$ and a fix an its lifting $\widetilde \frakX_S$ over $\AAKK(S)$. Denote by $X_{S,v}$ the $v$-site associated to $X_S$ in the sense of \cite{Sch-Diamond}, and by $\widehat \calO_{X_S}$ (resp. $\widehat \calO_{X_S}^+$, $\widehat \calO_{X_S}^{\flat}$ and $\widehat \calO_{X_S}^{\flat,+}$) the sheaf sending each affinoid perfectoid $U = \Spa(B,B^+)\in X_{S,v}$ to $B^{\flat}$ (resp. $B^+$, $B^{\flat}$ and $B^{\flat,+}$). For any $n\geq 1$, the \emph{canonical log-structure} on $\AAinfK(\widehat \calO^+_{X_S})/\Ker(\theta_K)^n$ is the log-structure associated to the morphism of monoids 
  \[\widehat \calO_{X_S}^{\flat,\times}\cap \widehat \calO_{X_S}^{\flat,+}\xrightarrow{[\cdot]}\AAinfK(\widehat \calO^+_{X_S})/\Ker(\theta_K)^n.\]
  It follows from Proposition \ref{prop:log structure on perfectoids} that the canonical log-structure on $\widehat \calO_{X_S}^+$ is exactly the log-structure
  \[\widehat \calO_{X_S}^{\times}\cap\widehat \calO_{X_S}^+=:\calM_{X_S}\hookrightarrow\widehat \calO_{X_S}^+.\]

\subsection{Integral Faltings' extension}\label{ssec:integral Faltings' extension}
  In this part, we follow the argument in \cite[\S2]{Wan23} to construct an analogue of integral Faltings' extension in \emph{loc.cit.} on $X_{S,v}$ with respect to the given lifting $\widetilde \frakX_S$. 
  
  Denote by $\calM_{\frakX_S}$ and $\calM_{\widetilde \frakX_S}$ the log-structures on $\frakX_S$ and $\widetilde \frakX_S$, respectively. Then we have the morphisms of log-ringed topoi over $\AAKK(S)$:
  \[(\calO_{\widetilde \frakX_S},\calM_{\widetilde \frakX_S})\to(\calO_{\frakX_S},\calM_{\frakX_S})\to(\widehat \calO^+_{X_S},\calM_{X_S}).\]
  This gives rise to an exact triangle of $p$-complete cotangent complexes 
  \begin{equation}\label{equ:exact triangle of ringed-topoi}
      \widehat \rL_{(\calO_{\frakX_S},\calM_{\frakX_S})/(\calO_{\widetilde \frakX_S},\calM_{\widetilde \frakX_S})}\widehat \otimes^{\rL}_{\calO_{\frakX_S}}\widehat \calO^+_{X_S}\to\widehat \rL_{(\widehat \calO^+_{X_S},\calM_{X_S})/(\calO_{\widetilde \frakX_S},\calM_{\widetilde \frakX_S})}\to \widehat \rL_{(\widehat \calO^+_{X_S},\calM_{X_S})/(\calO_{\frakX_S},\calM_{\frakX_S})}.
  \end{equation}
  
  The first term is easy to handle with: As the log-structure $\calM_{\frakX_S}$ is induced from $\calM_{\widetilde \frakX_S}$ via the composite $\calM_{\widetilde \frakX_S}\to\calO_{\widetilde \frakX_S}\to\calO_{\frakX_S}$, it follows from \cite[Lem. 8.22]{Ols} that
  \[\widehat \rL_{(\calO_{\frakX_S},\calM_{\frakX_S})/(\calO_{\widetilde \frakX_S},\calM_{\widetilde \frakX_S})}\simeq \widehat \rL_{\calO_{\frakX_S}/\calO_{\widetilde \frakX_S}}.\]
  As $\widetilde \frakX_S$ is flat over $\AAKK(S)$, using \cite[Cor. 2.3]{Wan23}, we have quasi-isomorphisms
  \[\widehat \rL_{\calO_{\frakX_S}/\calO_{\widetilde \frakX_S}}\simeq \widehat \rL_{A^+/\AAKK(S)}\widehat \otimes^{\rL}_{A^+}\calO_{\frakX_S}\simeq \calO_{\frakX_S}\{1\}[1]\oplus\calO_{\frakX_S}\{2\}[2].\]
  So we finally conclude that
  \begin{equation}\label{equ:first term}
      \widehat \rL_{(\calO_{\frakX_S},\calM_{\frakX_S})/(\calO_{\widetilde \frakX_S},\calM_{\widetilde \frakX_S})}\widehat \otimes^{\rL}_{\calO_{\frakX_S}}\widehat \calO^+_{X_S}\simeq \widehat \calO^+_{X_S}\{1\}[1]\oplus\widehat \calO^+_{X_S}\{2\}[2].
  \end{equation}

  The last term of (\ref{equ:exact triangle of ringed-topoi}) is also easy to handle with: Consider the morphisms of log-rings 
  \[(A^{\times}\cap A^+=:M_A\to A^+)\to(\calO_{\frakX_S},\calM_{\frakX_S})\to(\widehat \calO^+_{X_S},\calM_{X_S})\]
  and the induced exact triangle
  \[\widehat \rL_{(\calO_{\frakX_S},\calM_{\frakX_S})/(A^+,M_A)}\widehat \otimes^{\rL}_{\calO_{\frakX_S}}\widehat \calO^+_{X_S}\to\widehat \rL_{(\widehat \calO^+_{X_S},\calM_{X_S})/(A^+,M_A)}\to \widehat \rL_{(\widehat \calO^+_{X_S},\calM_{X_S})/(\calO_{\frakX_S},\calM_{\frakX_S})}.\]
  By Proposition \ref{prop:BMS Lem 314}(2), the middle term $\widehat \rL_{(\widehat \calO^+_{X_S},\calM_{X_S})/(A^+,M_A)}$ above vanishes, yielding a quasi-isomorphism
  \[\widehat \rL_{(\widehat \calO^+_{X_S},\calM_{X_S})/(\calO_{\frakX_S},\calM_{\frakX_S})}\simeq (\widehat \rL_{(\calO_{\frakX_S},\calM_{\frakX_S})/(A^+,M_A)}\widehat \otimes^{\rL}_{\calO_{\frakX_S}}\widehat \calO^+_{X_S})[1].\]
  If $\frakX_S$ is not only semi-stable but smooth over $A^+$, then we have a quasi-isomorphism
  \[\widehat \rL_{(\calO_{\frakX_S},\calM_{\frakX_S})/(A^+,M_A)}\simeq \Omega^{1,\log}_{\frakX_S}[0] = \Omega^1_{\frakX_S}[0].\]
  However, in the logarithmic case, it is not straightforward that $\widehat \rL_{(\calO_{\frakX_S},\calM_{\frakX_S})/(A^+,M_A)}$ is discrete. So we have to exhibit the discreteness of $\widehat \rL_{(\calO_{\frakX_S},\calM_{\frakX_S})/(A^+,M_A)}$ directly.
  \begin{lem}\label{lem:dicrete cotangent complex}
      We have $\widehat \rL_{(\calO_{\frakX_S},\calM_{\frakX_S})/(A^+,M_A)}\simeq \Omega^{1,\log}_{\frakX_S}[0]$.
  \end{lem}
  \begin{proof}
      It suffices to show the complex $\widehat \rL_{(\calO_{\frakX_S},\calM_{\frakX_S})/(A^+,M_A)}$ is discrete. Since the problem is local on $\frakX_{S,\et}$, we may assume $\frakX_S = \Spf(R_S^+)$ is small affine such that the log-structure $\calM_{\frakX_S}$ is induced by the pre-log-structure $M_{r,a}(A^+)\oplus(\oplus_{j=r+1}^d\bZ\cdot e_j)\to R_S^+$ where $M_{r,a}(A^+)$ is defined in (\ref{equ:log-structure on chart-I}) and $e_j$ is mapped to $T_j$ for any $r+1\leq j\leq d$. To conclude, we have to show 
      \[\widehat \rL_{(M_{r,a}(A^+)\oplus(\oplus_{j=r+1}^d\bZ\cdot e_j)\to R_S^+)/(M_A\to A^+)}\simeq \Omega^{1,\log}_{R_S^+}[0].\]
      By the \'etaleness of $A_{r,a}^+:=A^+\za T_0,\dots,T_r,T_{r+1}^{\pm 1},\dots,T_d^{\pm 1}\ya/(T_0\cdots T_r-p^a)\to R_S^+$, using \cite[Lem. 8.22]{Ols}, we are reduced to the case to show 
      \[\widehat \rL_{M_{r,a}(A^+)\oplus(\oplus_{j=r+1}^d\bZ\cdot e_j)\to A_{r,a}^+)/(M_A\to A^+)}\simeq \Omega^{1,\log}_{A_{r,a}^+}[0].\]
      Put $A_{r,a}^{+,\rm{nc}} = A^+[ T_0,\dots,T_r,T_{r+1}^{\pm 1},\dots,T_d^{\pm 1}]/(T_0\cdots T_r-p^a)$. Then we have
      \[A_{r,a}^{+,\rm{nc}}\cong A^+\otimes_{\bZ[M_A]}\bZ[M_{r,a}(A^+)\oplus(\oplus_{j=r+1}^d\bZ\cdot e_j)].\]
      As $\bZ[M_{r,a}(A^+)\oplus(\oplus_{j=r+1}^d\bZ\cdot e_j)]$ is flat over $\bZ[M_A]$, we see that $A_{r,a}^{+,\rm{nc}}$ is flat over $A^+$ and thus flat over $\Zp$. So we have a quasi-isomorphism
      \[\widehat \rL_{(M_{r,a}(A^+)\oplus(\oplus_{j=r+1}^d\bZ\cdot e_j)\to A_{r,a}^+)/(M_A\to A^+)}\simeq \widehat \rL_{(M_{r,a}(A^+)\oplus(\oplus_{j=r+1}^d\bZ\cdot e_j)\to A_{r,a}^{+,\rm{nc}})/(M_A\to A^+)}.\]
      By the flat base-change theorem \cite[Cor. 8.13]{Ols}, we then have a quasi-isomorphism
      \[\rL_{(M_{r,a}(A^+)\oplus(\oplus_{j=r+1}^d\bZ\cdot e_j)\to A_{r,a}^{+,\rm{nc}})/(M_A\to A^+)}\simeq \rL_{(M_{r,a}(A^+)\oplus(\oplus_{j=r+1}^d\bZ\cdot e_j)\to \bZ[M_{r,a}(A^+)\oplus(\oplus_{j=r+1}^d\bZ\cdot e_j)])/(M_A\to \bZ[M_A])}\otimes^{\rL}_{\bZ[M_A]}A^+.\]
      According to \cite[Lem. 8.23(ii)]{Ols}, we have a quasi-isomorphism
      \[\begin{split}
          \rL_{(M_{r,a}(A^+)\oplus(\oplus_{j=r+1}^d\bZ\cdot e_j)\to A_{r,a}^{+,\rm{nc}})/(M_A\to A^+)}&\simeq \left(M_{r,a}(A^+)^{\gp}/M_A^{\gp}\oplus(\oplus_{j=r+1}^d\bZ)\right)\otimes_{\bZ}^{\rL}A^+\\
          &\simeq \big(\left((\oplus_{i=0}^dA^+\cdot e_i)/A^+\cdot(e_0+\cdots+e_r)\right)\oplus(\oplus_{j=r+1}^dA^+\cdot e_j)\big)[0].
      \end{split}\]
      So we finally conclude a quasi-isomorphism 
      \[\widehat \rL_{(M_{r,a}(A^+)\oplus(\oplus_{j=r+1}^d\bZ\cdot e_j)\to A_{r,a}^+)/(M_A\to A^+)}\simeq \big(\left((\oplus_{i=0}^dA^+\cdot e_i)/A^+\cdot(e_0+\cdots+e_r)\right)\oplus(\oplus_{j=r+1}^dA^+\cdot e_j)\big)[0],\]
      yielding the desired discreteness of $\widehat \rL_{(M_{r,a}(A^+)\oplus(\oplus_{j=r+1}^d\bZ\cdot e_j)\to A_{r,a}^+)/(M_A\to A^+)}$.
  \end{proof}
  Thanks to the above lemma, the third term in (\ref{equ:exact triangle of ringed-topoi}) reads
  \begin{equation}\label{equ:third term}
      \widehat \rL_{(\widehat \calO^+_{X_S},\calM_{X_S})/(\calO_{\frakX_S},\calM_{\frakX_S})}\simeq (\widehat \calO^+_{X_S}\otimes_{\calO_{\frakX_S}}\Omega^{1,\log}_{\frakX_S})[1].
  \end{equation}

  \begin{dfn}[Integral Faltings' extension]\label{dfn:integral Faltings' extension}
      Let $\frakX_S$ be a liftable semi-stable formal scheme over $A^+$ with the generic fiber $X_S$. For a lifting $\widetilde \frakX_S$ of $\frakX_S$ over $\AAKK(S)$ (with the log-structure $\calM_{\widetilde \frakX_S}$), we call
      \[\calE^+_{\widetilde \frakX_S}:=\rH^{0}(\widehat \rL_{(\widehat \calO^+_{X_S},\calM_{X_S})/(\calO_{\widetilde \frakX_S},\calM_{\widetilde \frakX_S})}[-1])\]
      the \emph{integral Faltings' extension} associated to the lifting $\widetilde \frakX_S$.
  \end{dfn}
  \begin{prop}\label{prop:Faltings' extension}
      The integral Faltings' extension $\calE^+_{\widetilde \frakX_S}$ is a locally finite free $\widehat \calO^+_{X_S}$-module of rank $d+1$ and fits into the following short exact sequence
      \begin{equation}\label{equ:Faltings' extension}
          0\to \widehat \calO^+_{X_S}\{1\}\to \calE^+_{\widetilde \frakX_S}\to \widehat \calO^+_{X_S}\otimes_{\calO_{\frakX_S}}\Omega^{1,\log}_{\frakX_S}\to 0.
      \end{equation}
      Moreover, the cohomological class $[\calE^+_{\widetilde \frakX_S}]\in \Ext^1(\widehat \calO^+_{X_S}\otimes_{\calO_{\frakX_S}}\Omega^{1,\log}_{\frakX_S},\widehat \calO^+_{X_S}\{1\})$ is exactly the obstruction class for lifting the natural morphism of $A^+$-algebras $(\calO_{\frakX_S},\calM_{\frakX_S})\to(\widehat \calO^+_{X_S},\calM_{X_S})$ to a morphism of $\AAKK(S)$-algebras $(\calO_{\widetilde \frakX_S},\calM_{\widetilde \frakX_S})\dasharrow(\AAKK(\widehat \calO^+_{X_S}),\calM_{\widetilde X_S})$, where $\calM_{\widetilde X_S}$ denotes the canonical log-structure on $\AAKK(\widehat \calO^+_{X_S})$.
  \end{prop}
  \begin{proof}
      The first statement follows from Equations (\ref{equ:exact triangle of ringed-topoi}), (\ref{equ:first term}) and (\ref{equ:third term}) immediately. It remains to prove the ``moreover'' part. By the construction of $\calE^+_{\widetilde \frakX_S}$, as argued in \cite[Lem. 2.10]{Wan23}, the class $[\calE^+_{\widetilde \frakX_S}]$ is exactly the obstruction class for lifting $(\widehat \calO^+_{X_S},\calM_{X_S})$ to a log-ring over $(\calO_{\widetilde \frakX_S},\calM_{\widetilde \frakX_S})$.
      As $\widehat \rL_{(\widehat \calO^+_{X_S},\calM_{X_S})/(A^+,M_A)} = 0$ by Proposition \ref{prop:BMS Lem 314}(2), the lifting of $(\widehat \calO^+_{X_S},\calM_{X_S})$ over $(\AAKK(S),M_{\AAKK(S)})$ is unique (up to the unique isomorphism), where $M_{\AAKK(S)}$ denotes the canonical log-structure on $\AAKK(S)$. So it must be $(\AAKK(\widehat \calO^+_{X_S}),\calM_{\widetilde X_S})$. Thus one can conclude by noting that lifting the morphism
      \[(\calO_{\frakX_S},\calM_{\frakX_S})\to(\widehat \calO^+_{X_S},\calM_{X_S})\]
      over $(\calO_{\widetilde \frakX_S},\calM_{\widetilde \frakX_S})$ amounts to lifting the morphism of $A^+$-algebras 
      \[(\calO_{\frakX_S},\calM_{\frakX_S})\to(\widehat \calO^+_{X_S},\calM_{X_S})\]
      to a morphism of $\AAKK(S)$-algebras 
      \[(\calO_{\widetilde \frakX_S},\calM_{\widetilde \frakX_S})\dasharrow(\AAKK(\widehat \calO^+_{X_S}),\calM_{\widetilde X_S}).\]
      Using Proposition \ref{prop:BMS Lem 314}(2) instead of \cite[Lem. 3.14]{BMS18} in \emph{loc.cit.}, the result follows from the same proof for \cite[Prop. 2.14]{Wan23}. 
  \end{proof}

   When the lifting $\widetilde \frakX_S$ of $\frakX_S$ is fixed (as we did at the very beginning of this section), we also use $\calE^+$ to stand for the Breuil--Kisin--Fargues twist of the corresponding integral Faltings' extension 
   \begin{equation}\label{equ:twisted Faltings extension}
       \calE^+:=\calE^+_{\widetilde \frakX_S}\{-1\}
   \end{equation}
   for short. The following result is obvious from Proposition \ref{prop:Faltings' extension}.
   \begin{cor}\label{cor:Faltings' extension}
       The $\calE^+$ is a locally finite free $\widehat \calO^+_{X_S}$-module of rank $d+1$ and fits into the following short exact sequence
      \begin{equation*}
          0\to \widehat \calO^+_{X_S}\to \calE^+\to \widehat \calO^+_{X_S}\otimes_{\calO_{\frakX_S}}\Omega^{1,\log}_{\frakX_S}\{-1\}\to 0.
      \end{equation*}
   \end{cor}
   
  In the rest of this subsection, we want to describe of $\calE^+$ in the case where $\frakX_S = \Spf(R^+_S)$ is small affine with the chart $\psi$ (cf. Definition \ref{dfn:small affine}). We adapt the notations below Example \ref{exam:semi-stable formal scheme after CK}. Put $E^+_{\widetilde R_S^+}:=\calE^+_{\widetilde \frakX_S}(X_{\infty})$ and then it admits a continuous action of $\Gamma$ fitting into the following $\Gamma$-equivariant short exact sequence
  \begin{equation*}
      0 \to \widehat R_{\infty,S}^+\{1\} \to E^+_{\widetilde R_S^+}\to \widehat R_{\infty,S}^+\otimes_{R_S^+}\Omega^{1,\log}_{R_S^+}\to 0.
  \end{equation*}
  As an $\widehat R_{\infty,S}^+$-module, we have 
  \[E^+_{\widetilde R_S^+}\cong \widehat R_{\infty,S}^+\{1\}\oplus\widehat R_{\infty,S}^+\otimes_{R_S^+}\Omega^{1,\log}_{R_S^+}.\]
  So the $E^+_{\widetilde R_S^+}$ is uniquely determined by its $\Gamma$-action; that is, it is determined by an $1$-cocycle
  \[c\in \rH^1(\Gamma,\Hom_{\widehat R_{\infty,S}^+}(\widehat R_{\infty,S}^+\otimes_{R_S^+}\Omega^{1,\log}_{R_S^+},\widehat R_{\infty,S}^+\{1\}))\cong \rH^1(\Gamma,\Hom_{R_{S}^+}(\Omega^{1,\log}_{R_S^+},\widehat R_{\infty,S}^+\{1\})).\]
  
  Now, we are going to calculate the above $1$-cocycle $c$ by using that $E^+_{\widetilde R_S^+}$ stands for the obstruction class for lifting the $\Gamma$-equivariant morphism
  \[(M_{r,a}(A^+)\to R_S^+)\to (\widehat R_{\infty,S}^{\times}\cap \widehat R_{\infty,S}^{+}\to \widehat R_{\infty,S}^{+})\]
  to a $\Gamma$-equivariant morphism
  \[(M_{r,a}(A^{\flat,+})\oplus(\oplus_{j=r+1}^d\bZ\cdot e_j)\xrightarrow{\alpha} \widetilde R_S^+)\dasharrow (\widehat R_{\infty,S}^{\flat,\times}\cap \widehat R_{\infty,S}^{\flat,+}\to \AAKK(X_{\infty,S})).\]
  Here, the map $\alpha: M_{r,a}(A^{\flat,+})\oplus(\oplus_{j=r+1}^d\bZ\cdot e_j)\to \widetilde R_S^+$ is determined by (\ref{equ:log-structure on chart-III}) together with that $\alpha(e_j) = T_j$ for all $r+1\leq j\leq d$. As in the proof of Lemma \ref{lem:dicrete cotangent complex}, we have an isomorphism of $R^+_S$-modules
  \begin{equation}\label{equ:omega^1-II}
      \Omega^{1,\log}_{R_S^+}\cong (\oplus_{i=0}^d R_S^+\cdot e_i)/R_S^+\cdot(e_0+\cdots+e_r),
  \end{equation}
  where $e_j$ stands for $\dlog T_j$ for any $r+1\leq j\leq d$ via the identification (\ref{equ:omega^1}). Let
  \[\iota_{\psi}:(M_{r,a}(A^{\flat,+})\oplus(\oplus_{j=r+1}^d\bZ\cdot e_j)\xrightarrow{\alpha} \widetilde R_S^+)\to(\widehat R_{\infty,S}^{\flat,\times}\cap \widehat R_{\infty,S}^{\flat,+}\xrightarrow{[\cdot]}\AAKK((X_{\infty,S}))\]
  be the morphism of $\AAKK(S)$-algebras introduced in (\ref{equ:iota_psi}). Then one can describe the $1$-cocycle $c$ above by using $\iota_{\psi}$: For any $\delta = \prod_{i=0}^d\delta_i^{n_i}\in \Gamma$ (\ref{equ:Gamma group}), the map 
  \[c(\delta):\Omega^{1,\log}_{R_S^+}\to \widehat R_{\infty,S}^+\]
  is determined by that for any $0\leq j\leq d$, in $\xi\AAKK(X_{\infty,S})\cong \widehat R_{\infty,S}^+\{1\}$,
  \[c(\delta)(e_j)\cdot\alpha(e_j) = \delta(\iota_{\psi}(\alpha(e_j)))-\iota_{\psi}(\alpha(e_j)) .\]
  As $\alpha(e_j) = T_j$ for all $j$, it then follows from (\ref{equ:Gamma action-II}) that
  \[c(\delta)(e_j) = [\epsilon^{n_j}]-1 = \frac{[\epsilon^{n_j}]-1}{[\epsilon]-1}\cdot([\epsilon]-1)\in \xi_K\AAKK(X_{\infty,S}).\]
  As $\frac{[\epsilon]-1}{t}\in \BdRp$ goes to $1$ modulo $t$, via the identification 
  \[\calO_C\cdot \xi_K = \calO_C\{1\} = \rho_K^{-1}\calO_C(1) =  \calO_C\cdot \rho_K^{-1}t \text{ (cf. \S\ref{ssec:notation})},\]
  for any $\delta = \prod_{i=0}^d\delta_i^{n_i}\in \Gamma$ and any $0\leq j\leq d$, we have 
  \[c(\delta)(e_j) = n_j\rho_K\xi_K.\]
  In summary, we have proved the following proposition:
  \begin{prop}\label{prop:gamma action on Faltings extension}
      There exists an isomorphism of $\widehat R^+_{\infty,S}$-modules
      \[\calE_{\widetilde \frakX_S}^+(X_{\infty,S})=:E_{\widetilde R_S^+}^+\cong \widehat R_{\infty,S}^+\cdot\xi_K\oplus\left( (\oplus_{i=0}^d\widehat R_{\infty,S}^+\cdot e_i)/\widehat R_{\infty,S}^+\cdot(e_0+\cdots+e_r)\right)\]
      such that via this isomorphism, the $\Gamma$-action on $E_{\widetilde R_S^+}^+$ is given by that for any $\delta = \prod_{i=0}^d\delta_i^{n_i}\in\Gamma$, 
      \[\delta(a\xi_K+\sum_{j=0}^db_je_j) =(\delta(a)+\rho_K\sum_{j=0}^d\delta(b_j)n_j)\xi_K+\sum_{j=0}^d\delta(b_j)e_j.\]
  \end{prop}
  Put $E^+:=\calE^+(X_{\infty,S})$ where $\calE^+$ is the Breuil--Kisin--Fargues twist of the integral Faltings' extension (cf. \eqref{equ:twisted Faltings extension}). Then it admits a continuous action of $\Gamma$.
  \begin{cor}\label{cor:gamma action on Faltings extension}
      \begin{enumerate}
          \item[(1)] The $E^+$ is a free $\widehat R^+_{\infty,S}$-module of rank $d+1$ fitting into the short exact sequence
          \[0\to \widehat R^+_{\infty,S}\xrightarrow{i} E^+\xrightarrow{\pr} \widehat R^+_{\infty,S}\otimes_{R^+_S}\Omega^{1,\log}_{R_S^+}\{-1\}\to 0.\]

          \item[(2)] There exists an isomorphism of $\widehat R^+_{\infty,S}$-modules
          \[E^+\cong \widehat R_{\infty,S}^+\cdot e\oplus\left( (\oplus_{i=0}^d\widehat R_{\infty,S}^+\cdot y_i)/\widehat R_{\infty,S}^+\cdot(y_0+\cdots+y_r)\right)\]
          such that the following statements hold true:
          \begin{enumerate}
              \item The $\widehat R_{\infty,S}^+\cdot e$ is identified with $\widehat R_{\infty,S}^+$ via the injection $i$ above and $e = i(1)$.
              
              \item Via the isomorphism (cf. (\ref{equ:omega^1-II}))
              \[\Omega^{1,\log}_{R_S^+}\{-1\}\cong (\oplus_{i=0}^d R_S^+\cdot \frac{e_i}{\xi_K})/R_S^+\cdot(\frac{e_0}{\xi_K}+\cdots+\frac{e_r}{\xi_K}),\]
             the image of $y_i$ via the projection $\pr$ above is $\pr(y_i) = \frac{e_i}{\xi_K}$ for any $0\leq i\leq d$.

              \item The $\Gamma$-action on $E^+$ is given by that for any $\delta = \prod_{i=0}^d\delta_i^{n_i}\in\Gamma$, 
              \[\delta(ae+\sum_{j=0}^db_jy_j) =(\delta(a)+\rho_K\sum_{j=0}^d\delta(b_j)n_j)e+\sum_{j=0}^d\delta(b_j)y_j.\]
          \end{enumerate}
      \end{enumerate}
  \end{cor}
  \begin{proof}
      The Item (1) follows from Corollary \ref{cor:Faltings' extension}. The Item (2) follows from Proposition \ref{prop:gamma action on Faltings extension} by letting $y_i = \frac{e_i}{\xi_K}$ with $e_i$'s appearing there.
  \end{proof}

\subsection{Period sheaves}\label{ssec:period sheaves}
  Now, we are going to follow the strategy in \cite[\S2]{MW-AIM} to construct the period sheaves with Higgs fields $(\calO\widehat \bC_{\pd,S}^+,\Theta)$ and $(\calO\widehat \bC_{\pd,S},\Theta)$ by using the twisted integral Faltings' extension $\calE^+ = \calE_{\widetilde \frakX_S}^+\{-1\}$ associated to the lifting $\widetilde \frakX_S$ of $\frakX_S$ over $\AAKK(S)$.

  Recall the following well-known lemma due to Quillen and Illusie:
  \begin{lem}[\emph{\cite[Lem. A.28]{SZ}}]\label{lem:SZ}
      Let $B$ be a commutative ring. For any short exact sequence of flat $B$-modules
      \[0\to E\xrightarrow{u} F\xrightarrow{v} G\to 0\]
      and any $n\geq 0$, there exists an exact sequence of $B$-modules:
      \begin{equation}\label{equ:SZ-I}
          0\to\Gamma^n(E)\to\Gamma^n(F)\xrightarrow{\partial}\Gamma^{n-1}(F)\otimes G\xrightarrow{\partial}\cdots\xrightarrow{\partial}\Gamma^{n-i}(F)\otimes\wedge^iG\xrightarrow{\partial}\cdots \xrightarrow{\partial}\wedge^nG\to 0,
      \end{equation}
      where the differentials $\partial$ are induced by sending each 
      \[f_1^{[m_1]}\cdots f_r^{[m_r]}\otimes \omega\in\Gamma^m(F)\otimes\wedge^lG\]
      with $f_i\in F$, $m_i\geq 1$ satisfying $m_1+\cdots+m_r = m$ and $\omega\in\wedge^lG$ to 
      \[\sum_{i=1}^rf_1^{[m_1]}\cdots f_i^{[m_i-1]}\cdots f_r^{[m_r]}\otimes v(f_i)\wedge\omega\in\Gamma^{m-1}(F)\otimes\wedge^{l+1}G.
      \]
      Moreover, there exists an exact sequence
      \begin{equation}\label{equ:SZ-II}
          0\to\Gamma(E)\to\Gamma(F)\xrightarrow{\partial}\Gamma(F)\otimes G\xrightarrow{\partial}\Gamma(F)\otimes\wedge^2G\xrightarrow{\partial}\cdots,
      \end{equation}
      where the differentials $\partial$ are all $\Gamma(E)$-linear.
  \end{lem} 
  Applying the above lemma to the short exact sequence in Corollary \ref{cor:Faltings' extension}, we get an exact sequence
  \begin{equation}\label{equ:apply SZ to Faltings' extension}
      0\to \Gamma(\widehat \calO_{X_S}^+)\to \Gamma(\calE^+)\xrightarrow{\partial}\Gamma(\calE^+)\otimes_{\calO_{\frakX_S}}\Omega^{1,\log}_{\frakX_S}\{-1\}\to\cdots
  \end{equation}
  where $\partial$ is $\Gamma(\widehat \calO^+_{X_S})$-linear. Note that $\Gamma(\calE^+)\cong \widehat \calO^+_{X_S}[e]_{\pd}$ is the free pd-algebra generated by $e$ over $\widehat \calO_{X_S}^+$, where $e$ stands for the basis $1\in \widehat \calO^+_{X_S}$. As $\zeta_p-1$ admits arbitrary pd-powers in $\widehat \calO^+_{X_S}$, the $e-(\zeta_p-1)$ generates a pd-ideal $\calI_{\pd}$ of $\widehat \calO^+_{X_S}[e]_{\pd}$ such that we have $\Gamma(\widehat \calO_{X_S}^+)/\calI_{\pd}\cong \widehat \calO_{X_S}^+$.

  \begin{dfn}\label{dfn:period sheaf}
      Let $\calO\widehat \bC_{\pd,S}^+$ be the $p$-adic completion of 
      \[\calO\bC_{\pd,S}^+:= \Gamma(\calE^+)/\calI_{\pd}\cdot\Gamma(\calE^+)\]
      and let $\Theta:\calO\widehat \bC_{\pd,S}^+\to \calO\widehat \bC_{\pd,S}^+\otimes_{\calO_{\frakX_S}}\Omega^{1,\log}_{\frakX_S}\{-1\}$ be the morphism induced by $\partial$ in (\ref{equ:apply SZ to Faltings' extension}). Define
      \[(\calO\widehat \bC_{\pd,S},\Theta:\calO\widehat \bC_{\pd,S}\to \calO\widehat \bC_{\pd,S}\otimes_{\calO_{\frakX_S}}\Omega^{1,\log}_{\frakX_S}\{-1\}):= (\calO\widehat \bC_{\pd,S}^+,\Theta)[\frac{1}{p}].\]
  \end{dfn}

  \begin{thm}[Poincar\'e's Lemma]\label{thm:Poincare Lemma}
      For any $?\in \{\emptyset,+\}$, the following sequence
      \[0\to \widehat \calO_{X_S}^?\to \calO\widehat \bC_{\pd,S}^?\xrightarrow{\Theta}\calO\widehat \bC_{\pd,S}^?\otimes_{\calO_{\frakX_S}}\Omega^{1,\log}_{\frakX_S}\{-1\}\xrightarrow{\Theta} \calO\widehat \bC_{\pd,S}^?\otimes_{\calO_{\frakX_S}}\Omega^{2,\log}_{\frakX_S}\{-2\}\to\cdots\]
      is exact on $X_v$. In particular, $\Theta$ defines a Higgs field on $\calO\widehat \bC_{\pd,S}^+$.
  \end{thm}
  \begin{proof}
      It suffices to prove the case for $? = +$. By Corollary \ref{cor:Faltings' extension}, locally on $X_v$, the $\calE^+$ is a direct sum of $\widehat \calO_{X_S}^+$ and $\widehat \calO_{X_S}^+\otimes_{\calO_{\frakX_S}}\Omega^{1,\log}_{\frakX_S}\{-1\}$. So locally on $X_v$, the $\Gamma(\calE^+)$ is a free pd-polynomial ring over $\Gamma(\widehat \calO^+_{X_S})$. Modulo $\calI_{\pd}$, by the exactness of (\ref{equ:apply SZ to Faltings' extension}), we have the following exact sequence
      \begin{equation}\label{equ:Poincare's Lemma}
          0\to \widehat \calO_{X_S}^+\to \calO\bC_{\pd,S}^+\xrightarrow{\Theta}\calO\bC_{\pd,S}^+\otimes_{\calO_{\frakX_S}}\Omega^{1,\log}_{\frakX_S}\{-1\}\xrightarrow{\Theta} \calO\bC_{\pd,S}^+\otimes_{\calO_{\frakX_S}}\Omega^{2,\log}_{\frakX_S}\{-2\}\to\cdots
      \end{equation}
      where $\Theta$ denotes the reduction of $\partial$ in (\ref{equ:apply SZ to Faltings' extension}). Moreover, locally on $X_v$, the $\calO\bC_{\pd,S}^+$ is a free pd-polynomial ring and thus faithfully flat over $\widehat \calO_{X_S}^+$. In particular, taking $p$-adic completion preserves the exactness of (\ref{equ:Poincare's Lemma}), yielding the desired exactness in the case $?=+$.
  \end{proof}

  Now, we give the local description of $(\calO\widehat \bC_{\pd,X_S}^+,\Theta)$ in the case $\frakX_S = \Spf(R_S^+)$ is small affine. We first introduce some notations.
  Let 
  \[\widehat R_{\infty,S}^+[Y_0,\dots,Y_r,Y_{r+1},\dots,Y_d]^{\wedge}_{\pd}\]
  be the $p$-complete free pd-algebra over $\widehat R_{\infty,S}^+$ generated by $Y_0,\dots,Y_d$ and let 
  \[\Theta:\widehat R_{\infty,S}^+[Y_0,\dots,Y_r,Y_{r+1},\dots,Y_d]^{\wedge}_{\pd}\to\oplus_{i=0}^d\widehat R_{\infty,S}^+[Y_0,\dots,Y_r,Y_{r+1},\dots,Y_d]^{\wedge}_{\pd}\cdot\frac{e_i}{\xi_K}\]
  be the map sending each $f\in \widehat R_{\infty,S}^+[Y_0,\dots,Y_r,Y_{r+1},\dots,Y_d]^{\wedge}_{\pd}$ to
  \[\Theta(f) = \sum_{i=0}^d\frac{\partial f}{\partial Y_i}\cdot\frac{e_i}{\xi_K}.\]
  Noting that for any $n\geq 0$ and for any $f\in \widehat R_{\infty,S}^+[Y_0,\dots,Y_r,Y_{r+1},\dots,Y_d]^{\wedge}_{\pd}$, we have
  \begin{equation*}
      \begin{split}
          \Theta((Y_0+\cdots+Y_r)^{[n]}f) = (Y_0+\cdots+Y_r)^{[n]}\Theta(f)+(Y_0+\cdots+Y_r)^{[n-1]}f\cdot\sum_{i=0}^r\frac{e_i}{\xi_K}.
      \end{split}
  \end{equation*}
  So we get a well-defined map
  \begin{equation}\label{equ:local Theta}
      \Theta:P_{\infty,S}^+\to P_{\infty,S}^+\otimes_{R_S^+}\Omega^{1,\log}_{R_S^+}\{-1\} = (\oplus_{i=0}^dP_{\infty,S}^+\cdot\frac{e_i}{\xi_K})/P_{\infty,S}^+\cdot(\frac{e_0}{\xi_K}+\cdots+\frac{e_r}{\xi_K})
  \end{equation}
  via the identification $\Omega^{1,\log}_{R_S^+}\{-1\} \cong (\oplus_{i=0}^dR_S^+\cdot\frac{e_i}{\xi_K})/R_S^+\cdot(\frac{e_0}{\xi_K}+\cdots+\frac{e_r}{\xi_K})$ (cf. (\ref{equ:omega^1-II})),
  where 
  \[P_{\infty,S}^+:=\widehat R_{\infty,S}^+[Y_0,\dots,Y_r,Y_{r+1},\dots,Y_d]^{\wedge}_{\pd}/(Y_0+\cdots+Y_r)_{\pd}\]
  denotes the quotient of $\widehat R_{\infty,S}^+[Y_0,\dots,Y_r,Y_{r+1},\dots,Y_d]^{\wedge}_{\pd}$ by the pd-ideal generated by $Y_0+\cdots+Y_r$. Clearly, via the isomorphisms 
  \[\Omega^{1,\log}_{R_S^+}\{-1\} \cong \oplus_{i=1}^dR_S^+\cdot\frac{e_i}{\xi_K}\]
  and 
  \[P_{\infty,S}^+\cong \widehat R_{\infty,S}^+[Y_1,\dots,Y_d]^{\wedge}_{\pd},\]
  the $\Theta$ is exactly 
  \[\Theta = \sum_{i=1}^d\frac{\partial}{\partial Y_i}\otimes\frac{e_i}{\xi_K}.\]
  When context is clear, we also express $\Theta$ as
  \[\Theta = \sum_{i=0}^d\frac{\partial}{\partial Y_i}\otimes\frac{e_i}{\xi_K}:P_{\infty,S}^+\to P_{\infty,S}^+\otimes_{R_S^+}\Omega^{1,\log}_{R_S^+}\{-1\}.\]

  By construction of $\calO\widehat \bC_{\pd,S}^+$, we see that $\calO\widehat \bC_{\pd,S}^+(X_{\infty})$ is the $p$-adic completion of $\Gamma(E^+)/(e-(\zeta_p-1))_{\pd}$, where $E^+$ and $e$ are described in Corollary \ref{cor:gamma action on Faltings extension} and $(e-(\zeta_p-1))_{\pd}$ denotes the pd-ideal generated by $e-(\zeta_p-1)$. Let $y_i$ be elements in $E^+$ described in Corollary \ref{cor:gamma action on Faltings extension} as well. By abuse of notations, we will not distinguish $y_i$ with its image in $\calO\widehat \bC_{\pd,S}^+(X_{\infty})$. 
  \begin{prop}\label{prop:local OC}
      Keep notations above. The morphism of $\widehat R_{\infty,S}^+$-algebras
      \[\widehat R_{\infty,S}^+[Y_0,\dots,Y_r,Y_{r+1},\dots,Y_d]^{\wedge}_{\pd}\to \calO\widehat \bC_{\pd,S}^+(X_{\infty})\]
      sending each $Y_i$ to $y_i$ induces an isomorphism
      \[\iota:P_{\infty,S}^+\to \calO\widehat \bC_{\pd,S}^+(X_{\infty})\]
      compatible with Higgs fields. Via the isomorphism $\iota$, the $\Gamma$-action on $\calO\widehat \bC_{\pd,X_S}^+(X_{\infty})$ is given by that for any $\delta = \prod_{i=0}^d\delta_i^{n_i}\in\Gamma$ and any $0\leq j\leq d$, we have
      \[\delta(Y_j) = Y_j+n_j\rho_K(\zeta_p-1).\]
  \end{prop}
  \begin{proof}
      Consider the isomorphisms 
      \[E^+\cong \widehat R_{\infty,S}^+\cdot e\oplus(\oplus_{i=1}^d\widehat R_{\infty,S}^+\cdot y_i) \text{ and } \Omega^1_{R_S^+}\{-1\}\cong \oplus_{i=1}^dR_S^+\cdot\frac{e_i}{\xi_K}.\]
      Via the projection $\pr:E^+\to \widehat R_{\infty,S}^+\otimes_{R_S^+}\Omega^1_{R_S^+}\{-1\}$ in Corollary \ref{cor:gamma action on Faltings extension}, the image of $y_i$ is $\frac{e_i}{\xi_K}$.
      By Lemma \ref{lem:SZ}, we have $\Gamma(E^+) = \widehat R_{\infty,S}^+[e,y_1,\dots,y_d]_{\pd}$ is the free pd-polynomial ring over $\widehat R_{\infty,S}^+$ generated by $e, y_1,\dots,y_d$ while the differential map $\partial:\Gamma(E^+)\to\Gamma(E^+)\otimes_{R_S^+}\Omega^1_{R_S^+}\{-1\} $ reads
      \[\partial = \sum_{i=1}^d\frac{\partial}{\partial y_i}\otimes\frac{e_i}{\xi}.\]
      Modulo $(e-(\zeta_p-1))_{\pd}$ and after taking $p$-adic completion, we see that
      \[(\calO\widehat \bC_{\pd,X_S}^+(X_{\infty}),\Theta) = (\widehat R_{\infty,S}^+[y_1,\dots,y_d]^{\wedge}_{\pd},\sum_{i=1}^d\frac{\partial}{\partial y_i}\otimes\frac{e_i}{\xi_K}).\]
      Using the isomorphism $(P_{\infty,S}^+,\Theta)\cong (\widehat R_{\infty,S}^+[Y_1,\dots,Y_d]^{\wedge}_{\pd},\sum_{i=1}^d\frac{\partial}{\partial Y_i}\otimes\frac{e_i}{\xi_K})$ mentioned above, we conclude that $\iota$ is an isomorphism. As for any $\delta = \prod_{i=0}^d\delta_i^{n_i}\in\Gamma$ and any $0\leq j\leq d$, we have $\delta(y_j) = y_j+n_j\rho_Ke$ and $e$ is mapped to $\zeta_p-1$ via the map $E^+\to \calO\widehat \bC_{\pd,X_S}^+(X_{\infty})$, the desired $\Gamma$-action follows directly.
  \end{proof}
  Note that for any $?\in\{\emptyset,+\}$, the sheaf $\calO\widehat \bC_{\pd,S}^?$ is an $\widehat \calO_{X_S}^?$-algebra and the $\calO\widehat \bC_{\pd,S}^?\otimes_{\widehat \calO_{X_S}^+}\calO\widehat \bC_{\pd,S}^?$ is endowed with the tensor product Higgs field $\Theta\otimes\id+\id\otimes\Theta$.
  \begin{prop}\label{prop:multiplication on OC}
      For any $?\in\{\emptyset,+\}$, the multiplication map
      \[\calO\widehat \bC_{\pd,S}^?\otimes_{\widehat \calO_{X_S}^+}\calO\widehat \bC_{\pd,S}^?\to \calO\widehat \bC_{\pd,S}^?\]
      on $\calO\widehat \bC_{\pd,S}^?$ is compatible with Higgs fields; that is, for any local sections $f,g\in \calO\widehat \bC_{\pd,S}^?$, we have 
      \[\Theta(fg) = \Theta(f)g+f\Theta(g).\]
  \end{prop}
  \begin{proof}
      Since the problem is local on both $\frakX_{S,\et}$ and $X_{S,v}$, we may assume $\frakX_S = \Spf(R^+_S)$ is small affine and are reduced to showing that for any $f,g\in \calO\widehat \bC_{\pd,S}^+(X_{\infty})$, we have 
      \[\Theta(fg) = \Theta(f)g+f\Theta(g).\]
      But this follows from Proposition \ref{prop:local OC} immediately.
  \end{proof}
  We remark that Proposition \ref{prop:multiplication on OC} amounts to that the Higgs complex $\HIG(\calO\widehat \bC_{\pd,S}^?,\Theta)$ is a commutative differential graded algebra over $\widehat \calO_{X_S}^?$. Using this Theorem \ref{thm:Poincare Lemma} amounts to that the morphism
  \[\widehat \calO_{X_S}^?\to \HIG(\calO\widehat \bC_{\pd,S}^?)\]
  is actually an isomorphism of algebras in the derived category of $\widehat \calO_{X_S}^+$-modules.

  We end this section with the following remarks.

  \begin{rmk}\label{rmk:overconvergent period sheaf}
      Provided $\calE^+$ in Corollary \ref{cor:Faltings' extension}, for any $r\geq 0$, one can define $\calE_r^+$ as the pull-back of $\calE^+$ along the natural inclusion 
      \[p^r\widehat \calO_{X}^+\otimes_{\calO_{\frakX}}\Omega^{1,\log}_{\frakX}\{-1\}\hookrightarrow \widehat \calO_{X}^+\otimes_{\calO_{\frakX}}\Omega^{1,\log}_{\frakX}\{-1\}.\]
      So it fits into the short exact sequence
      \[0\to\widehat \calO_{X}^+\to \calE_r^+\to p^r\widehat \calO_{X}^+\otimes_{\calO_{\frakX_S}}\Omega^{1,\log}_{\frakX}\{-1\}\to 0.\]
      As in \cite[\S2]{Wan23}, one can construct period sheaves $\calO\widehat \bC_{r}^?$ with Higgs fields $\widetilde \Theta$ for $?\in \{\emptyset,+\}$. (Recall that $\calO_C\{-1\} = \rho_K\calO_C(-1)$, the $\calO\widehat \bC_{r}^?$ coincides with $\calO\widehat \bC_{p^r\rho_K}^?$ in \emph{loc.cit.}.) Put
      \[(\calO\widehat \bC^{\dagger,?},\widetilde \Theta):=\colim_{r\to 0^+}(\calO\widehat \bC_r^?,\widetilde \Theta).\]
      The one can show that the Higgs complex $\rD\rR(\calO\widehat \bC^{\dagger},\widetilde \Theta)$ is a resolution of $\OX$ (cf. \cite[Th. 2.28]{Wan23}). Similar to \cite[Consruction 2.9]{MW-AIM}, there exists a natural injection
      \[\iota_{\rm PS}:\calO\widehat \bC_{\pd}^+\to\calO\widehat \bC_0^+\]
      identifying $(\zeta_p-1)\Theta$ with $\widetilde \Theta$. As \cite[Prop. 2.10]{MW-AIM}, the natural inclusion $\calO\widehat \bC^{\dagger}\hookrightarrow\calO\widehat \bC_0$ factors through the image of $\iota_{\rm PS}$, yielding the following inclusions of period sheaves with Higgs fields
      \[(\calO\widehat \bC^{\dagger},\widetilde \Theta)\to(\calO\widehat \bC_{\pd},(\zeta_p-1)\Theta)\to(\calO\widehat \bC_0,\widetilde \Theta).\]
  \end{rmk}

  \begin{rmk}\label{rmk:compatible with smooth case}
      Recall for smooth $\frakX$ over $\calO_C$, the log-structure on $\frakX$ is induced from the canonical log-structure on $\calO_C$. By \cite[Lem. 8.22]{Ols}, all (log)-contangent complexes in \eqref{equ:exact triangle of ringed-topoi} reduces to the usual (p-complete) cotangent complexes as in \cite[Eq.(2-4)]{Wan23}. So all constructions in this section is compatible with those in \cite{Wan23} and \cite{MW-AIM} (ans thus so are our main results in \S\ref{sec:introduction}).
  \end{rmk}
  
\section{A local Simpson correspondence}\label{sec:local Simpson}
  Fix an affinoid perfectoid $S=\Spa(A,A^+)\in\Perfd$.
  In this section, we always assume $R_S^+$ is small semi-stable and keep the notations at the end of \S\ref{ssec:semistable schemes}. We first make some definitions.
  \begin{dfn}\label{dfn:small Gamma-representation}
      Let $B\in\{R_S^+,R_S,\widehat R_{\infty,S}^+,\widehat R_{\infty,S}\}$.
      \begin{enumerate}
        \item[(1)] By a \emph{$\Gamma$-representation of rank $r$ over $B$}, we mean a finite projective $B$-module $M$ of rank $r$ which is endowed with a continuous $\Gamma$-action. 
        
        A $\Gamma$-representation $M^+$ over $R_S^+$ of rank $r$ is called \emph{Hitchin-small} if it admits an $R_S^+$-basis $e_1,\dots,e_d$ such that for any $\delta\in\Gamma$, the matrix of $\delta$ with respect to the given basis is of the form $\exp(-(\zeta_p-1)\rho_KA_{\delta})$ for some ($p$-adically) topologically nilpotent matrix $A_{\delta}\in \Mat_r(R_S^+)$.
        
        By a \emph{Hitchin-small $\Gamma$-representation of rank $r$ over $B$}, we mean a finite free $B$-module $M$ which is endowed with a (continuous) $\Gamma$-action which is of the form $M = M^+\otimes_{R_S^+}B$ for some Hitchin-small $\Gamma$-representation $M^+$ over rank $r$ over $R_S^+$.

        \item[(2)] By a \emph{Higgs module of rank $r$ over $R_S^+$}, we mean a pair 
        \[(H^+,\theta:H^+\to H^+\otimes_{R_S^+}\Omega^{1,\log}_{R_S^+}\{-1\})\]
        consisting of a finite free $R_S^+$-module $H^+$ of rank $r$ and an Higgs field $\theta$, i.e., an $R_S^+$-linear morphism $\theta$ satisfying $\theta\wedge\theta=0$. For any Higgs module $(H^?,\theta)$, we denote by $\rD\rR(H^?,\theta)$ the induced Higgs complex.
        
        A Higgs module $(H^+,\theta)$ over $R_S^+$ is called 
        \begin{enumerate}
            \item \emph{twisted Hitchin-small} if the Higgs field $\theta$ is topologically nilpotent;

            \item \emph{Hitchin-small} if it is of the form 
            \[(H^+,\theta) = (H^+,(\zeta_p-1)\theta^{\prime})\]
            for some twisted Hitchin-small Higgs module $(H^+,\theta^{\prime})$.
        \end{enumerate}

        By a \emph{(twisted) Hitchin-small Higgs module of rank $r$ over $R_S$}, we mean a pair 
        \[(H,\theta:H\to H\otimes_{R_S^+}\Omega^{1,\log}_{R_S^+}\{-1\})\]
        consisting of a finite free $R_S$-module $H$ of rank $r$ and an Higgs field $\theta$, which is of the form $(H,\theta) = (H^+[\frac{1}{p}],\theta)$ for some (twisted) Hitchin-small Higgs module $(H^+,\theta)$ of rank $r$ over $R_S^+$.

        \item[(3)] For any $?\in\{\emptyset, +\}$, we denote by $\Rep^{\Hsmall}_{\Gamma}(B^?)$ the category of Hitchin-small $\Gamma$-representations over $B^?$, and by $\HIG^{\text{(t-)}\Hsmall}(R_S^?)$ the category of (twisted) Hitchin-small Higgs modules over $R_S^?$.
      \end{enumerate}
  \end{dfn}
  Roughly, the purpose of this section is to establish the equivalences of categories
  \[\Rep^{\Hsmall}_{\Gamma}(\widehat R_{\infty,S}^?)\simeq \Rep^{\Hsmall}_{\Gamma}(R_{S}^?)\simeq \HIG^{\tHsmall}(R_S^?)\simeq \HIG^{\Hsmall}(R_S^?).\]
  Clearly, it suffices to deal with the case for $? = +$, which will be handled with in the next three subsections.

\subsection{$\Gamma$-representations over $R_S^+$ v.s. $\Gamma$-representations over $\widehat R_{\infty,S}^+$}\label{ssec:decompletion theory}
  Note that the base-change $M^+\mapsto M_{\infty}^+:=M^+\otimes_{R_S^+}\widehat R_{\infty,S}^+$ induces a well-defined functor
  \[\Rep_{\Gamma}^{\Hsmall}(R_S^+)\to \Rep_{\Gamma}^{\Hsmall}(\widehat R_{\infty,S}^+).\]
  We now show this functor is exactly an equivalence of categories. More precisely, we shall prove the following result:
  \begin{prop}\label{prop:decompletion}
      The base-change along $R_S^+\to \widehat R_{\infty,S}^+$ induces an equivalence of categories
      \[\Rep_{\Gamma}^{\Hsmall}(R_S^+)\xrightarrow{\simeq} \Rep_{\Gamma}^{\Hsmall}(\widehat R_{\infty,S}^+)\]
      such that for any $M^+\in\Rep_{\Gamma}^{\Hsmall}(R_S^+)$ with the induced $M_{\infty}^+\in\Rep_{\Gamma}^{\Hsmall}(\widehat R_{\infty,S}^+)$, the natural morphism
      \[\rR\Gamma(\Gamma,M^+)\to \rR\Gamma(\Gamma,M_{\infty}^+)\]
      identifies the former with a direct summand of the latter whose complementary is concentrated in degree $\geq 1$ and killed by $\zeta_p-1$.
  \end{prop}
  We will prove Proposition \ref{prop:decompletion} later.
  
  Recall \cite[Lem. 7.3]{BMS18} that for any $p$-complete $\Zp$-module $N$ equipped with a continuous $\Gamma$-action, $\rR\Gamma(\Gamma,N)$ can be calculated by the Koszul complex 
  \[\rK(\gamma_1-1,\dots,\gamma_d-1;N):~N\xrightarrow{\gamma_1-1,\dots,\gamma_d-1} N^d\to\cdots,\]
  where $\gamma_i\in \Gamma$ is defined in (\ref{equ:Gamma group-II}).
  Recall we have the $\Gamma$-equivariant decomposition (\ref{equ:Gamma decomposition})
  \begin{equation*}
      \widehat R_{\infty,S}^+ = \widehat \bigoplus_{\underline \alpha\in J_r}R^+_S\cdot \underline T^{\underline \alpha} = R_S^+\oplus\widehat \bigoplus_{\underline \alpha\in J_r\setminus\{0\}}R^+_S\cdot \underline T^{\underline \alpha},
  \end{equation*}
  where $J_r$ is the set of indices defined in (\ref{equ:index set}). 
  %We equip $\widehat R_{\infty,S}$ with the supreme norm $|\cdot|$ corresponding to the spectral norm on $R_S$ with repsect to the above decomposition. Then for any $f\in \widehat R_{\infty,S}$, we have $f\in \widehat R_{\infty,S}^+\Leftrightarrow |f|\leq 1$. To simplify the notations, throughout this section, we put $\pi:=\zeta_p-1$, and then $|\pi| = p^{-\frac{1}{p-1}}$.
  \begin{lem}\label{lem:fully faithful}
      Let $M^+$ be a finite free $\Gamma$-representation over $R_S^+$ of rank $r$ such that it admits an $R_S^+$-basis $e_1,\dots,e_r$ such that for any $\delta\in \Gamma$, its action on $M^+$ is given by the matrix $\exp(-(\zeta_p-1)A_{\delta})$ for some topologically nilpotent $A_{\delta}\in \Mat_r(R_S^+)$. Then the $\rR\Gamma(\Gamma,M^+\otimes\widehat \bigoplus_{\underline \alpha\in J_r\setminus\{0\}}R^+_S\cdot \underline T^{\underline \alpha})$ is concentrated in degree $\geq 1$ and killed by $\zeta_p-1$.
  \end{lem}
  \begin{proof}
      We claim that for any $\underline \alpha = (\alpha_0,\dots,\alpha_d)\in J_r\setminus\{0\}$, the $\rR\Gamma(\Gamma,M^+\cdot \underline T^{\underline \alpha})$ is concentrated in degree $\geq 1$ and killed by $\zeta_p-1$. Granting this, we can conclude by noting that
      \[\rR\Gamma(\Gamma,M^+\otimes_{R_S^+}(\widehat \bigoplus_{\underline \alpha\in J_r\setminus\{0\}}R_S^+\cdot \underline T^{\underline \alpha})) = \bigoplus_{\underline \alpha\in J_r\setminus\{0\}}\rR\Gamma(\Gamma,M^+\cdot \underline T^{\underline \alpha})\]
      because the right hand side above is killed by $\zeta_p-1$ and thus already $p$-complete.

      It remains to prove the claim. Without loss of generality, we may assume $\alpha_0 = 0$. Then $\alpha_1,\dots,\alpha_d$ are not all zero. Without loss of generality, we may assume $\alpha_1\neq 1$. By Hochschild--Serre spectral sequence, it suffices to show $\rR\Gamma(\Zp\cdot\gamma_1,M^+\cdot \underline T^{\underline \alpha})$ is concentrated in degree $\geq 1$ and killed by $\zeta_p-1$. We now check this by working with the Koszul complex
      \[\rK(\gamma_1-1;M^+\cdot \underline T^{\underline \alpha}): M^+\cdot \underline T^{\underline \alpha}\xrightarrow{\gamma_1-1}M^+\cdot \underline T^{\underline \alpha}.\]
      Put $\theta_i:=A_{\gamma_i}$ for any $1\leq i\leq d$ and 
      \begin{equation}\label{equ:F(X)}
          F(\theta):=\frac{1-\exp(-(\zeta_p-1)\theta)}{(\zeta_p-1)\theta}:= 1+\sum_{n\geq 1}(-1)^n(\zeta_p-1)^{[n]}\theta^n \in \calO_C[[\theta]].
      \end{equation}
      Then we have $\theta_i\in \Mat_r(R_S^+)$ which is topologically nilpotent and thus $F(\theta_i)$ is a well-defined matrix in $\GL_r(R_S^+)$ such that $\gamma_1$ acts on $M^+$ via $1-(\zeta_p-1)\theta_iF(\theta_i)$. Thus, for any $m\in M^+$, we have 
      \[(\gamma_1-1)(m\underline T^{\underline \alpha}) = \big(\zeta^{\alpha_1}\left(1-(\zeta_p-1)\theta_iF(\theta_i)\right)-1\big)(m\underline T^{\underline \alpha}) = (\zeta^{\alpha_1}-1)\left(1-\zeta^{\alpha_1}\frac{\zeta_p-1}{\zeta^{\alpha_1}-1}\theta_1F(\theta_1)\right)(m\underline T^{\underline \alpha}).\]
      As $\alpha_1\neq 0$, we have $\frac{\zeta_p-1}{\zeta^{\alpha_1}-1}\in \calO_C$ and thus $1-\zeta^{\alpha_1}\frac{\zeta_p-1}{\zeta^{\alpha_1}-1}\theta_1F(\theta_1)\in \GL_r(R_S^+)$ because $\theta_1$ is topologically nilpotent, yielding that
      \begin{equation*}
          \rH^n(\Zp\cdot\gamma_1,M\cdot\underline T^{\alpha}) = \left\{
            \begin{array}{rcl}
                0, & n=0 \\
                (M\cdot\underline T^{\alpha})/(\zeta_p-1), & n=1
            \end{array}
          \right.
      \end{equation*}
      as desired. This completes the proof.
  \end{proof}

  \begin{proof}[\textbf{Proof of Proposition \ref{prop:decompletion}}]
      As the essential surjectivity is a part of definition, it suffices to show that full faithfulness of the base-change functor. As the functor preserves tensor products and dualitis, it suffices to show the expected cohomological comparison. That is, we have to show that for any $M^+\in\Rep_{\Gamma}^{\Hsmall}(R_S^+)$, the natural map
      \[\rR\Gamma(\Gamma, M^+)\to \rR\Gamma(\Gamma, M^+\otimes_{R_S^+}\widehat R_{\infty,S}^+)\]
      identifies the former as a direct summand of the latter whose complementary is concentrated in degree $\geq 1$ and killed by $\zeta_p-1$. But this follows from Lemma \ref{lem:fully faithful} immediately.
  \end{proof}

\subsection{Hitchin-small Higgs modules v.s. Hitchin-small twisted Higgs modules}  
  Clearly, for any $?\in \{\emptyset,+\}$ the rule $(H^+,\theta)\mapsto (H^+,(\zeta_p-1)\theta)$ defines a well-defined functor
  \[\HIG^{\tHsmall}(R_S^?)\to \HIG^{\Hsmall}(R_S^?).\]
  One can prove this functor is indeed an equivalence of categories.
  \begin{prop}\label{prop:twist functor}
      The twist functor $(H^+,\theta)\mapsto (H^+,(\zeta_p-1)\theta)$ induces an equivalence of categories
      \[\HIG^{\tHsmall}(R_S^+)\to \HIG^{\Hsmall}(R_S^+)\]
      such that for any $(H^+,\theta)\in \HIG^{\tHsmall}(R_S^+)$ with induced $(H^+,\theta^{\prime})\in \HIG^{\Hsmall}(R_S^+)$, there exists a quasi-isomorphism
      \[\rL\eta_{\zeta_p-1}\rD\rR(H^+,\theta^{\prime})\simeq \rD\rR(H^+,\theta).\]
      where $\rL\eta_{\zeta_p-1}$ denotes the d\'ecalage functor in \cite[\S6]{BMS18}. 
  \end{prop}
  \begin{proof}
      Via the isomorphism \eqref{equ:omega^1-II}, one can write
      \[\theta = \sum_{i=1}^d\theta_i\otimes\frac{e_i}{\xi_K} \text{ and }\theta_i^{\prime} = \sum_{i=1}^d\theta_i^{\prime}\otimes\frac{e_i}{\xi_K}.\]
      Using this, we have $\theta_i^{\prime} = (\zeta_p-1)\theta_i$ for all $i$. Note that $\rD\rR(H^+,\theta^{\prime})$ and $\rD\rR(H^+,\theta)$ can be computed by the Koszul complexes
      \[\rK(\theta_1^{\prime},\dots,\theta_d^{\prime};H^+) \text{ and }\rK(\theta_1,\dots,\theta_d;H^+),\]
      respectively. By \cite[Lem. 7.9]{BMS18}, we have a quasi-isomorphism
      \[\eta_{\zeta_p-1}\rK(\theta_1^{\prime},\dots,\theta_d^{\prime};H^+)\simeq\rK(\theta_1,\dots,\theta_d;H^+),\]
      yielding the desired quasi-isomorphism
      \[\rL\eta_{\zeta_p-1}\rD\rR(H^+,\theta^{\prime})\simeq \rD\rR(H^+,\theta).\]
      In particular, taking $\rH^0$, we have 
      \[\rH^0(\rD\rR(H^+,\theta))\cong\rH^0(\rL\eta_{\zeta_p-1}\rD\rR(H^+,\theta^{\prime}))\cong\rH^0(\rD\rR(H^+,\theta^{\prime})\]
      where the last isomorphism follows from \cite[Lem. 6.4]{BMS18} because $\rH^0(\rD\rR(H^+,\theta^{\prime})$ is a sub-$R_S^+$-module of $H^+$ and thus $(\zeta_p-1)$-torsion free. As the twist functor preserves tensor product and dualities, the above cohomological comparison implies the full faithfulness of the twist functor. One can conclude as the essential surjectivity is a part of the definition.
  \end{proof}

\subsection{Local Simpson correspondence}\label{ssec:local Simpson correspondence}
  Let $(P_{\infty,S}^+,\Theta = \sum_{i=0}^d\frac{\partial}{\partial Y_i}\otimes\frac{e_i}{\xi_K})$ be as in Proposition \ref{prop:local OC}. Let $P_S^+ = R_S^+[Y_0,\dots,Y_d]^{\wedge}_{\pd}/(Y_0+\cdots+Y_r)_{\pd}$ be the quotient of $R_S^+[Y_0,\dots,Y_d]^{\wedge}_{\pd}$, the free pd-polynomial ring over $R_S^+$ generated by $Y_0,\dots,Y_d$, by the closed pd-ideal generated by $Y_0+\cdots+Y_r$. Then $P_S^+$ is a sub-$R_S^+$-algebra of $P_{\infty,S}^+$ which is $\Theta$-preserving and stable under the action of $\Gamma$ such that via the decomposition \eqref{equ:Gamma decomposition}, there is a $\Gamma$-equivariant isomorphism
  \begin{equation}\label{equ:decomposition for P}
      P_{S,\infty}^+ = P_S^+\widehat \otimes_{R_S^+}\widehat R_{\infty,S}^+ = \widehat \bigoplus_{\underline \alpha\in J_r}P_S\cdot\underline T^{\underline \alpha}.
  \end{equation}
  
  Now, we are going to prove the following local Simpson correspondence.
  \begin{thm}[Local Simpson correspondence]\label{thm:local Simpson}
    Let $?\in\{\emptyset,\infty\}$
      \begin{enumerate}
          \item[(1)] For any $M_?^+\in \Rep_{\Gamma}^{\Hsmall}(\widehat R_{?,S}^+)$ of rank $r$, define 
          \[\Theta_{M_?^+}=\id\otimes\Theta:M_?^+\otimes_{\widehat R_{?,S}^+}P_{?,S}^+\to M_?^+\otimes_{\widehat R_{?,S}^+}P_{?,S}^+\otimes_{R_S^+}\Omega^{1,\log}_{R_S^+}\{-1\}.\]
          Then we have
          \begin{equation*}
              \rH^n(\Gamma,M_?^+\otimes_{\widehat R_{?,S}^+}P_{?,S}^+) = \left\{
                \begin{array}{rcl}
                    H^+(M_?^+), & n=0 \\
                    (\zeta_p-1)\rho_K\text{-torsion}, & n\geq 1
                \end{array}
              \right.
          \end{equation*}
          where $H^+(M_?^+)$ is a finite free $R_S^+$-module of rank $r$ such that the restriction of $\Theta_{M_?^+}$ to $H^+(M_?^+)$ induces a Higgs field $\theta$ making $(H^+(M_?^+),\theta)\in \HIG^{\tHsmall}(R_S^+)$.

          \item[(2)] For any $(H^+,\theta)\in\HIG^{\tHsmall}(R_S^+)$ of rank $r$, define
          \[\Theta_{H^+}:=\theta\otimes\id+\id\otimes\Theta: H\otimes_{R_S^+}P_{?,S}^+\to H^+\otimes_{\widehat R_{?,S}^+}P_{?,S}^+\otimes_{R_S^+}\Omega^{1,\log}_{R_S^+}\{-1\}.\]
          Then the $M_?^+(H^+,\theta):= (H\otimes_{R_S^+}P_{?,S}^+)^{\Theta_{H^+} = 0}$ together with the induced $\Gamma$-action from $P_{?,S}^+$ is a well-defined object in $\Rep_{\Gamma}^{\Hsmall}(\widehat R_{?,S}^+)$ of rank $r$.

          \item[(3)] The functors $M_?^+\mapsto (H^+(M_?^+),\theta)$ and $(H^+,\theta)\mapsto M_?^+(H^+,\theta)$ above define an equivalence of categories
          \[\Rep_{\Gamma}^{\Hsmall}(\widehat R_{?,S}^+)\simeq\HIG^{\tHsmall}(R_S^+)\]
          such that for any $M_?^+\in\Rep_{\Gamma}^{\Hsmall}(\widehat R_{?,S}^+)$ with the induced $(H^+,\theta)$, there exists an isomorphism of Higgs modules
          \begin{equation}\label{equ:isomorphism of Higgs modules}
              (M_?^+\otimes_{\widehat R_{?,S}^+}P_{?,S}^+,\Theta_{M_?^+})\simeq (H^+\otimes_{\widehat R_{?,S}^+}P_{?,S}^+,\Theta_{H^+})
          \end{equation}
          and a quasi-isomorphism
          \begin{equation}\label{equ:quasi-isomorphism inverting p}
              \rR\Gamma(\Gamma,M_?^+[\frac{1}{p}])\simeq \rD\rR(H^+[\frac{1}{p}],\theta).
          \end{equation}

          \item[(4)] The following diagram is commutative
          \begin{equation}
              \xymatrix@C=0.5cm{
                \Rep_{\Gamma}^{\Hsmall}(R_S^+)\ar[rr]^{-\otimes_{R_S^+}\widehat R_{\infty,S}^+}&&\Rep_{\Gamma}^{\Hsmall}(\widehat R_{\infty,S}^+)\\
                &\ar[lu]^{M^+}\HIG^{\tHsmall}(R_S^+).\ar[ru]_{M_{\infty}^+}
              }
          \end{equation}
      \end{enumerate}
  \end{thm}
  Before we establish the above local Simpson correspondence, let us do some preparations.

  \begin{construction}\label{construction:MW}
     Let $B$ be a $p$-complete $p$-torsion free $\calO_C$-algebra and $G = \Zp\cdot\gamma$ such that $G$ acts on $B$ trivially. Let $B[Y]^{\wedge}_{\pd}$ be the free pd-polynomial ring over $B$ generated by $Y$ which is equipped with an action of $G$ such that $\gamma(Y) = Y+(\zeta_p-1)\rho_K$.
     For any $\alpha\in \bN[\frac{1}{p}]\cap [0,1)$, let $B\cdot e_{\alpha}$ be the $1$-dimensional representation of $G$ over $B$ with the basis $e_{\alpha}$ such that $\gamma(e_{\alpha}) = \zeta^{\alpha}e_{\alpha}$. For any finite free $B$-module $V$ of rank $r$ together with a fixed $B$-basis $e_1,\dots,e_r$ which is endowed with a $G$-action satisfying the condition:
     \begin{enumerate}
         \item[$(\ast)$] With respect to the given basis, the action of $\gamma$ on $V$ is given by the matrix $\exp(-(\zeta_p-1)\rho_K\theta)$ for some topologically nilpotent matrix $\theta\in \Mat_r(B)$.
     \end{enumerate}
      we define 
      \[M_{\alpha}(V):=V\otimes_{B}B\cdot e_{\alpha}\otimes_{B}B[Y]^{\wedge}_{\pd}\]
      and equip it with the diagonal $G$-action.
  \end{construction}

  We remark that the $G$-representation $V$ satisfying the condition $(\ast)$ is exactly the \emph{log-nilpotent} $B$-representation of $G$ in the sense of \cite[Def. 3.2]{MW-AIM}.
  The following lemma plays the key role in this section.

  \begin{lem}\label{lem:MW}
      Keep notations in Construction \ref{construction:MW}. 
      \begin{enumerate}
          \item[(1)] Suppose that $\alpha\neq 0$. We have
          \begin{equation*}
              \rH^n(G,M_{\alpha}(V)) = \left\{
                \begin{array}{rcl}
                    0, & n=0 \\
                    M_{\alpha}(V)/(\zeta^{\alpha}-1), & n=1.
                \end{array}
              \right.
          \end{equation*}

          \item[(2)] Suppose that $\alpha = 0$. We have
          \begin{equation*}
              \rH^n(G,M_{0}(V)) = \left\{
                \begin{array}{rcl}
                    \exp(\theta Y)(V), & n=0 \\
                    M_{\alpha}(V)/\rho_K(\zeta^{\alpha}-1), & n=1,
                \end{array}
              \right.
          \end{equation*}
          where 
          \[\exp(\theta Y)(V):=\{\sum_{i\geq 0}\theta^i(v)Y^{[i]}\in M_0(V)\mid v\in V\}.\]
          Moreover, the natural inclusion $M_0(V)^{G}\to M_0(V)$ induces a $G$-equivariant isomorphism 
          \[M_0(V)^G\otimes_BB[Y]^{\wedge}_{\pd}\xrightarrow{\cong}M_0(V).\]
      \end{enumerate}
  \end{lem}
  \begin{proof}
      Comparing Construction \ref{construction:MW} with \cite[Not. 3.1 and Def. 3.2]{MW-AIM}, one can conclude
      by applying \cite[Prop. 3.4]{MW-AIM}.
  \end{proof}
  \begin{cor}\label{cor:MW}
      Keep notations in Construction \ref{construction:MW}. 
      Let $G^d = \oplus_{i=1}^d\Zp\cdot\gamma_i$ which acts on $B$ trivially and $B[Y_1,\dots,Y_d]^{\wedge}_{\pd}$ be the $p$-complete free pd-polynomial ring over $B$ generated by $Y_1,\dots, Y_d$ with the a $G^d$-action such that $\gamma_i(Y_j) = Y_j+\delta_{ij}(\zeta_p-1)\rho_K$ for any $1\leq i,j\leq d$. 
      Let $V$ be a finite free $B$-module of rank $r$ which is equipped with a $G$-action such that for some $B$-basis $e_1,\dots,e_r$ of $V$ and for any $1\leq i\leq d$, the action of $\gamma_i$ on $V$ satisfies the condition $(\ast)$ in Construction \ref{construction:MW} for some topologically nilpotent matrix $\theta_i\in \Mat_r(B)$.
      For any $\underline \alpha = (\alpha_1,\dots,\alpha_d)\in (\bN[\frac{1}{p}]\cap[0,1))^d$, let $B\cdot e_{\underline \alpha}$ be the $1$-dimensional representation of $G^d$ over $B$ such that $\gamma_i(e_{\underline \alpha}) = \zeta^{\alpha_i}e_{\underline \alpha}$ for any $1\leq i\leq d$, and define 
      \[M_{\underline \alpha}(V):=V\otimes_BB\cdot e_{\underline \alpha}\otimes_BB[Y_1,\dots,Y_d]^{\wedge}_{\pd}.\]
      Then the following assertions are true:
      \begin{enumerate}
          \item[(1)] Suppose that $\underline \alpha\neq 0$. Then we have
          \begin{equation*}
              \rH^n(G^d,M_{\underline \alpha}(V)) = \left\{
                \begin{array}{rcl}
                    0, & n=0 \\
                    (\zeta_p-1)\text{-torison}, & n\geq 1.
                \end{array}
              \right.
          \end{equation*}

          \item[(2)] Suppose that $\underline \alpha = 0$. Then we have
          \begin{equation*}
              \rH^n(G^d,M_{0}(V)) = \left\{
                \begin{array}{rcl}
                    \exp(\sum_{i=1}^d\theta_iY_i)(V), & n=0 \\
                    (\zeta_p-1)\rho_K\text{-torison}, & n\geq 1,
                \end{array}
              \right.
          \end{equation*}
          where 
          \[\exp(\sum_{i=1}^d\theta_iY_i)(V):=\{\sum_{J\in\bN^d}\underline \theta^J(v)\underline Y^{[J]}\mid v\in V\}.\]
          Moreover, the natural inclusion $M_0(V)^{G}\to M_0(V)$ induces a $G$-equivariant isomorphism 
          \[M_0(V)^G\otimes_BB[Y_1,\dots,Y_d]^{\wedge}_{\pd}\xrightarrow{\cong}M_0(V).\]
      \end{enumerate}
  \end{cor}
  \begin{proof}
      Note that as $G^d$ is commutative, the $\theta_i$'s commute with each others. In particular, the 
      \[\exp(\sum_{i=1}^d\theta_iY_i):=\sum_{J\in\bN^d}\underline \theta^J\underline Y^{[J]}\]
      is a well-defined matrix in $\GL_r(B[Y_1,\dots,Y_d]^{\wedge}_{\pd})$. So $\exp(\rho_K\sum_{i=1}^d\theta_iY_i)(V)$ is also well-defined and is a finite free $R_S^+$-module of rank $r$.

      For Item (1): Without loss of generality, we may assume $\alpha_d\neq 0$. By Serre--Hochschild spectral sequence, it suffices to show $\rH^n(\Zp\cdot\gamma_d,M_{\underline \alpha}(V))$ is killed by $\zeta_p-1$ for $n=1$ and vanishes for $n=0$. But this follows from Lemma \ref{lem:MW}(1) (by working with $B[Y_1,\cdots,Y_{d-1}]^{\wedge}_{\pd}$ instead of $B$ there).

      For Item (2): Using the same spectral sequence argument as in Item (1), one can deduce from Lemma \ref{lem:MW}(2) that for any $n\geq 1$, the $\rH^n(G^d,M_{0}(V))$ is killed by $(\zeta_p-1)\rho_K$. Using Lemma \ref{lem:MW}(2) again, by iteration, we have
      \[\begin{split}
        \rH^0(G^d,M_0(V)) &= M_0(V)^{\gamma_1=\cdots=\gamma_d = 1} \\
        &= \left(\exp(\theta_dY_d)(V\otimes_{B}B[Y_1,\dots,Y_{d-1}]^{\wedge}_{\pd})\right)^{\gamma_1=\cdots=\gamma_{d-1}=1}\\
        &=\cdots\\
        &=\exp(\sum_{i=1}^d\theta_iY_i)(V).
      \end{split}\]
      Now, the final isomorphism 
      \[M_0(V)^G\otimes_BB[Y_1,\dots,Y_d]^{\wedge}_{\pd}\xrightarrow{\cong}M_0(V)\]
      follows as $\exp(\sum_{i=1}^d\theta_iY_i)\in\GL_r(B[Y_1,\dots,Y_d]^{\wedge}_{\pd})$.
  \end{proof}
  Now, we are able to establish the local Simpson correspondence.
  \begin{proof}[\textbf{Proof of Theorem \ref{thm:local Simpson}}]
      We first prove Items (1), (2) and (3) for $? = \emptyset$.

      For Item (1):  Fix an $M^+\in \Rep_{\Gamma}^{\Hsmall}(R_S^+)$ of rank $r$ such that via the isomorphism (\ref{equ:Gamma group-II}), the $\gamma_i$-action on $M^+$ is given by $\exp(-(\zeta_p-1)\rho_K\theta_i)$ for some topologically nilpotent $\theta_i\in\End_{R_S^+}(M^+)$ for any $1\leq i\leq d$. Consider the isomorphisms
      \begin{equation}\label{equ:local Simpson-I}
          \Omega^{1,\log}_{R_S^+}\cong \oplus_{i=1}^dR_S^+\cdot e_i\text{ and }(P_S^+,\Theta)\cong (R_S^+[Y_1,\dots,Y_d]^{\wedge}_{\pd},\sum_{i=1}^d\frac{\partial}{\partial Y_i}\otimes\frac{e_i}{\xi_K}).
      \end{equation}
      It follows from Corollary \ref{cor:MW} that $\rH^n(\Gamma, M^+\otimes_{R_S^+}P_S^+)$ is killed by $\rho_K(\zeta_p-1)$ for $n\geq 1$ and gives rise to a finie free $R_S^+$-module
      \begin{equation}\label{equ:local Simpson-II}
          H^+(M^+)= \exp(\sum_{i=1}^d\theta_iY_i)(M^+).
      \end{equation}
      Then the restriction of $\Theta_{M^+}$ to $H^+(M^+)$ is given by
      \begin{equation}\label{equ:local Simpson-III}
          \theta = \sum_{i=1}^d\theta_i\otimes\frac{e_i}{\xi_K}:H^+(M^+)\to H^+(M^+)\otimes_{R^+_S}\Omega^{1,\log}_{R_S^+}\{-1\}.
      \end{equation}
      In particular, we have $\theta$ is topologically nilpotent as each $\theta_i$ does. This completes the proof for Item (1).

      For Item (2): Again we use the isomorphisms (\ref{equ:local Simpson-I}). Fix an $(H^+,\theta)\in\HIG^{\tHsmall}(R_S^+)$ of rank $r$ and write $\theta = \sum_{i=1}^d\theta_i\otimes\frac{e_i}{\xi_K}$ with $\theta_i\in\End_{R_S^+}(H^+)$ topologically nilpotent for all $i$. Fix an 
      \[x = \sum_{J\in\bN^d}h_J\underline Y^{J}\in H^+\otimes_{R_S^+}P_S^+.\]
      Then we have 
      \[\Theta_{H^+}(x) = \sum_{i=1}^d\left(\sum_{J\in\bN^d}(\theta_i(x_J)+x_{J+E_i})\underline Y^J\right)\otimes\frac{e_i}{\xi_K}.\]
      In particular, $\Theta_{H^+}(x) = 0$ if and only if for any $1\leq i\leq d$, we have $x_{J+E_i} = -\theta_i(x_J)$. By iteration, this amounts to that for any $J\in\bN^d$,
      \[x_J = (-1)^{|J|}\underline \theta^J(x_0),\]
      yielding that 
      \begin{equation}\label{equ:local Simpson-IV}
          M^+(H^+,\theta) = \exp(-\sum_{i=1}^d\theta_iY_i)(H^+):=\{\sum_{J\in\bN^d}(-1)^{|J|}\underline \theta^J(h)\underline Y^{[J]}\mid h\in H^+\}.
      \end{equation}
      As $\gamma_i(Y_j) = Y_j+(\zeta_p-1)\rho_K$ via the isomorphism (\ref{equ:Gamma group-II}) by Corollary \ref{cor:gamma action on Faltings extension}, we see that the $\gamma_i$-action on $M^+(H^+,\theta)$ is given by 
      \begin{equation}\label{equ:local Simpson-V}
          \gamma_i = \exp(-\rho_K(\zeta_p-1)\theta_i).
      \end{equation}
      This forces $M^+(H^+,\theta)\in \Rep^{\Hsmall}_{\Gamma}(R_S^+)$ as desired.

      For Item (3): For any $M^+\in\Rep^{\Hsmall}_{\Gamma}(R_S^+)$, it follows from the ``moreover'' part of Corollary \ref{cor:MW} that the natural morphism $H^+(M^+)\to M^+\otimes_{R_S^+}P_S^+$ induces a $\Gamma$-equivariant isomorphism
      \[H^+(M^+)\otimes_{R_S^+}P_S^+\xrightarrow{\cong}M^+\otimes_{R_S^+}P_S^+.\]
      It follows from the construction of $\theta$ on $H^+(M^+)$ that the above isomorphism is compatible with Higgs fields. By Poincar\'e's Lemma \ref{thm:Poincare Lemma}, we have quasi-isomorphisms
      \[\begin{split}
          \rR\Gamma(\Gamma,M^+[\frac{1}{p}])&\xrightarrow{\simeq}\rR\Gamma(\Gamma,\rD\rR(M^+\otimes_{R_S^+}P_S^+[\frac{1}{p}]),\Theta_{M^+})\\
          &\xleftarrow{\simeq}\rR\Gamma(\Gamma,\rD\rR(H^+\otimes_{R_S^+}P_S^+[\frac{1}{p}]),\Theta_{H^+})\\
          &\xleftarrow{\simeq}\rD\rR(H^+[\frac{1}{p}],\theta),
      \end{split}\]
      where the last quasi-isomorphism follows from $\rR\Gamma(\Gamma,P_S^+[\frac{1}{p}]) = 0$, by Item (1).
      Thus, one only need to show the functors in Items (1) and (2) are quasi-inverses of each other. We divide the proof into three steps:

      \textbf{Step 1:} Fix an $M^+\in \Rep_{\Gamma}^{\Hsmall}(R_S^+)$ and put $(H^+,\theta) = (H^+(M^+),\theta)$. We denote by
      \[\iota_{M^+(H^+,\theta)}:M^+(H^+(M^+),\theta) = M^+(H^+,\theta)\to M^+\]
      the natural morphism induced by the composites
      \begin{equation*}
          \begin{split}
              M^+(H^+,\theta) & = (H^+\otimes_{R_S^+}P_S^+)^{\Theta_{H^+} = \theta\otimes\id+\id\otimes\Theta = 0}\\
              & = \left((M^+\otimes_{R_S^+}P_S^+)^{\Gamma}\otimes_{R_S^+}P_S^+\right)^{\Theta_{M^+}\otimes\id+(\id\otimes\id)\otimes\Theta = 0}\\
              &\hookrightarrow (M^+\otimes_{R_S^+}P_S^+\otimes_{R_S^+}P_S^+)^{\id\otimes\Theta\otimes\id+\id\otimes\id\otimes\Theta = 0}\\
              & \to (M^+\otimes_{R_S^+}P_S^+)^{\id\otimes\Theta = 0}\\
              & = M^+,
          \end{split}
      \end{equation*}
      where the last equality follows from the Poincar\'e's Lemma \ref{thm:Poincare Lemma} while the last arrow is induced by the multiplication on $P_S^+$ (cf. Proposition \ref{prop:multiplication on OC}).

      \textbf{Step 2:} Fix an $(H^+,\theta)\in \HIG^{\tHsmall}(R_S^+)$ and put $M^+=M^+(H^+,\theta)$. We denote by
      \[\iota_{(H^+,\theta)}:H^+(M^+)\to H^+\]
      the natural morphism compatible with Higgs fields induced by the composites
      \begin{equation*}
          \begin{split}
              H^+(M^+) & = (M^+\otimes_{R_S^+}P_S^+)^{\Gamma}\\
              & = \left((H^+\otimes_{R_S^+}P_S^+)^{\Theta_{H^+} = 0}\otimes_{R_S^+}P_S^+\right)^{\Gamma}\\
              &\hookrightarrow (H^+\otimes_{R_S^+}P_S^+\otimes_{R_S^+}P_S^+)^{\Gamma}\\
              & \to (H^+\otimes_{R_S^+}P_S^+)^{\Gamma}\\
              & = H^+,
          \end{split}
      \end{equation*}
      where the last equality follows as $P_S^{\Gamma} = R_S^+$ (by Item (1)) while the last arrow is again induced by the multiplication on $P_S^+$.

      \textbf{Step 3:} It remains to show the morphism $\iota_{M^+}$ and $\iota_{(H^+,\theta)}$ are both isomorphisms. But this can be deduced from their constructions together with Equations \eqref{equ:local Simpson-II}, \eqref{equ:local Simpson-III}, \eqref{equ:local Simpson-IV} and \eqref{equ:local Simpson-V} immediately.

      We have finished the proof in the $? = \empty$ case. Now, we move to the case for $? = \infty$.

      For Item (1): Fix an $M_{\infty}^+\in \Rep_{\Gamma}^{\Hsmall}(\widehat R_{\infty,S}^+)$. By Proposition \ref{prop:decompletion}, there exists an $M^+\in \Rep_{\Gamma}^{\Hsmall}(\widehat R_{S}^+)$ such that $M_{\infty}^+ \cong M^+\otimes_{R_S^+}\widehat R_{\infty,S}^+$. Then the decomposition \eqref{equ:decomposition for P} induces a $\Gamma$-equivariant decomposition
      \[M_{\infty}^+\otimes_{\widehat R_{\infty,S}^+}P_{\infty,S}^+ = \widehat \bigoplus_{\underline \alpha\in J_r}M^+\otimes_{R_S^+}P_S^+\cdot \underline T^{\underline \alpha}.\] 
      Fix an $\underline \alpha = (\alpha_0,\dots,\alpha_r,\alpha_{r+1},\dots,\alpha_d)\in J_r\setminus\{0\}$. We claim that $\rR\Gamma(\Gamma,M^+\otimes_{R_S^+}P_S^+\cdot \underline T^{\underline \alpha})$ vanishes for $n = 0$ and is killed by $\zeta_p-1$ for any $n\geq 1$. Without loss of generality, we may assume $\alpha_0=0$ and then $\alpha_1,\dots,\alpha_d$ are not all zero. Then the claim follows from the isomorphism \eqref{equ:Gamma group-II} together with Corollary \ref{cor:MW}(1). 

      Thanks to the claim, the Item (1) follows from the $? = \emptyset$ case and moreover we have the identity
      \[(H^+(M^+_{\infty}),\theta) = (H^+(M^+),\theta).\]

      The Items (2) and (3) follows from the same arguments in the proof for the $? = \emptyset$ case. In particular, for any $(H^+,\theta = \sum_{i=1}^d\theta_i\otimes\frac{e_i}{\xi_K})\in\HIG^{\tHsmall}(R_S^+)$, we have 
      \begin{equation}\label{equ:local Simpson-VI}
          M_{\infty}^+(H^+,\theta) = \exp(-\sum_{i=1}^d\theta_iY_i)(H^+\otimes_{R_S^+}\widehat R_{\infty,S}^+) = \exp(-\sum_{i=1}^d\theta_iY_i)(H^+)\otimes_{R_S^+}\widehat R_{\infty,S}^+  = M^+(H^+,\theta)\otimes_{R_S^+}\widehat R_{\infty,S}^+
      \end{equation}
      on which $\gamma_i$ acts via the formulae \eqref{equ:local Simpson-V}. This implies the Item (4), and then completes the proof.
  \end{proof}

  Finally, we obtain the following equivalence of categories.
  \begin{cor}\label{cor:correct local Simpson}
      For any $?\in\{\emptyset,+\}$, we have the following equivalences of categories
      \[\Rep_{\Gamma}^{\Hsmall}(\widehat R_{\infty,S}^?)\simeq \Rep_{\Gamma}^{\Hsmall}(R_S^?)\simeq\HIG^{\tHsmall}(R_S^?)\simeq \HIG^{\Hsmall}(R_S^?).\]
      Moreover, for any $M_{\infty}\in \Rep_{\Gamma}^{\Hsmall}(\widehat R_{\infty,S})$ with corresponding $M\in \Rep_{\Gamma}^{\Hsmall}(\widehat R_{S})$, $(H,\theta)\in \HIG^{\tHsmall}(R_S)$ and $(H^{\prime},\theta^{\prime})\in \HIG^{\Hsmall}(R_S)$, we have quasi-isomorphisms
      \[\rR\Gamma(\Gamma,M_{\infty})\simeq \rR\Gamma(\Gamma,M)\simeq \rD\rR(H,\theta)\simeq\rD\rR(H^{\prime},\theta^{\prime}).\]
  \end{cor}
  \begin{proof}
      This follows from Proposition \ref{prop:decompletion}, Proposition \ref{prop:twist functor} and Theorem \ref{thm:local Simpson} immediately.
  \end{proof}
  It is worth showing the explicit form of the equivalence 
  \[\Rep_{\Gamma}^{\Hsmall}(\widehat R_{\infty,S})\simeq \HIG^{\Hsmall}(R_S).\]
  Let $M_{\infty}$ be a Hitchin-small $\Gamma$-representation of rank $r$ over $\widehat R_{\infty,S}$ with associated $M\in \Rep_{\Gamma}^{\Hsmall}(R_S)$; that is, we have $M_{\infty} = M\otimes_{R_S}\widehat R_{\infty,S}$. Fix an $R_S$-basis $e_1,\dots,e_d$ of $M$ such that for any $1\leq i\leq d$, the action of $\gamma_i\in\Gamma$ on $M$ is given by $\exp(-(\zeta_p-1)\rho_K\theta_i)$ for some topolohically nilpotent $\theta_i\in\Mat_r(R_S^+)$ with respect to the given basis. Let $(H,\theta)$ be the Hitchin-small Higgs module associated to $M_{\infty}$ over $R_S$. Then we have 
  \begin{equation}\label{equ:explicit form}
      (H = M,\theta = \sum_{i=1}^d(\zeta_p-1)\theta_i\otimes\frac{\dlog T_i}{\xi_K}),
  \end{equation}
  where we use the isomorphism $\Omega^{1}_{R_S} = \Omega^{1,\log}_{R_S^+}[\frac{1}{p}]$ and identify $e_i\in\Omega^{1,\log}_{R_S^+}$ (cf. \eqref{equ:omega^1-II}) with $\dlog T_i\in \Omega^{1}_{R_S}$.
  Conversely, if we start with a Hitchin-small Higgs module 
  \[(H,\theta = \sum_{i=1}^d(\zeta_p-1)\theta_i\otimes\frac{\dlog T_i}{\xi_K})\in \HIG^{\Hsmall}(R_S)\]
  with $\theta_i\in \Mat_r(R_S^+)$ topologically nilpotent, then its associated Hitchin-small $\Gamma$-representation $M$ over $R_S$ is given by $M = H$ together with the $\Gamma$-action such that for any $1\leq i\leq d$, the $\gamma_i$ acts via $\exp(-(\zeta_p-1)\rho_K\theta_i)$.

\subsection{An integral comparison theorem}
  Let $(H^+,\theta)\in\HIG^{\tHsmall}(R_S^+)$ be a twisted Hitchin-small Higgs module of rank $r$ with associated Hitchin-small $\Gamma$-representation $M^+\in \Rep_{\Gamma}^{\Hsmall}(R_S^+)$ (resp. $M_{\infty}^+\in \Rep_{\Gamma}^{\Hsmall}(\widehat R_{\infty,S}^+)$). As $\rD\rR(P_S^+,\Theta)$ (resp. $\rD\rR(P_{\infty,S}^+,\Theta)$) is a resolution of $R_S^+$ (resp. $\widehat R_{\infty,S}^+$), we have the left part of the following commutative diagram:
  \begin{equation}\label{diag:local canonical map}
      \xymatrix@C=0.5cm{
        \rR\Gamma(\Gamma,M^+)\ar@{^{(}->}[d]\ar[rr]^{\simeq\qquad\qquad}&&\rR\Gamma(\Gamma,\rD\rR(M^+\otimes_{R_S^+}P_S^+,\Theta_{M^+}))\ar@{^{(}->}[d]&&\ar@{_{(}->}[ll]\rD\rR(H^+,\theta)\ar@{=}[d]\\
        \rR\Gamma(\Gamma,M_{\infty}^+)\ar[rr]^{\simeq\qquad\qquad}&&\rR\Gamma(\Gamma,\rD\rR(M_{\infty}^+\otimes_{\widehat R_{\infty,S}^+}P_{\infty,S}^+,\Theta_{M_{\infty}^+}))&&\ar@{_{(}->}[ll]\rD\rR(H^+,\theta)
      }
  \end{equation}
  while the right part follows from the Theorem \ref{thm:local Simpson}.
  On the other hand, combining Proposition \ref{prop:decompletion} with \cite[Lem. 6.4 and Lem. 6.10]{BMS18}, we have the following commutative diagram
  \begin{equation}\label{diag:local BMS map}
      \xymatrix@C=0.5cm{
        \rL\eta_{\rho_K(\zeta_p-1)}\rR\Gamma(\Gamma,M^+)\ar[rr]\ar[d]^{\simeq}&&\rR\Gamma(\Gamma,M^+)\ar@{^{(}->}[d]\\
        \rL\eta_{\rho_K(\zeta_p-1)}\rR\Gamma(\Gamma,M_{\infty}^+)\ar[rr]&&\rR\Gamma(\Gamma,M_{\infty}^+).
      }
  \end{equation}
  Then we have the following result:
  \begin{prop}\label{prop:local integral comparison}
      Keep notations as above. For any $?\in\{\emptyset,\infty\}$, the composite
      \[\tau^{\leq 1}\rD\rR(H^+,\theta)\rightarrow\rD\rR(H^+,\theta)\xrightarrow{\eqref{diag:local canonical map}} \rR\Gamma(\Gamma,M_?^+)\]
      uniquely factors through the composite
      \[\tau^{\leq 1}\rL\eta_{\rho_K(\zeta_p-1)}\rR\Gamma(\Gamma,M_?^+)\to\rL\eta_{\rho_K(\zeta_p-1)}\rR\Gamma(\Gamma,M_?^+)\xrightarrow{\eqref{diag:local BMS map}}\rR\Gamma(\Gamma,M^+)\]
      and induces a quasi-isomorphism
      \[\tau^{\leq 1}\rD\rR(H^+,\theta)\xrightarrow{\simeq}\tau^{\leq 1}\rL\eta_{\rho_K(\zeta_p-1)}\rR\Gamma(\Gamma,M_?^+).\]
  \end{prop}
  \begin{proof}
      By the commutativity of diagrams \eqref{diag:local canonical map} and \eqref{diag:local BMS map}, it suffices to deal with the case for $? = \emptyset$. Then one can conclude by the same argument used in the proof of \cite[Prop. 5.6]{MW-AIM}. But for the convenience of readers, we exhibit details here.

      We claim the morphism $\rD\rR(H^+,\theta)\to \rR\Gamma(\Gamma,M^+)$ induces an isomorphism
      \begin{equation}\label{equ:local integral comparison-I}
          \rH^1_{\dR}(H^+,\theta):=\rH^1(\rD\rR(H^+,\theta))\cong \rho_K(\zeta_p-1)\rH^1(\Gamma,M^+).
      \end{equation}
      Granting this, one can conclude the result as follows:
      By \cite[Lem. 8.16]{BMS18}, the morphism $\tau^{\leq 1}\rD\rR(H^+,\theta)\to \tau^{\leq 1}\rR\Gamma(\Gamma,M^+)\to \rR\Gamma(\Gamma,M^+)$ uniquely factors as
      \[\tau^{\leq 1}\rD\rR(H^+,\theta)\to\tau^{\leq 1}\rL\eta_{\rho_K(\zeta_p-1)}\rR\Gamma(\Gamma,M^+)\to \tau^{\leq 1}\rR\Gamma(\Gamma,M^+)\to \rR\Gamma(\Gamma,M^+).\]
      Here, we implicitly commute $\tau^{\leq 1}$ and $\rL\eta_{\rho_K(\zeta_p-1)}$ (by \cite[Cor. 6.5]{BMS18}). By the proof of \cite[Lem. 6.10]{BMS18}, the $\rH^i$ of the morphism
      $\tau^{\leq 1}\rL\eta_{\rho_K(\zeta_p-1)}\rR\Gamma(\Gamma,M^+)\to \tau^{\leq 1}\rR\Gamma(\Gamma,M^+)$ is given by 
      \[\rH^i(\tau^{\leq 1}\rL\eta_{\rho_K(\zeta_p-1)}\rR\Gamma(\Gamma,M^+)) = \rH^i(\Gamma,M^+)/\rH^i(\Gamma,M^+)[\rho_K(\zeta_p-1)]\xrightarrow{\times(\rho_K(\zeta_p-1))^i}\rH^i(\Gamma,M^+)\]
      for $i\in\{0,1\}$. In particular, this induces an isomorphism 
      \[\rH^i(\tau^{\leq 1}\rL\eta_{\rho_K(\zeta_p-1)}\rR\Gamma(\Gamma,M^+))\cong (\rho_K(\zeta_p-1))^i\rH^i(\Gamma,M^+)\]
      for $i\in\{0,1\}$. Then it follows from the claim above that the morphism
      \[\tau^{\leq 1}\rD\rR(H^+,\theta)\to\tau^{\leq 1}\rL\eta_{\rho_K(\zeta_p-1)}\rR\Gamma(\Gamma,M_?^+)\]
      is a quasi-isomorphism as desired.

      Now, we focus on the proof of the claim above. Put $P(M^+):=M^+\otimes_{R_S^+}P_S^+$ for short and consider the following commutative diagram: 
\begin{equation}\label{diag:double complex}
    \xymatrix@C=0.5cm{
      && H^+\ar[rr] \ar[d]&& H^+\otimes\Omega^{1,\log}_{R_S^+}\{-1\}\ar[rr]\ar[d] && H^+\otimes\rho^2\widehat\Omega^2(-2)\ar[r]\ar[d] &\dots\\
      M^+\ar[rr]\ar[d]&& P(M^+)\ar[rr]\ar[d]&& P(M^+)\otimes\Omega^{1,\log}_{R_S^+}\{-1\}\ar[rr]\ar[d] && P(M^+)\otimes\Omega^{2,\log}_{R_S^+}\{-2\}\ar[r]\ar[d] &\dots\\
      \wedge^1(M^+)^d\ar[d]\ar[rr]&& \wedge^1P(M^+)^d\ar[rr]\ar[d]&& \wedge^1P(M^+)^d\otimes\Omega^{1,\log}_{R_S^+}\{-1\}\ar[rr]\ar[d] && \wedge^1P(M^+)^d\otimes\Omega^{2,\log}_{R_S^+}\{-2\}\ar[r]\ar[d] &\dots\\
      \wedge^2(M^+)^d\ar[d]\ar[rr]&& \wedge^2P(M^+)^d\ar[rr]\ar[d]&& \wedge^2P(M^+)^d\otimes\Omega^{1,\log}_{R_S^+}\{-1\}\ar[rr]\ar[d] && \wedge^2P(M^+)^d\otimes\Omega^{2,\log}_{R_S^+}\{-2\}\ar[r]\ar[d] &\dots\\
      \vdots&& \vdots&& \vdots &&\vdots &
    }
\end{equation}
where the horizontal arrows are induced by Higgs fields and the vertical arrows are induced by Koszul complexes associated to $\Gamma$-actions. Then we have the following commutative diagram
\begin{equation}\label{diag:total complex}
    \xymatrix@C=0.5cm{
      H^+\ar[r] \ar[d]& H^+\otimes_{R_S^+}\Omega^{1,\log}_{R_S^+}\{-1\}\ar[rr]\ar[d]&&\dots\\
      P(M^+)\ar[r]& P(M^+)\otimes_{R_S^+}\Omega^{1,\log}_{R_S^+}\{-1\}\oplus\wedge^1P(M^+)^d\ar[rr]&&\dots\\
      M^+\ar[r]\ar[u]&\wedge^1(M^+)^d\ar[rr]\ar[u]&&\dots
    }
\end{equation}
such that the arrows from the bottom to the middle induce the quasi-isomorphism 
\[\rR\Gamma(\Gamma,M^+)\simeq\rR\Gamma(\Gamma,\HIG(P(M^+),\Theta_{M^+}])).\]
We have to deduce the relation between $\rH_{\dR}^1(H^+,\theta)$ and $\rH^1(\Gamma,M^+)$ from the diagram \eqref{diag:total complex}.

Let $x_1,\dots,x_d\in H^+$ such that $\omega = \sum_{i=1}^dx_i\otimes\frac{e_i}{\xi_K}$ represents an element in $\rH_{\dR}^1(H^+,\theta)$ via the isomorphism $\Omega^{1,\log}_{R_S^+}\cong \oplus_{i=1}^dR_S^+\cdot e_i$, cf. \eqref{equ:omega^1-II}. Equivalently, we have that for any $1\leq i,j\leq d$, $\theta_i(x_j) = \theta_j(x_i)$, where we write $\theta = \sum_{i=1}^d\theta_i\otimes\frac{e_i}{\xi_K}$. We want to determine the element in $\rH^1(\Gamma,M^+)$ induced by $\omega$. To do so, we have to solve the equation 
\begin{equation}\label{equ:bridge}
    \Theta_{M^+}(\sum_{\underline n}h_{\underline n}\underline Y^{[\underline n]})=\omega,
\end{equation}
where $\sum_{\underline n}h_{\underline n}\underline Y^{[\underline n]}=:m\in M$. Note that
\begin{equation*}
    \begin{split}
        \Theta_{M^+}(\sum_{\underline n}h_{\underline n}\underline Y^{[\underline n]})
        =&\sum_{i=1}^d(\sum_{\underline n}\theta_i(h_{\underline n})\underline Y^{[\underline n]}+\sum_{\underline n}h_{\underline n}\underline Y^{[\underline n-\underline 1_i]})\otimes\frac{e_i}{\xi_K}\\
        =&\sum_{i=1}^d\sum_{\underline n}(\theta_i(h_{\underline n})+h_{\underline n+\underline 1_i})\underline Y^{[\underline n]}\otimes\frac{\dlog T_i}{\xi_K}.
    \end{split}
\end{equation*}
So \eqref{equ:bridge} holds true if and only if for any $1\leq i\leq d$ and $\underline n\in\bN^d$ satisfying $|\underline n| \geq 1$,
\begin{equation}\label{equ:iteration-I}
    h_{\underline n+\underline 1_i} = -\theta_i( h_{\underline n} )
\end{equation}
and 
\begin{equation}\label{equ:iteration-II}
    h_{\underline 1_i} = -\theta_i(h_0)+x_i.
\end{equation}
As $\theta_i(x_j) = \theta_j(x_i)$ for any $1\leq i,j\leq d$, it is easy to see that for any $h\in H^+$, one can put $h_0 = h$ and use \eqref{equ:iteration-I} and \eqref{equ:iteration-II} to achieve an element $m(h)\in M$ satisfying $\Theta_{M^+}(m(h)) = \omega$. Moreover, if we put $m(\omega):=m(0)$, then 
  \[m(h) = m(\omega)+\exp(-\sum_{k=1}^d\theta_kY_k)h.\]
As a consequence, the image of $\omega$ in $\rH^1(\Gamma,M^+) = \rH^1(\rK(\gamma_1-1,\dots,\gamma_d-1;M^+))$ is represented by 
  \[v(\omega):=(\gamma_1(m(\omega))-m(\omega),\dots,\gamma_d(m(\omega))-m(\omega))\in \wedge^1 (M^+)^d.\]

  On the other hand, as $\gamma_1$ acts on $Y_2,\dots,Y_d$ trivially (cf. Proposition \ref{prop:local OC}), we have 
  \begin{equation*}
      \begin{split}
          &\gamma_1(m(\omega))-m(\omega)\\
          =&\gamma_1\left(\sum_{n_1\geq 1,n_2,\dots,n_d\geq 0}(-\theta_1)^{n_1-1}(-\theta_2)^{n_2}\cdots(-\theta_d)^{n_d}x_1Y_1^{[n_1]}\cdots Y_d^{[n_d]}\right)\\
          &-\sum_{n_1\geq 1,n_2,\dots,n_d\geq 0}(-\theta_1)^{n_1-1}(-\theta_2)^{n_2}\cdots(-\theta_d)^{n_d}x_1Y_1^{[n_1]}\cdots Y_d^{[n_d]}\\
          =&\sum_{n_1\geq 1,n_2,\dots,n_d\geq 0}(-\theta_1)^{n_1-1}(-\theta_2)^{n_2}\cdots(-\theta_d)^{n_d}x_1(Y_1+\rho_K(\zeta_p-1))^{[n_1]}Y_2^{[n_2]}\cdots Y_d^{[n_d]}\\
          &-\sum_{n_1\geq 1,n_2,\dots,n_d\geq 0}(-\theta_1)^{n_1-1}(-\theta_2)^{n_2}\cdots(-\theta_d)^{n_d}x_1Y_1^{[n_1]}\cdots Y_d^{[n_d]}\\
          =&\sum_{n_1\geq 1,n_2,\dots,n_d\geq 0,0\leq l\leq n_1}(-\theta_1)^{n_1-1}(-\theta_2)^{n_2}\cdots(-\theta_d)^{n_d}x_1\rho_K^{n_1-l}(\zeta_p-1)^{[n_1-l]}Y_1^{[l]}Y_2^{[n_2]}\cdots Y_d^{[n_d]}\\
          &-\sum_{n_1\geq 1,n_2,\dots,n_d\geq 0}(-\theta_1)^{n_1-1}(-\theta_2)^{n_2}\cdots(-\theta_d)^{n_d}x_1Y_1^{[n_1]}\cdots Y_d^{[n_d]}\\
          =&\sum_{n_1\geq 1,n_2,\dots,n_d\geq 0,0\leq l\leq  n_1-1}(-\theta_1)^{n_1-1}(-\theta_2)^{n_2}\cdots(-\theta_d)^{n_d}x_1\rho_K^{n_1-l}(\zeta_p-1)^{[n_1-l]}Y_1^{[l]}Y_2^{[n_2]}\cdots Y_d^{[n_d]}\\
          =&\sum_{n_1,n_2,\dots,n_d\geq 0,l\geq 0}(-\theta_1)^{n_1+l}(-\theta_2)^{n_2}\cdots(-\theta_d)^{n_d}x_1\rho_K^{l+1}(\zeta_p-1)^{[l+1]}Y_1^{[n_1]}\cdots Y_d^{[n_d]}\\
          =&\sum_{l\geq 0}\rho_K^{l+1}(-\theta_1)^{l}(\zeta_p-1)^{[l+1]}\exp(-\sum_{k=1}^d\theta_kY_k)x_1\\
          =&(\zeta_p-1)\rho_K F(\rho_K\theta_1)\exp(-\sum_{k=1}^d\theta_kY_k)x_1,
      \end{split}
  \end{equation*}
  where $F(\theta) = \frac{1-\exp(-(\zeta_p-1)\theta)}{(\zeta_p-1)\theta}$ was defined in \eqref{equ:F(X)}. Similarly, for any $1\leq i\leq d$, we have
  \[(\gamma_i-1)(m(\omega)) = \rho(\zeta_p-1)F(\theta_i)\exp(-\sum_{k=1}^d\theta_kY_k)x_i.\]
  As a consequence, the image of $\omega$ in $\rH^1(\Gamma,M^+)$ is represented by
  \begin{equation}\label{equ:v(omega)}
      v(\omega)=(\zeta_p-1)\rho_K(F(\rho_K\theta_1)\exp(-\sum_{k=1}^d\theta_kY_k)x_1,\dots,F(\rho_K\theta_d)\exp(-\sum_{k=1}^d\theta_kY_k)x_d).
  \end{equation}
  Since for any $1\leq i,j\leq d$,
  \begin{equation}\label{equ:bridge-II}
      \begin{split}
          (\gamma_j-1)(F(\rho_K\theta_i)\exp(-\sum_{k=1}^d\theta_kY_k)x_i)=&F(\rho_K\theta_i)\exp(-\sum_{k=1}^d\theta_kY_k)(\exp(-(\zeta_p-1)\rho_K\theta_j)-1)x_i\\
          =&(\zeta_p-1)\rho_KF(\rho_K\theta_i)F(\rho_K\theta_j)\exp(-\sum_{k=1}^d\theta_kY_k)\theta_j(x_i),
      \end{split}
  \end{equation}
  and $\theta_j(x_i) = \theta_i(x_j)$, we deduce that
  \[(\gamma_j-1)(F(\rho_K\theta_i)\exp(-\sum_{k=1}^d\Theta_kY_k)x_i) = (\gamma_i-1)(F(\rho_K\theta_i)\exp(-\sum_{k=1}^d\theta_kY_k)x_j)\]
  and hence that
  \[v'(\omega):=(F(\rho_K\theta_1)\exp(-\sum_{k=1}^d\theta_kY_k)x_1,\dots,F(\rho_K\theta_d)\exp(-\sum_{k=1}^d\theta_kY_k)x_d)\in\wedge^1(M^+)^d\]
  represents an element in $\rH^1(\Gamma,M^+)$. 
  
  Therefore, as a cohomological class, we have 
  \[v(\omega) = \rho_K(\zeta_p-1)v'(\omega)\in\rho_K(\zeta_p-1)\rH^1(\Gamma,V_0).\]
  In other words, the map $\rD\rR(H,\theta_H)\to\rR\Gamma(\Gamma,M^+)$ carries $\rH_{\dR}^1(H^+,\theta_H)$ into $\rho_K(\zeta_p-1)\rH^1(\Gamma,M^+)$. We have to show it induces an isomorphism 
  \[\rH_{\dR}^1(H^+,\theta_H)\cong\rho_K(\zeta_p-1)\rH^1(\Gamma,M^+).\]

  The injectivity is obvious. Indeed, let $T$ be the total complex of the double complex in (\ref{diag:double complex}) representing $\rR\Gamma(\Gamma,\rD\rR(P(M^+),\Theta_M))$. By the spectral sequence argument, we have
  \[\rH_{\dR}^1(H^+,\Theta) = E_2^{1,0} = E_{\infty}^{1,0}\subset \rH^1(T) \cong \rH^1(\Gamma,M^+)\]
  and the desired injectivity follows.

  It remains to prove the surjectivity. For this, let $y_1,\dots,y_d\in H$ such that 
  \[v' = (F(\rho_K\theta_1)\exp(-\sum_{j=1}^d\theta_jY_j)y_1,\dots,F(\rho_K\theta_d)\exp(-\sum_{j=1}^d\theta_jY_j)y_d)\in \wedge^1(M^+)^d\]
  represents an element in $\rH^1(\rK(\gamma_1-1,\dots,\gamma_d-1;M^+))\cong\rH^1(\Gamma,M^+)$. Equivalently, we have that for any $1\leq i,j\leq d$,
  \begin{equation*}
      \begin{split}
          (\gamma_j-1)(F(\rho_K\theta_i)\exp(-\sum_{k=1}^d\theta_kY_k)y_i)=(\gamma_i-1)(F(\rho_K\theta_j)\exp(-\sum_{k=1}^d\theta_kY_k)y_j).
      \end{split}
  \end{equation*}
  By (\ref{equ:bridge-II}), this amounts to that
  \begin{equation*}
      \begin{split}
          (\zeta_p-1)\rho_KF(\rho_K\theta_i)F(\rho_K\theta_j)\exp(-\sum_{k=1}^d\theta_kY_k)\theta_j(y_i)=\rho_K(\zeta_p-1)F(\rho_K\theta_i)F(\rho_K\theta_j)\exp(-\sum_{k=1}^d\theta_kY_k)\theta_i(y_j).
      \end{split}
  \end{equation*}
  By noting that $F(\rho_K\theta_i)$'s are invertible (as $\theta_i$'s are topologically nilpotent and $F(\theta)\equiv 1\mod \theta$), we conclude that $v'$ represents an element in $\rH^1(\Gamma,M^+)$ if and only if $\theta_i(y_j) = \theta_j(y_i)$ for any $1\leq i,j\leq d$. As a consequence, $\omega' = \sum_{i=1}^dy_i\otimes\frac{\dlog T_i}{\xi_K}$ represents an element in $\rH_{\dR}^1(H,\theta)$.
  Now for any $v'\in \rH^1(\Gamma,M^+)$, by (\ref{equ:v(omega)}), we know that $v(\omega') = \rho_K(\zeta_p-1)v'$. That is, $\rho_K(\zeta_p-1)v'$ is contained in the image of 
  \[\rH_{\dR}^1(H,\theta)\to\rho_K(\zeta_p-1)\rH^1(\Gamma,M^+).\]
  As $v'$ is arbitrary in $\rH^1(\Gamma,M^+)$, this proves the desired surjectivity and thus the desired claim.
  \end{proof}

\section{The $p$-adic Simpson correspondence on Hitchin-small objects}\label{sec:global Simpson}
  In this section, we aim to generalize the integral Simpson correspondence \cite[Th. 1.1]{MW-AIM} and the geometric stacky Simpson correspondence \cite[Th. 1.1]{AHLB-small} to the case for semi-stable $\frakX$. From now on, we always assume $\frakX$ is a liftable semi-stable formal scheme over $\calO_C$ of relative dimension $d$ over $\calO_C$ with the generic fiber $X$ in the sense of \cite{CK19} and  fix an its lifting $\widetilde \frakX$ over $\bfA_{2,K}$. For any $S = \Spa(A,A^+)\in\Perfd$, denote by $\frakX_S$, $X_S$ and $\widetilde \frakX_S$ the base-changes of $\frakX$, $X$ and $\widetilde \frakX$ to $A^+$, $A$ and $\AAKK(S)$, respectively. Then $\frakX_S$ is a liftable semi-stable formal scheme over $A^+$ with the generic fiber $X_S$ and the induced lifting $\widetilde \frakX_S$ from $\widetilde \frakX$. 
  When $\frakX = \Spf(R^+)$ with the lifting $\widetilde \frakX = \Spf(\widetilde R^+)$ and the generic fiber $X = \Spa(R,R^+)$, we also denote by $\Spf(R_S^+)$, $\Spa(R_S,R_S^+)$ and $\widetilde R^+_S$ for the corresponding base-changes. 
%  When $R^+$ is small semi-stable, the notations at the end of \S\ref{ssec:semistable schemes} is functorial in $S$.
  For any $?\in \{\emptyset,+\}$, let $\calO\widehat \bC_{\pd,S}^?$ be the period sheaf with Higgs fields constructed in \S\ref{ssec:period sheaves} corresponding to the lifting $\widetilde \frakX_S$. By construction, it is also functorial in $S$ and so is the notation $P_{\infty,S}^?$ in Proposition \ref{prop:local OC}.

\subsection{The integral $p$-adic Simpson correspondence}\label{ssec:integral Simpson}
  Fix an $S = \Spa(A,A^+)\in \Perfd$ and let $\frakX_S$, $\widetilde \frakX_S$, $X_S$, and etc. be as before.
  \begin{dfn}\label{dfn:small O^+-representations}
      By a \emph{Hitchin-small integral $v$-bundle of rank $r$} on $X_{S,v}$, we mean a sheaf of locally finite free $\widehat \calO_{X_S}^+$-modules $\calM^+$ on $X_{S,v}$ such that there exists an \'etale covering $\{\frakX_{i,S}\to \frakX_S\}_{i\in I}$ by small affine $\frakX_{i,S} = \Spf(R_{i,S}^+)$ such that for any $i\in I$, the $\calM^+(X_{i,\infty,S})$ is a Hitchin-small $\Gamma$-representation of rank $r$ over $\widehat R_{i,\infty,S}^+$ in the sense of Definition \ref{dfn:small Gamma-representation}. Denote by $\rL\rS^{\Hsmall}(\frakX_S,\widehat \calO_{X_S}^+)$ the category of Hitchin-small integral $v$-bundles on $X_{S,v}$.
  \end{dfn}
  \begin{dfn}\label{dfn:small Higgs bundles}
      By an \emph{Higgs bundle of rank $r$} on $\frakX_{S,\et}$, we mean a pair $(\calH^+,\theta)$ consisting of a sheaf $\calH^+$ of locally finite free $\calO_{\frakX_S}$-modules of rank $r$ on $X_{S,\et}$ and a Higgs field $\theta$ on $\calH^+$, i.e., an $\calO_{\frakX_S}$-linear homomorphism 
      \[\theta:\calH^+\to\calH^+\otimes_{\calO_{\frakX_S}}\Omega^{1,\log}_{\frakX_S}\{-1\}\]
      satisfying the condition $\theta\wedge\theta = 0$. 
      For any Higgs bundle $(\calH^+,\theta)$, denote by $\rD\rR(\calH^+,\theta)$ the induced Higgs complex.
      A Higgs bundle $(\calH^+,\theta)$ is called 
      \begin{enumerate}
          \item \emph{twisted Hitchin-small} if $\theta$ is topologically nilpotent;

          \item \emph{Hitchin-small} if it is of the form
          \[(\calH^+,\theta) = (\calH^+,(\zeta_p-1)\theta^{\prime})\]
          for some twisted Hitchin-small integral Higgs bundles.
      \end{enumerate} Denote by $\HIG^{\text{(t-)}\Hsmall}(\frakX_S,\calO_{\frakX_S})$ the category of (twisted) integral Hitchin-small Higgs bundles on $\frakX_{S,\et}$.
  \end{dfn}
  \begin{lem}\label{lem:global twist functor}
      The twist functor $(\calH^+,\theta)\mapsto (\calH^+,(\zeta_p-1)\theta)$ induces an equivalence of categories
      \[\HIG^{\tHsmall}(\frakX_S,\calO_{\frakX_S})\xrightarrow{\simeq}\HIG^{\Hsmall}(\frakX_S,\calO_{\frakX_S})\]
      such that for any $(\calH^+,\theta)\in \HIG^{\tHsmall}(\frakX_S,\calO_{\frakX_S})$ with the induced $(\calH^+,\theta^{\prime})\in\HIG^{\Hsmall}(\frakX_S,\calO_{\frakX_S})$, there exists a quasi-isomorphism
      \[\rL\eta_{\zeta_p-1}\rD\rR(\calH^+,\theta^{\prime})\simeq \rD\rR(\calH^+,\theta).\]
  \end{lem}
  \begin{proof}
      This indeed follows from the definition (of $\rL\eta$). By \cite[\href{https://stacks.math.columbia.edu/tag/06YQ}{Tag 06YQ}]{Sta}, the $\rD\rR(\calH^+,\theta^{\prime})$ is strongly $K$-flat in $D(\calO_{\frakX_S})$ in the sense of \cite[\S6]{BMS18} (as it is a bounded complex of vector bundles). So $\rL\eta_{\zeta_p-1}\rD\rR(\calH^+,\theta^{\prime})$ is represented by the complex $\eta_{\zeta_p-1}\rD\rR(\calH^+,\theta^{\prime})$ defined in \cite[Def. 6.2]{BMS18}. Denote its differentials by $\rd$ for simplicity. As $\theta^{\prime} = (\zeta_p-1)\theta$, for any $n\geq 0$, we have
      \[(\eta_{\zeta_p-1}\rD\rR(\calH^+,\theta^{\prime}))^n = \calH^+\otimes_{\calO_{\frakX_S}}\Omega^{n,\log}_{\frakX_S}\{-n\}\otimes(\zeta_p-1)^n.\]
      By definition of $\rd$ (cf. the commutative diagram in \cite[Def. 6.2]{BMS18}), for any local section $x\in \calH^+\otimes_{\calO_{\frakX_S}}\Omega^{n,\log}_{\frakX_S}\{-n\}$, we have
      \[\rd(x)\otimes(\zeta_p-1)^{n+1} = \theta^{\prime}(x)\otimes(\zeta_p-1)^n = \theta(x)\otimes(\zeta_p-1)^{n+1}.\]
      Via the isomorphism $(\eta_{\zeta_p-1}\rD\rR(\calH^+,\theta^{\prime}))^n\cong \calH^+\otimes_{\calO_{\frakX_S}}\Omega^{n,\log}_{\frakX_S}\{-n\}$, we have an isomorphism of complexes
      \[(\eta_{\zeta_p-1}\rD\rR(\calH^+,\theta^{\prime}),\rd)\simeq \rD\rR(\calH^+,\theta),\]
      yielding the desired quasi-isomorphism.
  \end{proof}
  The main theorem in this section is the following generalisation of \cite[Th. 1.1]{MW-AIM}.
  \begin{thm}\label{thm:integral Simpson}
      Let $(\calO\widehat \bC_{\pd,S}^+,\Theta)$ be the period sheaf with Higgs field associated to $\widetilde \frakX_S$. Let $\nu: X_{S,v}\to \frakX_{S,\et}$ be the natural morphism of sites.
      \begin{enumerate}
          \item[(1)] For any $\calM^+\in \rL\rS^{\Hsmall}(\frakX_S,\widehat \calO^+_{X_S})$ of rank $r$, put 
          \[\Theta_{\calM^+}:=\id\otimes\Theta:\calM^+\otimes\calO\widehat \bC_{\pd,S}^+\to \calM^+\otimes\calO\widehat \bC_{\pd,S}^+\otimes_{\calO_{\frakX_S}}\Omega^{1,\log}_{\frakX_S}\{-1\}.\] 
          Then the complex $\rR\nu_*(\calM^+\otimes\calO\widehat \bC_{\pd,S}^+)$ is concentrated in degree $[0,d]$ such that 
          \[\rL\eta_{\rho_K(\zeta_p-1)}\rR\nu_*(\calM^+\otimes\calO\widehat \bC_{\pd,S}^+)\simeq \left(\nu_*(\calM^+\otimes\calO\widehat \bC_{\pd,S}^+)\right)[0].\]
          Moreover, the push-forward
          \[(\calH^+(\calM^+),\theta):=(\nu_*(\calM^+\otimes\calO\widehat \bC_{\pd,S}^+),\nu_*(\Theta_{\calM^+}))\]
          defines a twisted Hitchin-small Higgs bundle of rank $r$ on $\frakX_{S,\et}$.
        
          \item[(2)] For any $(\calH^+,\theta)\in \HIG^{\tHsmall}(\frakX_S,\calO_{\frakX_S})$ of rank $r$, put
          \[\Theta_{\calH^+}=\theta\otimes\id+\id\otimes\Theta:\calH^+\otimes\calO\widehat \bC_{\pd,S}^+\to \calH^+\otimes\calO\widehat \bC_{\pd,S}^+\otimes_{\calO_{\frakX_S}}\Omega^{1,\log}_{\frakX_S}\{-1\}.\]
          Then the
          \[\calM^+(\calH^+,\theta):=(\calH^+\otimes\calO\widehat \bC_{\pd,S}^+)^{\Theta_{\calH^+} = 0}\]
          defines a Hitchin-small integral $v$-bundle of rank $r$ on $X_{S,v}$.
          
          \item[(3)] The functors in Items (1) and (2) defines an equivalence of categories
          \[\rL\rS^{\Hsmall}(\frakX_S,\widehat \calO_{X_S}^+)\xrightarrow{\simeq}\HIG^{\tHsmall}(\frakX_S,\calO_{\frakX_S}).\]
          which preserves ranks, tensor products and dualities.
      \end{enumerate}
  \end{thm}
  Combining this with Lemma \ref{lem:global twist functor}, we obtain
  \begin{cor}\label{cor:untwist integral Simpson}
      The composite
      \[\rL\rS^{\Hsmall}(\frakX_S,\widehat \calO_{X_S}^+)\rightarrow\HIG^{\tHsmall}(\frakX_S,\calO_{\frakX_S})\xrightarrow{\times(\zeta_p-1)}\HIG^{\Hsmall}(\frakX_S,\calO_{\frakX_S})\]
      defines an equivalence of categories
      \[\rL\rS^{\Hsmall}(\frakX_S,\widehat \calO_{X_S}^+)\xrightarrow{\simeq}\HIG^{\Hsmall}(\frakX_S,\calO_{\frakX_S})\]
      preserving ranks, tensor products and dualities.
  \end{cor}
\subsubsection{Proof of Theorem \ref{thm:integral Simpson}}
   We follow the strategy in \cite[\S4.2]{MW-AIM}.
   \begin{lem}\label{lem:reduce to group cohomology}
       Suppose $\frakX_S = \Spf(R^+_S)$ is small affine with the generic fiber $X_S = \Spa(R_S,R_S^+)$
       Let $\calM^+$ be a Hitchin-small integral $v$-bundle of rank $r$ on $X_{S,v}$ such that $M_{\infty}^+:=\calM^+(X_{\infty,S})$ is a Hitchin-small $\Gamma$-representation of rank $r$ over $\widehat R_{\infty,S}^+$. Then there exists a natural morphism
       \[\rR\Gamma(\Gamma,(\calM^+\otimes_{\widehat \calO^+_{X_S}}\calO\widehat \bC_{\pd,S}^+)(X_{\infty,S}))\to \rR\Gamma(X_{S,v},\calM^+\otimes_{\widehat \calO^+_{X_S}}\calO\widehat \bC_{\pd,S}^+)\]
       which is an almost quasi-isomorphism and is an isomorphism in degree $0$.
   \end{lem}
   \begin{proof}
       It suffices to show that there exists some $c>0$ such that there is an almost isomorphism of sheaves of $\widehat \calO_{X_S}^+$-modules
       \[\calM^+/p^c \cong (\widehat \calO_{X_S}^+/p^c)^r.\]
       Granting this, one can conclude by the same argument in the proof of \cite[Lem. 4.15]{MW-AIM}.
       
       Denote by $X_{\infty,S}^{\bullet}$ the \v Cech nerve associated to the $\Gamma$-torsor $X_{\infty,S}\to X_S$. Then we have 
       \[X_{\infty,S}^{\bullet} = \Spa(\rC(\Gamma^{\bullet},\widehat R_{\infty,S}),\rC(\Gamma^{\bullet},\widehat R_{\infty,S}^+))\]
       where $\rC(\Gamma^{\bullet},N)$ denotes the ring of continuous functions from $\Gamma^{\bullet}$ into $N$ for any $\Zp$-module $N$ equipped with a continuous $\Gamma$-action.
       The Hitchin-smallness of $M_{\infty}^+$ yields a $\Gamma$-equivariant isomorphism
       \[M_{\infty}^+/\rho_K(\zeta_p-1)\cong (\widehat R_{\infty,S}^+/\rho_K(\zeta_p-1))^r.\]
       Note that the continuous $\Gamma$-action $M_{\infty}^+$ induces a cosimplicial $\rC(\Gamma^{\bullet},\widehat R_{\infty,S}^+)$-module $\rC(\Gamma^{\bullet},M_{\infty}^+)$, yielding an isomorphism
       \[\rC(\Gamma^{\bullet},M_{\infty}^+/\rho_K(\zeta_p-1))\cong \rC(\Gamma^{\bullet},\widehat R_{\infty,S}^+/\rho_K(\zeta_p-1))^r\]
       of cosimplicial $\rC(\Gamma^{\bullet},\widehat R_{\infty,S}^+)$-modules. As $X_{\infty,S}\to X_S$ is a $\Gamma$-torsor, by the proof of \cite[Lem. 4.10(i)]{Sch-Pi}, we have the desired almost isomorphism
       \[\calM^+/p^c \cong (\widehat \calO_{X_S}^+/p^c)^r\]
       for any $0<c<\nu_p((\zeta_p-1)\rho_K)$.
   \end{proof}
   \begin{cor}\label{cor:reduce to group cohomology}
       Keep assumption and notations in Lemma \ref{lem:reduce to group cohomology}. Then the above natural morphism induces a quasi-isomorphism
       \[\rL\eta_{(\zeta_p-1)\rho_K}\rR\Gamma(\Gamma,(\calM^+\otimes_{\widehat \calO^+_{X_S}}\calO\widehat \bC_{\pd,S}^+)(X_{\infty,S}))\xrightarrow{\simeq} \rL\eta_{(\zeta_p-1)\rho_K}\rR\Gamma(X_{S,v},\calM^+\otimes_{\widehat \calO^+_{X_S}}\calO\widehat \bC_{\pd,S}^+).\]
       Moreover, we have the following quasi-isomorphisms
       \[\begin{split}\rH^0(X_{v,S},(\calM^+\otimes_{\widehat \calO_{X_S}^+} \calO\widehat \bC^+_{\pd,S}))[0] &\cong\rH^0(\Gamma,(\calM^+\otimes_{\widehat \calO_{X_S}^+} \calO\widehat \bC^+_{\pd,S})(X_{\infty,S}))[0]\\
      &\cong \rL\eta_{\rho_K(\zeta_p-1)}\rR\Gamma(\Gamma,(\calM^+\otimes_{\widehat \calO_{X_S}^+} \calO\widehat \bC^+_{\pd,S})(X_{\infty}))\\
      &\cong\rL\eta_{\rho_K(\zeta_p-1)}\rR\Gamma(X_{v,S},(\calM^+\otimes_{\widehat \calO_{X_S}^+} \calO\widehat \bC^+_{\pd,S})).
      \end{split}\]
   \end{cor}
   \begin{proof}
       For the first desired quasi-isomorphism, by \cite[Lem. 8.11(2)]{BMS18} and Lemma \ref{lem:reduce to group cohomology}, it suffices to show for any $k\geq 0$, the
       \[\rH^k(\Gamma,(\calM^+\otimes_{\widehat \calO^+_{X_S}}\calO\widehat \bC_{\pd,S}^+)(X_{\infty,S})) = \rH^k(\Gamma,M_{\infty}^+\otimes_{\widehat R^+_{\infty,S}}P_{\infty,S}^+).\]
       has no $\frakm_C$-torsion. Let $(H^+,\theta)$ be the twisted Hitchin-small Higgs module over $R_S^+$ associated to $M_{\infty}^+$ via Theorem \ref{thm:integral Simpson}. Then we have a quasi-isomorphism
       \[\rR\Gamma(\Gamma,M_{\infty}^+\otimes_{\widehat R^+_{\infty,S}}P_{\infty,S}^+)\simeq \rR\Gamma(\Gamma,H^+\otimes_{R^+_S}P_{\infty,S}^+)\simeq H^+\otimes_{R_S^+}\rR\Gamma(\Gamma,P_{\infty,S}^+),\]
       where the second follows as $H^+$ is finite free over $R_S^+$ with trivial $\Gamma$-action. So we have an isomorphism
       \[\rH^k(\rR\Gamma(\Gamma,M_{\infty}^+\otimes_{\widehat R^+_{\infty,S}}P_{\infty,S}^+))\cong H^+\otimes_{R_S^+}\rH^k(\rR\Gamma(\Gamma,P_{\infty,S}^+))\]
       and thus are reduced to showing that $\rH^k(\Gamma,P_{\infty,S}^+)$ has no $\frakm_C$-torsion. By decomposition \eqref{equ:decomposition for P}, it suffices to show for any $\underline \alpha\in J_r$, the $\rH^k(\Gamma,P_S^+\cdot \underline T^{\underline \alpha})$ has no $\frakm_C$-torsion. 
       For $\underline \alpha = 0$, the $\rR\Gamma(\Gamma,P_S^+)$ is computed by the Koszul complex
       \[\rK(\gamma_1-1,\dots,\gamma_d-1;P_S^+)\simeq \widehat{\bigotimes}_{1\leq i\leq d}^{\rL}\rK(\gamma_i-1;R_S^+[Y_i]^{\wedge}_{\pd})\]
       where $\widehat{\bigotimes}^{\rL}$ denotes the $p$-complete derived tensor product. By K\"unneth formulae, we are reduced to the case for $d = 1$; that is, we just need to show 
       \[R_S^+[Y]^{\wedge}_{\pd}\xrightarrow{\gamma-1}R_S^+[Y]^{\wedge}_{\pd}\]
       has cohomologies with no $\frakm_C$-torsion, where the $\Gamma$-action on $R_S[Y]^{\wedge}_{\pd}$ is given by Construction \ref{construction:MW} (for $B = R_S^+$). It follows from Lemma \ref{lem:MW}(2) that $\rH^i(\Gamma,R_S^+[Y]^{\wedge}_{\pd})$ has no $\frakm_C$-torsion for any $i\geq 0$. For $\underline \alpha = (\alpha_0,\dots,\alpha_d) \neq 0$, without loss of generality, we may assume $\alpha_0 = 0$ and then the $\rR\Gamma(\Gamma,P_S^+\cdot \underline T^{\underline \alpha})$ is computed by the Koszul complex
       \[\rK(\gamma_1-1,\dots,\gamma_d-1;P_S^+\cdot \underline T^{\underline \alpha})\simeq \widehat{\bigotimes}_{1\leq i\leq d}^{\rL}\rK(\gamma_i-1;R_S^+[Y_i]^{\wedge}_{\pd}\cdot T_i^{\alpha_i}).\]
       Using Lemma \ref{lem:MW} (1) and (2) again, we may conclude the result from the same argument as above.

       For the ``moreover'' part, by \cite[Cor. 6.5]{BMS18}, we have to show that for any $k\geq 1$, the
       \[\rH^k(\Gamma,(\calM^+\otimes_{\widehat \calO_{X_S}^+} \calO\widehat \bC^+_{\pd,S})(X_{\infty})) = \rH^k(\Gamma,M_{\infty}\otimes_{\widehat R_{\infty,S}^+}P_{\infty,S}^+)\]
       is killed by $\rho_K(\zeta_p-1)$. But this follows from Theorem \ref{thm:local Simpson}(1) immediately.
   \end{proof}
   \begin{proof}[\textbf{Proof of Theorem \ref{thm:integral Simpson}}]
       We proceed as in the proof of of \cite[Th. 4.4]{MW-AIM}.

       For Item (1): Let $\calM^+$ be a Hitchin-small integral $v$-bundle on $X_{S,v}$ of rank $r$ with the corresponding covering $\{\frakX_{i,S}\to\frakX_S\}_{i\in I}$. As argued in the proof of \cite[Th. 4.4]{MW-AIM}, the $\rL\eta_{\rho_K(\zeta_p-1)}\rR\nu_*(\calM^+\otimes_{\widehat \calO_{X_S}^+}\calO\widehat \bC_{\pd,S}^+)$ is the sheafification of the presheaf
       \[\frakU_S\in\frakX_{S,\et}\mapsto \rL\eta_{\rho_K(\zeta_p-1)}\rR\Gamma(U_v,\calM^+\otimes_{\widehat \calO_{X_S}^+}\calO\widehat \bC_{\pd,S}^+)\]
       while $\nu_*(\calM^+\otimes_{\widehat \calO_{X_S}^+}\calO\widehat \bC_{\pd,S}^+)$ is the sheafification of the presheaf
       \[\frakU_S\in\frakX_{S,\et}\mapsto \rH^0(U_v,\calM^+\otimes_{\widehat \calO_{X_S}^+}\calO\widehat \bC_{\pd,S}^+).\]
       To conclude the desired quasi-isomorphism
       \[\rL\eta_{\rho_K(\zeta_p-1)}\rR\nu_*(\calM^+\otimes_{\widehat \calO_{X_S}^+}\calO\widehat \bC_{\pd,S}^+)\simeq \nu_*(\calM^+\otimes_{\widehat \calO_{X_S}^+}\calO\widehat \bC_{\pd,S}^+)[0],\]
       as this is local problem on $\frakX_{S,\et}$, we are reduced to showing that for any small semi-stable $\frakU_S = \Spf(R_S^+) \in \frakX_{S,\et}$ lying over some $\frakX_{i,S}$, the followings are true:
       \begin{enumerate}
           \item[(i)] There is a quasi-isomorphism 
           \[\rL\eta_{\rho_K(\zeta_p-1)}\rR\Gamma(U_v,\calM^+\otimes_{\widehat \calO_{X_S}^+}\calO\widehat \bC_{\pd,S}^+)\simeq \rH^0(U_v,\calM^+\otimes_{\widehat \calO_{X_S}^+}\calO\widehat \bC_{\pd,S}^+)[0].\]

           \item[(ii)] The $H^+:=\rH^0(U_v,\calM^+\otimes_{\widehat \calO_{X_S}^+}\calO\widehat \bC_{\pd,S}^+)$ is a finite free $R_S^+$-module of rank $r$ such that the induced Higgs field $\theta$ from $\nu_*(\Theta_{\calM^+})$ on $H^+$ makes $(H^+,\theta)$ a Hitchin-small twisted Higgs module over $R_S^+$.
       \end{enumerate}
       As $\calM^+(U_{\infty,S})$ is a Hitchin-small $\Gamma$-representation over $\widehat R_{\infty,S}^+$ of rank $r$, the Item (i) above follows from Corollary \ref{cor:reduce to group cohomology} while the Item (ii) follows from Theorem \ref{thm:local Simpson}.

       For Item (2): Let $(\calH^+,\theta)$ be a twisted Hitchin-small Higgs bundle on $X_{S,\et}$ of rank $r$. Then there exists an \'etale covering $\{\frakX_{i,S} = \Spf(R_{i,S}^+)\to\frakX_S\}_{i\in I}$ by small semi-stable $\frakX_i$'s such that the evaluation $(H_i^+,\theta_i)$ of $(\calH^+,\theta)$ at $\frakX_{i,S}$ is a twisted Hitchin-small Higgs module over $R_{i,S}^+$ of rank $r$. Put $\calM^+:=\calM^+(\calH^+,\theta)$ for short. By Theorem \ref{thm:local Simpson}, we have $\calM(X_{i,S,\infty})$ is a Hitchin-small $\Gamma$-representation of rank $r$ over $\widehat R^+_{i,\infty,S}$, yielding that $\calM^+$ is a Hitchin-small integral $v$-bundle of rank $r$ as desired.

       For Item (3): Using Proposition \ref{prop:multiplication on OC}, similar to the proof of Theorem \ref{thm:local Simpson}(3), for any $\calM^+\in \rL\rS^{\Hsmall}(\frakX_S,\widehat \calO_{X_S}^+)$ and any $(\calH^+,\theta)\in\HIG^{\tHsmall}(\frakX_S,\calO_{\frakX_S})$, one can construct canonical morphisms
       \[\iota_{\calM^+}:\calM^+(\calH^+(\calM^+),\theta)\to\calM^+\]
       and
       \[\iota_{(\calH^+,\theta)}:(\calH^+(\calM^+(\calH^+,\theta)),\theta)\to(\calH^+,\theta)\]
       of integral $v$-bundles and Higgs bundles respectively. To conclude, it suffices to show these two maps are both isomorphism. But this is again a local problem, and thus we are reduced to Theorem \ref{thm:local Simpson}(3) and the proof therein. By standard linear algebra, the equivalence preserves tensor products and dualities. This completes the proof.
   \end{proof}
   
\subsubsection{Cohomological comparison}
   Let $\calM^+$ be a Hitchin-small integral $v$-bundle with the associated twisted Hitchin-small Higgs bundle $(\calH^+,\theta)$ in the sense of Theorem \ref{thm:integral Simpson}. It remains to compare the complexes 
   \[\rD\rR(\calH^+,\theta) \text{ and }\rR\nu_*(\calM^+).\]
   By Poincar\'e's Lemma \ref{thm:Poincare Lemma}, we have a quasi-isomorphism
   \[\rR\nu_*(\calM^+)\xrightarrow{\simeq}\rR\nu_*(\rD\rR(\calM^+\otimes_{\widehat \calO_{X_S}^+}\calO\widehat \bC_{\pd,S}^+,\Theta_{\calM^+})).\]
   By Theorem \ref{thm:integral Simpson}(1), we have 
   \[\rD\rR(\calH^+,\theta) = \nu_*(\rD\rR(\calM^+\otimes_{\widehat \calO_{X_S}^+}\calO\widehat \bC_{\pd,S}^+,\Theta_{\calM^+})).\]
   Thus, we get a canonical morphism
   \begin{equation}\label{equ:canonical map in cohomological coparison}
       \rD\rR(\calH^+,\theta) \to \rR\nu_*(\calM^+).
   \end{equation}
   \begin{cor}\label{cor:bounded torsion}
       Recall $\frakX$ has the relative dimension $d$ over $\calO_C$.
       The canonical map \eqref{equ:canonical map in cohomological coparison} has cofiber killed by $(\rho_K(\zeta_p-1))^{D}$ where $D = \max(d+1,2(d-1))$.
   \end{cor}
   \begin{proof}
       This follows from the same argument in the proof of \cite[Cor. 4.5]{MW-AIM}.
   \end{proof}

   On the other hand, as $\nu_*(\calM^+)\simeq \rH^0(\rD\rR(\calH^+,\theta))\subset\calH^+$ is $p$-torsionfree, by \cite[Lem. 6.10]{BMS18}, there exists a canonical map
   \begin{equation}\label{equ:canonical map in BMS}
       \rL\eta_{\rho_K(\zeta_p-1)}\rR\nu_*\calM^+\to \rR\nu_*\calM^+.
   \end{equation}
   Similar to \cite[Th. 5.4]{MW-AIM}, we have the following cohomological comparison theorem.
   \begin{thm}\label{thm:truncated comparison}
       The natural morphism 
       \[\tau^{\leq 1}\rD\rR(\calH^+,\theta)\to\rD\rR(\calH^+,\theta) \xrightarrow{\eqref{equ:canonical map in cohomological coparison}} \rR\nu_*(\calM^+)\]
       uniquely factors over the composite 
       \[\tau^{\leq 1}\rL\eta_{\rho_K(\zeta_p-1)}\rR\nu_*\calM^+\to\rL\eta_{\rho_K(\zeta_p-1)}\rR\nu_*\calM^+\xrightarrow{\eqref{equ:canonical map in BMS}} \rR\nu_*\calM^+\] and induces a quasi-isomorphism
       \[\tau^{\leq 1}\rD\rR(\calH^+,\theta)\xrightarrow{\simeq}\tau^{\leq 1}\rL\eta_{\rho_K(\zeta_p-1)}\rR\nu_*\calM^+.\]
   \end{thm}
   \begin{proof}
       To show the map $\tau^{\leq 1}\rD\rR(\calH^+,\theta)\to\rR\nu_*(\calM^+)$ uniquely factors through $\tau^{\leq 1}\rL\eta_{\rho_K(\zeta_p-1)}\rR\nu_*\calM^+$, by \cite[Lem. 8.16]{BMS18}, we have to show the induced map
       \[\rH^1(\rD\rR(\calH^+,\theta))\to\rR^1\nu_*(\calM^+)\]
       factors through $\rho_K(\zeta_p-1)\rR^1\nu_*(\calM^+)$. Since the problem is local on $\frakX_{S,\et}$, we may assume $\frakX_S = \Spf(R_S^+)$ is small semi-stable and then apply Proposition \ref{prop:local integral comparison}. It remains to show this map
       \[\tau^{\leq 1}\rD\rR(\calH^+,\theta)\to\tau^{\leq 1}\rL\eta_{\rho_K(\zeta_p-1)}\rR\nu_*\calM^+\]
       is a quasi-isomorphism. But this is again a local problem on $\frakX_{S,\et}$ and thus we can conclude by using Proposition \ref{prop:local integral comparison} again.
   \end{proof}
   By letting $\calM^+ = \widehat \calO^+_{X_S}$ in Theorem \ref{thm:truncated comparison}, we obtain a quasi-isomorphism
   \begin{equation}\label{equ:DI}
       \gamma_1:\calO_{\frakX_S}\oplus\Omega^{1,\log}_{\frakX_S}\{-1\}[-1] = \tau^{\leq 1}\rD\rR(\calO_{\frakX_S},0)\xrightarrow{\simeq}\tau^{\leq 1}\rL\eta_{\rho_K(\zeta_p-1)}\rR\nu_*\widehat \calO^+_{X_S}.
   \end{equation}
   Using this, we have the following analogue of Deligne--Illusie decomposition \cite[Th. 2.1]{DI} in mixed characteristic case for semi-stable formal schemes.
   \begin{thm}\label{thm:DI}
       The quasi-isomorphism $\gamma_1$ above extends to a quasi-isomorphism
       \[\gamma:\bigoplus_{i=0}^{p-1}\Omega^{i,\log}_{\frakX_S}\{-i\} = \tau^{\leq p-1}\rD\rR(\calO_{\frakX_S},0)\xrightarrow{\simeq}\tau^{\leq 1}\rL\eta_{\rho_K(\zeta_p-1)}\rR\nu_*\widehat \calO^+_{X_S}.\]
   \end{thm}
   \begin{proof}
       This follows from a standard trick of Deligne--Illusie. See for example \cite[Th. 4.1]{Min21}.
   \end{proof}
   \begin{rmk}\label{rmk:Dependence on log-lifting}
       Recall to obtain Theorem \ref{thm:DI}, we need $\frakX_S$ admits a lifting $\widetilde \frakX_S$ over $\AAKK(S)$ as a \emph{log-scheme} (as this is used to construct the period sheaf $(\calO\widehat \bC_{\pd,S}^+,\Theta)$). It seems that such a decomposition $\gamma$ in Theorem \ref{thm:DI} \emph{never} exists if we only assume $\frakX_S$ lifts to $\AAKK(S)$ as a \emph{scheme} without lifting the log-structure at the same time. The phenomenon also appears in \cite{SSt}. In \emph{loc.cit.}, for a semi-stable formal scheme $\frakX_0$ over $\rW(\kappa)$ with the canonical log-structure $\calM_{\frakX_0}$ and the special fiber $\calX$ over $\kappa$ with the induced log-structure $\calM_{\calX}$. Let $\rF:\calX\to \calX^{(1)}$ be the relative Frobenius map associated to $\calX$. Then there exists a quasi-isomorphism
       \[\oplus_{i=0}^{p-1}\Omega^{i,\log}_{\calX^{(1)}}[-i]\to\tau^{\leq p-1}\rF_*(\rD\rR(\calO_{\calX},\rd))\]
       if and only if $(\calX,\calM_{\calX})$ admits a lifting over $\frakS:=\rW(\kappa)[u]$ with the log-structure associated to $\bN\xrightarrow{1\mapsto u}\frakS$, which lifts the usual log-structure on $\rW(\kappa)$ induced by $(\bN\xrightarrow{1\mapsto p}\rW(\kappa))$ via the surjection $\frakS\xrightarrow{u\mapsto p}\rW(\kappa)$ (cf. \cite[Th. 2.9]{SSt}). Recall that the map $\frakG\xrightarrow{u\mapsto[\varpi]}\bfA_{2,K}$ lifts the natural morphism $\rW(\kappa)\hookrightarrow\calO_C$ which is compatible with log-structures as well. This also suggests that the the canonical log-structure on $\AAKK(S)$ should be the \emph{only} reasonable one to obtain Theorem \ref{thm:DI}.
   \end{rmk}
   In general, we make the following conjecture.
   \begin{conj}\label{conj:OV}
       Let $\frakX_S$ be a liftable semi-stable formal scheme over $A^+$ with a fixed lifting $\widetilde \frakX_S$ over $\AAKK(S)$. For any Hitchin-small integral $v$-bundle $\calM$ on $X_{S,v}$ with the associated twisted Hitchin-small Higgs bundle $(\calH^+,\theta)$ in the sense of Theorem \ref{thm:integral Simpson}, if we denote by $r$ the nilpotency length of $(\zeta_p-1)\theta$ modulo $p$, then the canonical morphism 
       \[\rD\rR(\calH^+,\theta)\to \rR\nu_*\calM^+\]
       induces a quasi-isomorphism
       \[\tau^{\leq p-r}\rD\rR(\calH^+,\theta)\simeq \tau^{\leq p-r}\rL\eta_{\rho_K(\zeta_p-1)}\rR\nu_*\calM^+.\]
   \end{conj}
   We remark that Theorem \ref{thm:DI} tells us Conjecture \ref{conj:OV} holds true when $r = 1$. Another special case of Conjecture \ref{conj:OV} we are able to prove is the following result:
   \begin{thm}\label{thm:curve case}
       Let $\frakX_S$ be a liftable semi-stable \emph{curve} over $A^+$ with a fixed lifting $\widetilde \frakX_S$ over $\AAKK(S)$. For any Hitchin-small integral $v$-bundle $\calM$ on $X_{S,v}$ with the associated twisted Hitchin-small Higgs bundle $(\calH^+,\theta)$ in the sense of Theorem \ref{thm:integral Simpson}, the canonical morphism \eqref{equ:canonical map in cohomological coparison} induces a quasi-isomorphism
       \[\rD\rR(\calH^+,\theta)\simeq \rL\eta_{\rho_K(\zeta_p-1)}\rR\nu_*\calM^+.\]
   \end{thm}
   \begin{proof}
       When $\frakX_S$ is a curve, both $\rD\rR(\calH^+,\theta)$ and $\rL\eta_{\rho_K(\zeta_p-1)}\rR\nu_*\calM^+$ are concentrated in degree $[0,1]$. So the result follows from Theorem \ref{thm:truncated comparison} immediately.
   \end{proof}

\subsection{The stacky Simpson correspondence}
   This part is devoted to establishing an equivalence between the stacks of Hitchin-small $v$-bundles on $X_{S,v}$ and Hitchin-small rational Higgs bundles on $\frakX_{S,\et}$, generalizing the previous work of Ansch\"utz--Heuer--Le Bras \cite[Th. 1.1]{AHLB-small} to the semi-stable reduction case.

   We first make the following definitions.
   \begin{dfn}\label{dfn:Hitchin small v-bundle}
       By a \emph{Hitchin-small $v$-bundle of rank $r$} on $X_{S,v}$, we mean a sheaf of locally finite free $\widehat \calO_{X_S}$-modules $\calM$ on $X_{S,v}$ such that there exists an \'etale covering $\{\frakX_{i,S}\to \frakX_S\}_{i\in I}$ by small affine $\frakX_{i,S} = \Spf(R_{i,S}^+)$ such that for any $i\in I$, the $\calM(X_{i,\infty,S})$ is a Hitchin-small $\Gamma$-representation of rank $r$ over $\widehat R_{i,\infty,S}$ in the sense of Definition \ref{dfn:small Gamma-representation}. Denote by $\rL\rS^{\Hsmall}(\frakX_S,\widehat \calO_{X_S})$ the category of Hitchin-small $v$-bundles on $X_{S,v}$.
   \end{dfn}
   \begin{dfn}\label{dfn:Hitchin small rational Higgs bundle}
       By an \emph{rational Higgs bundle of rank $r$} on $\frakX_{S,\et}$, we mean a pair $(\calH,\theta)$ consisting of a sheaf $\calH$ of locally finite free $\calO_{\frakX_S}[\frac{1}{p}]$-modules of rank $r$ on $\frakX_{S,\et}$ and a Higgs field $\theta$ on $\calH$.
       For any rational Higgs bundle $(\calH,\theta)$, denote by $\rD\rR(\calH,\theta)$ the induced Higgs complex.
       A rational Higgs bundle $(\calH,\theta)$ is called \emph{(twisted) Hitchin-small} if there exists an \'etale covering $\{\frakX_{i,S}\to \frakX_S\}_{i\in I}$ by small affine $\frakX_{i,S} = \Spf(R_{i,S}^+)$ such that for any $i\in I$, the evaluation $(\calH,\theta)(\frakX_{i,S})$ is a (twisted) Hitchin-small Higgs modules over $R_{i,S}$ in the sense of Definition \ref{dfn:small Gamma-representation}(2). Denote by $\HIG^{(\text{t-})\Hsmall}(\frakX_S,\calO_{\frakX_S}[\frac{1}{p}])$ the category of (twisted) Hitchin-small rational Higgs bundles on $\frakX_{S,\et}$.
   \end{dfn}
   \begin{lem}\label{lem:Hitchin-small vs twisted Hitchin-small}
       The functor $(\calH,\theta)\mapsto (\calH,(\zeta_p-1)\theta)$ induces an equivalence of categories
       \[\HIG^{\Hsmall}(\frakX_S,\calO_{\frakX_S}[\frac{1}{p}])\xrightarrow{\simeq}\HIG^{\tHsmall}(\frakX_S,\calO_{\frakX_S}[\frac{1}{p}])\]
       such that for any twisted Hitchin-small $(\calH,\theta)$ with the induced $(\calH,\theta^{\prime} = (\zeta_p-1)\theta)$, there exists a quasi-isomorphism
       \[\rD\rR(\calH,\theta)\simeq\rD\rR(\calH,\theta^{\prime}).\]
   \end{lem}
   \begin{proof}
       By unwinding definitions, the result follows from Lemma \ref{lem:global twist functor} immediately.
   \end{proof}
   
   We give a remark on rational Higgs bundles on $\frakX_{S,\et}$.
   Let $i: X_{S,\et}\to\frakX_{S,\et}$ be the natural morphism of sites. Clearly, for any locally finite free $\calO_{X,S}$-module $\calE$ of rank $r$ on $X_{\et,S}$, we have 
   \[\rR i_*\calE = i_*\calE\] 
   which is a locally finite free $\calO_{\frakX_S}[\frac{1}{p}]$-module of rank $r$ and $i_*\calO_{X_S} = \calO_{frakX_S}[\frac{1}{p}]$. So the functors $i_*$ and $i^{-1}$ induces an equivalence between the category of locally finite free $\calO_{X_S}$-module on $X_{S,\et}$ and the category of locally finite free $\calO_{\frakX_S}$-module on $\frakX_{S,\et}$, yielding an equivalence of categories
   \begin{equation}\label{equ:Higgs bundles on generic fiber}
       \HIG(X_S,\calO_{X_S}) \simeq \HIG(\frakX_S,\calO_{\frakX_S}[\frac{1}{p}]).
   \end{equation}
   Here, $\HIG(X_S,\calO_{X_S})$ denotes the category of Higgs bundles on $X_{S,\et}$. Denote by
   \[\HIG^{(\text{t-})\Hsmall}(\frakX_S,\calO_{X_S})\subset \HIG(X_S,\calO_{X_S})\]
   the full sub-category corresponding to $\HIG^{(\text{t-})\Hsmall}(\frakX_S,\calO_{\frakX_S}[\frac{1}{p}])$ via the above equivalence \eqref{equ:Higgs bundles on generic fiber}; that is, we have an equivalence of categories induced by $i_*$ and $i^{-1}$:
   \begin{equation}\label{equ:Hitchin-small Higgs bundles on generic fiber}
       \HIG^{(\text{t-})\Hsmall}(\frakX_S,\calO_{X_S})\simeq \HIG^{(\text{t-})\Hsmall}(\frakX_S,\calO_{\frakX_S}[\frac{1}{p}]).
   \end{equation}

   The following theorem is the analogue of Theorem \ref{thm:integral Simpson} on the rational level.
   
  \begin{thm}\label{thm:rational Simpson}
            Let $(\calO\widehat \bC_{\pd,S},\Theta)$ be the period sheaf with Higgs field associated to $\widetilde \frakX_S$. Let $\nu: X_{S,v}\to \frakX_{S,\et}$ be the natural morphism of sites.
      \begin{enumerate}
          \item[(1)] For any $\calM\in \rL\rS^{\Hsmall}(\frakX_S,\widehat \calO_{X_S})$ of rank $r$, put 
          \[\Theta_{\calM}:=\id\otimes\Theta:\calM\otimes\calO\widehat \bC_{\pd,S}\to \calM\otimes\calO\widehat \bC_{\pd,S}\otimes_{\calO_{\frakX_S}}\Omega^{1,\log}_{\frakX_S}\{-1\}.\] 
          Then we have a quasi-isomorphism 
          \[\rR\nu_*(\calM\otimes\calO\widehat \bC_{\pd,S}^+)\simeq\nu_*(\calM\otimes\calO\widehat \bC_{\pd,S}^+)[0]\]
          Moreover, the push-forward
          \[(\calH(\calM),\theta):=(\nu_*(\calM\otimes\calO\widehat \bC_{\pd,S}),\nu_*(\Theta_{\calM}))\]
          defines a twisted Hitchin-small rational Higgs bundle of rank $r$ on $\frakX_{S,\et}$.
        
          \item[(2)] For any $(\calH,\theta)\in \HIG^{\tHsmall}(\frakX_S,\calO_{\frakX_S}[\frac{1}{p}])$ of rank $r$, put
          \[\Theta_{\calH}=\theta\otimes\id+\id\otimes\Theta:\calH\otimes\calO\widehat \bC_{\pd,S}\to \calH\otimes\calO\widehat \bC_{\pd,S}\otimes_{\calO_{\frakX_S}}\Omega^{1,\log}_{\frakX_S}\{-1\}.\]
          Then the
          \[\calM(\calH,\theta):=(\calH\otimes\calO\widehat \bC_{\pd,S})^{\Theta_{\calH} = 0}\]
          defines a Hitchin-small integral $v$-bundle of rank $r$ on $X_{S,v}$.
          
          \item[(3)] The functors in Items (1) and (2) defines an equivalence of categories
          \[\rL\rS^{\Hsmall}(\frakX_S,\widehat \calO_{X_S})\xrightarrow{\simeq}\HIG^{\tHsmall}(\frakX_S,\calO_{\frakX_S}[\frac{1}{p}]).\]
          which preserves ranks, tensor products and dualities. Moreover, for any $\calM\in \rL\rS^{\Hsmall}(\frakX_S,\widehat \calO_{X_S})$ with associated $(\calH,\theta)$, there exists a quasi-isomorphism
          \[\rR\nu_*\calM\simeq \rD\rR(\calH,\theta).\]
      \end{enumerate}
  \end{thm}
  \begin{proof}
      For Item (1): Let $\calM$ be a Hitchin-small $v$-bundle of rank $r$ with the associated \'etale covering $\{\frakX_{i,S}\to\frakX_S\}_{i\in I}$ by small semi-stable $\frakX_i$'s as in Definition \ref{dfn:Hitchin small v-bundle}. Then by definition \ref{dfn:small O^+-representations}, the restriction $\calM_{|_{X_{i,S}}}$ is of the form $\calM_{|_{X_{i,S}}} = \calM_i^+[\frac{1}{p}]$ for some Hitchin-small integral $v$-bundle on $X_{i,S,v}$.

      To see $\rR\nu_*(\calM\otimes\calO\widehat \bC_{\pd,S})\simeq \nu_*(\calM\otimes\calO\widehat \bC_{\pd,S})$, it is enough to show that for any $k\geq 1$, we have
      \[\rR^k\nu_*(\calM\otimes\calO\widehat \bC_{\pd,S}) = 0.\]
      As $\rR^k\nu_*(\calM\otimes\calO\widehat \bC_{\pd,S})$ is the sheafification of the presheaf
      \[\frakU\in\frakX_{\et}\mapsto \rH^k(U_v,(\calM\otimes\calO\widehat \bC_{\pd,S})_{|_U}),\]
      to see $\rR^k\nu_*(\calM\otimes\calO\widehat \bC_{\pd,S}) = 0$, we may work locally on $\frakX_{\et}$. So one can check this on each $\frakX_i$. But this follows from Theorem \ref{thm:integral Simpson} (1) immediately.
      
      Let $(\calH_i^+,\theta_i)$ be the twisted Hitchin-small integral Higgs bundle on $\frakX_{i,S,\et}$ associated to $\calM^+_i$ in the sense of Theorem \ref{thm:integral Simpson}. It is clearly that
      \[(\calH(\calM),\theta)_{|_{\frakX_{i,S}}} = (\calH^+_i,\theta_i)[\frac{1}{p}].\]
      This implies $(\calH(\calM),\theta)$ is a twisted Hitchin-small rational Higgs bundle of rank $r$ as desired.

      For Item (2): It follows from a similar argument as above by using Theorem \ref{thm:integral Simpson}(2) directly.

      For Item (3): One can conclude the functors in Items (1) and (2) are the quasi-inverses of each other by the same argument for the proof of Theorem \ref{thm:integral Simpson}(3). So we get the desired equivalence of categories, which preserves ranks by construction. By standard linear algebra, this equivalence preserves tensor products and dualities. Finally, by Poincar\'e's Lemma, we have the quasi-isomorphisms
      \[\rR\nu_*(\calM)\xrightarrow{\simeq}\rR\nu_*(\rD\rR(\calM\otimes\calO\widehat \bC_{\pd,S},\Theta_{\calM})).\]
      On the other hand, by Item (1), we have a quasi-isomorphism
      \[\rD\rR(\calH,\theta)\xrightarrow{\simeq}\rR\nu_*(\rD\rR(\calM\otimes\calO\widehat \bC_{\pd,S},\Theta_{\calM})),\]
      yielding the quasi-isomorphism 
      \[\rR\nu_*(\calM)\simeq \rD\rR(\calH,\theta)\]
      as desired. This completes the proof.
  \end{proof}

  \begin{rmk}\label{rmk:rational Simpson}
      Denote by $\pi:X_{S,v}\to X_{S,\et}$ the natural morphism of sites, and then we have $\nu = i\circ \pi$. Via the equivalence \eqref{equ:Hitchin-small Higgs bundles on generic fiber}, it is easy to see that Theorem \ref{thm:rational Simpson} still holds true if one replaces $\nu$ and $\HIG^{\tHsmall}(\frakX_S,\calO_{\frakX_S}[\frac{1}{p}])$ by $\pi$ and $\HIG^{\tHsmall}(\frakX_S,\calO_{X_S})$, respectively.
  \end{rmk}
  
  \begin{cor}\label{cor:rational Simpson}
      The following composite
      \[\rL\rS^{\Hsmall}(\frakX_S,\widehat \calO_{X_S})\xrightarrow{\text{Th. \ref{thm:rational Simpson}}}\HIG^{\tHsmall}(\frakX_S,\calO_{\frakX_S}[\frac{1}{p}])\xrightarrow{\text{Lem. \ref{lem:Hitchin-small vs twisted Hitchin-small}}}\HIG^{\Hsmall}(\frakX_S,\calO_{\frakX_S}[\frac{1}{p}])\]
      defines an equivalence of categories 
      \[\rho_{\widetilde \frakX_S}:\rL\rS^{\Hsmall}(\frakX_S,\widehat \calO_{X_S})\xrightarrow{\simeq}\HIG^{\Hsmall}(\frakX_S,\calO_{\frakX_S}[\frac{1}{p}])\simeq \HIG^{\Hsmall}(\frakX_S,\calO_{X_S})\]
      which is functorial in $S$ such that for any Hitchin-small $v$-bundle $\calM$ on $X_{S,v}$ with associated Hitchin-small (rational) Higgs bundle $(\calH,\theta)$ on $X_{S,\et}$ (resp. $\frakX_{S,\et}$), there exists a quasi-isomorphism
      \[\rR\nu_*\calM\simeq \rD\rR(\calH,\theta).\]
  \end{cor}
  \begin{proof}
      This follows from Lemma \ref{lem:Hitchin-small vs twisted Hitchin-small} and Theorem \ref{thm:rational Simpson} directly.
  \end{proof}

  For any $r$, we denote by $(\rR^1\pi_*\GL_r)(X_S)$ and $\big((\Mat_r\otimes\Omega^1_{X_S}(-1))\!/\!/\GL_r\big)(X_S)$ the sheafifications of the presheaves
  \[U_S\in X_{S,\et}\mapsto\{\text{the set of isomorphic classes of $v$-bundles on $U_{S,v}$ of rank $r$}\}\]
  and
  \[U_S\in X_{S,\et}\mapsto\{\text{the set of isomorphic classes of Higgsbundles on $U_{S,\et}$ of rank $r$}\},\]
  respectively. By \cite[Th. 1.2]{Heu-Moduli}, there exists an isomorphism 
  \begin{equation}\label{equ:HTlog}
      \rm{HTlog}: (\rR^1\pi_*\GL_r)(X_S)\xrightarrow{\cong}\big((\Mat_r\otimes\Omega^1_{X_S}(-1))\!/\!/\GL_r\big)(X_S).
  \end{equation}
  See \cite[\S 5]{Heu-Moduli} for its construction and \cite[\S 4]{Heu-Moduli} for the construction of its inverse $\rm{HTlog}^{-1}$ (denoted by $\Psi$ in \emph{loc.cit.}).

  \begin{lem}\label{lem:compatible with HTlog}
      The equivalence $\rho_{\widetilde \frakX_S}:\rL\rS^{\Hsmall}(\frakX_S,\widehat \calO_{X_S})\xrightarrow{\simeq}\HIG^{\Hsmall}(\frakX_S,\calO_{X_S})$ is compatible with the isomorphism $\rm{HTlog}$ above in the sense that the following diagram commutes
      \begin{equation}\label{equ:compatible with HTlog}
          \begin{tikzcd}
               {\rL\rS^{\Hsmall}_r(\frakX_S,\widehat \calO_{X_S})} \arrow[rr, "\rho_{\widetilde \frakX_S}"] \arrow[d] &  & {\HIG^{\Hsmall}_r(\frakX_S,\calO_{X_S})} \arrow[d]          \\
               (\rR^1\pi_*\GL_r)(X_S) \arrow[rr, "\rm{HTlog}"]  &  & \big((\Mat_r\otimes\Omega^1_{X_S}(-1))\!/\!/\GL_r\big)(X_S)
          \end{tikzcd}
      \end{equation}
      where $\rL\rS^{\Hsmall}_r(\frakX_S,\widehat \calO_{X_S})$ and $\HIG^{\Hsmall}_r(\frakX_S,\calO_{X_S})$ denotes the corresponding full sub-categories of rank-$r$ objects.
  \end{lem}
  \begin{proof}
      We follow the strategy used in the proof of \cite[Cor. 3.12]{AHLB-small}.
      For any $\calM\in \rL\rS_r^{\Hsmall}(\frakX_S,\widehat \calO_{X_S})$ with associated $(\calH,\theta)\in \HIG_r^{\Hsmall}(\frakX_S,\calO_{X_S})$, we have to show that the (local) isomorphic class $c(\calM)$ (resp. $c(\calH,\theta)$) associated to $\calM$ (resp. $(\calH,\theta)$) satisfies the condition $\rm{HTlog}(c(\calM)) = c(\calH,\theta)$. 
      Clearly one can check this locally on $X_{S,\et}$. So we may assume $\frakX_S = \Spf(R_S^+)$ is small semi-stable such that $M_{\infty}:=\calM(X_{\infty,S})$ is a Hitchin-small $\Gamma$-representation of rank $r$ over $\widehat R_{\infty,S}$ with the associated Hicthin-small Higgs module $(H,\theta)$ over $R_S$. Let $M\in\Rep_{\Gamma}^{\Hsmall}(R_S)$ such that $M_{\infty}\cong M\otimes_{R_S}\widehat R_{\infty,S}$. Let $\theta_i\in \Mat_r(R_S^+)$ be the topologically nilpotent matrix such that $\gamma_i$ acts on $M$ via $\exp(-(\zeta_p-1)\rho_K\theta_i)$ with respect to a fixed $R_S$-basis $e_1,\dots,e_r$ of $M$. As argued in the paragraph around \eqref{equ:explicit form}, we have 
      \begin{equation}\label{equ:compatible with HTlog-I}
          (H,\theta)\cong (M,\sum_{i=1}^d(\zeta_p-1)\theta_i\otimes\frac{\dlog T_i}{\xi_K}).
      \end{equation}
      Let $\Spa(R_{1,S},R_{1,S}^+)=:X_{1,S}\to X_S$ be the finite \'etale Galois covering associated to the Galois group $\Gamma/\Gamma^p$, where $\Gamma^p = \oplus_{i=1}^d\Zp\cdot\gamma_i^p\subset \Gamma$ via the ismorphism \eqref{equ:Gamma group-II}. As the problem is local on $X_{S,\et}$ as well, we are reduced to showing that
      \[\rm{HTlog}(c(\calM_{|_{X_{1,S}}})) = c((\calH,\theta)_{|_{X_{1,S}}}).\]

      By \eqref{equ:compatible with HTlog-I}, the Higgs module over $R_{1,S}$ induced by $(\calH,\theta)_{|_{X_{1,S}}}$ is 
      \begin{equation}\label{equ:compatible with HTlog-II}
          (H\otimes_{R_S}R_{1,S},\theta\otimes\id)\cong (M,\sum_{i=1}^dp(\zeta_p-1)\theta_i\otimes\frac{\dlog T_i^{\frac{1}{p}}}{\xi_K}) = (M,\sum_{i=1}^dp(\zeta_p-1)\rho_K\theta_i\otimes\frac{\dlog T_i^{\frac{1}{p}}}{t})
      \end{equation}
      while the $\Gamma^p$-representation over $\widehat R_{\infty,S}$ from $\calM_{|_{X_{1,S}}}$ is the restriction of $M_{\infty}$ to $\Gamma^p$. 
      In \eqref{equ:compatible with HTlog-II}, we use the identification $\calO_C\{-1\} = \rho_K\calO_C(-1)$ (cf. \S\ref{ssec:notation}).
      In particular, the $\gamma_i^p$ acts on $M_{\infty}$ (with respect to the choose basis $e_1,\dots,e_r$) via the matrix 
      \begin{equation}\label{equ:compatible with HTlog-III}
          \gamma_i^p = \exp(-p(\zeta_p-1)\rho_K\theta_i),~\forall~1\leq i\leq d.
      \end{equation}
      Comparing \eqref{equ:compatible with HTlog-II} with \eqref{equ:compatible with HTlog-III}, it then follows from the construction of $\rm{HTlog}$ (cf. \cite[Prop. 5.3 and its proof]{Heu-Moduli}) that $\rm{HTlog}(c(\calM_{|_{X_{1,S}}})) = c((\calH,\theta)_{|_{X_{1,S}}})$. This completes the proof.
  \end{proof}

   Now, we are going to give a geometric interpretion of Hitchin-smallness. For any $r\geq 0$, we consider two functors on $\Perfd$:
   \begin{equation}\label{equ:stack of v-bundle}
       \rL\rS_r(X,\OX): S\in \Perfd \mapsto \{\text{the groupoid of $v$-bundles of rank $r$ on $X_{S,v}$}\}
   \end{equation}
   and
   \begin{equation}\label{equ:stack of Higgs bundle}
       \HIG_r(X,\calO_X): S\in \Perfd \mapsto \{\text{the groupoid of Higgs bundles of rank $r$ on $X_{S,\et}$}\}.
   \end{equation}
   By \cite[Th. 1.4]{Heu-Moduli}, these two functors are both small $v$-stacks (on $\Perfd$). Consider the $v$-sheaf
   \begin{equation}\label{equ:Hitchin-base}
       \calA_r: S\in \Perfd \mapsto \calA_r(S):=\oplus_{i=1}^r\rH^0(X_S,\Sym^i(\Omega^1_{X_S}\{-1\})),
   \end{equation}
   which is referred as \emph{Hitchin-base} in \emph{loc.cit.}.
   Then there are morphisms of $v$-sheaves, called \emph{Hitchin fibrations},
   \begin{equation}\label{equ:Hitchin map}
       h:\HIG_r(X,\calO_X)\to \calA_r \text{ and }\widetilde h:\rL\rS_r(X,\OX) \to \calA_r
   \end{equation}
   where $h$ is defined by sending each Higgs bundle $(\calH,\theta)$ to the characteristic polynomial of $\theta$ and $\widetilde h$ is defined by the composite of $h\circ\rm{HTlog}$ (cf. \cite[\S 1.3]{Heu-Moduli}). In our setting, one can define an open subset $\calA^{\Hsmall}_r\subset\calA_r$ such that
   \begin{equation}\label{equ:Hitchin-small Hitchin base}
       \calA_r^{\Hsmall}(S) = \oplus_{i=1}^rp^{<\frac{i}{p-1}}\rH^0(\frakX_S,\Sym^i(\Omega^{1,\log}_{\frakX_S}\{-1\})),
   \end{equation}
   where $p^{<\frac{i}{p-1}}$ denotes the ideal $(\zeta_p-1)^i\frakm_C\subset \calO_C$. For any $v$-stack $Z$ over $\calA_r$, define its \emph{Hitchin-small locus} by
   \begin{equation}\label{equ:Hitchin-small locus}
       Z^{\Hsmall}:=Z\times_{\calA_r}\calA_r^{\Hsmall}.
   \end{equation}
   In particular, we have $\rL\rS_r(X,\OX)^{\Hsmall}$ and $\HIG_r(X,\calO_X)^{\Hsmall}$.

   \begin{prop}\label{prop:geometric interpretion of Hitchin-small}
       Keep notations as above. Then we have
       \begin{enumerate}
           \item[(1)] $\HIG^{\Hsmall}(\frakX_S,\calO_{X_S}) = \cup_{r\geq 0}\HIG_r(X,\calO_X)^{\Hsmall}(S)$, and

           \item[(2)] $\rL\rS^{\Hsmall}(\frakX_S,\widehat \calO_{X_S}) = \cup_{r\geq 0}\rL\rS_r(X,\OX)^{\Hsmall}(S)$.
       \end{enumerate}
       More precisely, a Higgs bundle $(\calH,\theta)$ (resp. $v$-bundle $\calM$) of rank $r$ on $X_{S,\et}$ (resp. $X_{S,v}$) is Hitchin-small if and only if as an $S$-point of $\HIG_r(X,\calO_X)$ (resp. $\rL\rS(X,\OX)$),
       \[(\calH,\theta)\in\HIG_r(X,\calO_X)^{\Hsmall}(S) \text{ (resp. $\calM\in\rL\rS(X,\OX)^{\Hsmall}(S)$)}.\]
   \end{prop}
   \begin{proof}
       It suffcies to prove Item (1) while Item (2) follows by Lemma \ref{lem:compatible with HTlog}. For any Hitchin-small Higgs bundle $(\calH,\theta)\in \HIG^{\Hsmall}(\frakX_S,\calO_{X_S})$ of rank $r$, we regard it as a Hitchin-small rational Higgs bundle on $\frakX_{S,\et}$. It follows from Definition \ref{dfn:Hitchin small rational Higgs bundle} that as $S$-point of $\HIG_r(X,\calO_X)$, 
       \[(\calH,\theta)\in \HIG_r(X,\calO_X)^{\Hsmall}(S).\]
       It remains to show for any Higgs bundle $(\calH,\theta)\in \HIG_r(X,\calO_X)^{\Hsmall}(S)$ (again viewed as a rational Higgs bundle in $\HIG(\frakX_S,\calO_{\frakX_S}[\frac{1}{p}])$), \'etale locally on $\frakX_{S,\et}$, it is of the form $(\calH^+,\theta)[\frac{1}{p}]$ for some Hitchin-small integral Higgs bundle $(\calH^+,\theta)$. To do so, we may assume $\frakX_S = \Spf(R_S^+)$ is small semi-stable such that $(\calH,\theta)$ is induced by a Higgs module $(H,\theta = \sum_{i=1}^d\theta_i\otimes\frac{e_i}{\xi_K})$ of rank $r$ over $R_S$ (such that $\theta_i\in\Mat_r(R_S)$ with respect to a fixed $R_S$-basis of $H$). As $(\calH,\theta)\in \HIG_r(X,\calO_X)^{\Hsmall}(S)$, for any $1\leq i\leq d$, the $\theta_i$ has eigenvalues in $(\zeta_p-1)\frakm_CR_S^+$. As $\theta_i$'s commute with each others, by standard linear algebra, there exists a matrix $X\in\GL_r(\overline{\Frac(R_S)})$ such that $\theta_i^{\prime}:=X\theta_iX^{-1}\in (\zeta_p-1)\frakm_C\Mat_r(R_S^+)$ for all $i$, where $\overline{\Frac(R_S)}$ denotes the algebraic closure of the fractional field of $R_S$. Thus $X$ induces an isomorphism 
       \[(H,\theta)\cong (H^{\prime},\theta^{\prime} = \sum_{i=1}^d\theta_i^{\prime}\otimes\frac{e_i}{\xi_K})\]
       for some Hitchin-small Higgs module $(H^{\prime},\theta^{\prime})$ over $R_S$.
   \end{proof}

  Now, we are able to give the following equivalence of stacks, generalizing the previous work of \cite[Th. 1.1]{AHLB-small}.
  \begin{thm}\label{thm:stacky Simpson}
      Let $\frakX$ be liftable a semi-stable formal scheme over $\calO_C$ with a fixed lifting $\widetilde \frakX$ over $\bfA_{2,K}$. Then for any $r\geq 0$, there exists an equivalence of stacks
      \[\rho_{\widetilde \frakX}:\rL\rS_r(X,\OX)^{\Hsmall}\xrightarrow{\simeq}\HIG_r(X,\calO_X)^{\Hsmall}.\]
  \end{thm}
  \begin{proof}
      For our purpose, by Proposition \ref{prop:geometric interpretion of Hitchin-small}, it is enough to assign to each $S\in \Perfd$ a rank-preserving equivalence of categories
      \[\rho_{\widetilde \frakX_S}:\rL\rS^{\Hsmall}(X_S,\widehat \calO_{X_S})\simeq \HIG^{\Hsmall}(X_S,\calO_{X_S})\]
      which is functorial in $S$. But this follows from Corollary \ref{cor:rational Simpson} directly.
  \end{proof}


\begin{thebibliography}{99}

  \bibitem[AGT16]{AGT} Ahmed Abbes, Michel Gros, and Takeshi Tsuji: {\it The $p$-adic Simpson Correspondence (AM-193)}, volume 193. Princeton University Press, 2016.

  \bibitem[AHLB23a]{AHLB23} Johannes Ansch\"utz, Ben Heuer, and Arthur-C\'esar Le Bras: {\it Hodge-Tate stacks and non-abelian $p$-adic Hodge theory of $v$-perfect complexes on smooth rigid spaces}, arXiv preprint arXiv:2302.12747, 2023.
  
  \bibitem[AHLB23b]{AHLB-small} Johannes Ansch\"utz, Ben Heuer, and Arthur-C\'esar Le Bras: {\it The small $p$-adic Simpson correspondence in terms of moduli spaces}, arXiv preprint arXiv:2312.07554, 2023. Accepted by Math. Res. Lett.
  
  \bibitem[Bha19]{Bha-LectNote} Bhargav Bhatt: {\it Lecture 4: Perfect prism and perfectoid rings}, \href{http://www-personal.umich.edu/~bhattb/teaching/prismatic-columbia/lecture4-perfect-prisms-and-perfectoids.pdf}{http://www-personal.umich.edu/~bhattb/teaching/prismatic-columbia/lecture4-perfect-prisms-and-perfectoids.pdf}, 2019.
  
  \bibitem[BL22a]{BL22a} Bhargav Bhatt and Jacob Lurie: {\it Absolute prismatic cohomology}, arXiv preprint arXiv:2201.06120, 2022.

  \bibitem[BL22b]{BL22b} Bhargav Bhatt and Jacob Lurie: {\it The prismatization of $p$-adic formal schemes}, arXiv preprint arXiv:2201.06124, 2022.
  
  \bibitem[BMS18]{BMS18} Bhargav Bhatt, Matthew Morrow, and Peter Scholze {\it Integral p-adic Hodge theory}, Publications math\'ematiques de l’IHES , 128(1):219–397, 2018.
  
  \bibitem[BS22]{BS22} Bhargav Bhatt and Peter Scholze: {\it Prisms and prismatic cohomology}, Ann. of Math. (2), 196(3):1135–1275, 2022.
  
  \bibitem[CK19]{CK19} Kestutis \v Cesnavi\v cius and Teruhisa Koshikawa: {\it The $\Ainf$-cohomology in the semistable case}, Compos. Math., 155(11):2039–2128, 2019.
  
  \bibitem[Col02]{Col02} Pierre Colmez: {\it Espaces de Banach de dimension finie}, J. Inst. Math. Jussieu, 1(3):331–439, 2002.

  \bibitem[CS24]{CS24} Kestutis \v Cesnavi\v cius and Peter Scholze: {\it Purity for flat cohomology}, Annals of Mathematics, 199(1):51–180, 2024.
  
  \bibitem[DH23]{DH} Gabriel Dorfsman-Hopkins: {\it Untilting line bundles on perfectoid spaces}, Int. Math. Res. Not. IMRN,(3):2572–2591, 2023.
  
  \bibitem[DI87]{DI} Pierre Deligne and Luc Illusie: {\it Rel\'evements modulo $p$ et d\'ecomposition du complexe de de Rham}, Invent. Math., 89:247–270, 1987.
  
  \bibitem[DLLZ23]{DLLZ} Hansheng Diao, Kai-Wen Lan, Ruochuan Liu, and Xinwen Zhu: {\it Logarithmic Riemann–Hilbert correspondences for rigid varieties}, Journal of the American Mathematical Society, 36(2):483–562, 2023.
  
  \bibitem[Fal05]{Fal} Gerd Faltings: {\it A $p$-adic Simpson correspondence}, Advances in Mathematics, 198(2):847–862, 2005.

  \bibitem[Heu22]{Heu-Moduli} Ben Heuer: {\it Moduli spaces in $p$-adic non-abelian Hodge theory}, arXiv preprint arXiv:2207.13819, 2022.
  
  \bibitem[Heu23]{Heu-Simpson} Ben Heuer: {\it A $p$-adic Simpson correspondence for smooth proper rigid varieties}, arXiv:2307.01303, 2023.
  
  \bibitem[Hub93]{Hub93} Roland Huber: {\it Continuous valuations}, Math. Z., 212(3):455–477, 1993.
  
  \bibitem[Hub94]{Hub94} Roland Huber: {\it A generalization of formal schemes and rigid analytic varieties}, Math. Z., 217(4):513–551, 1994.

  \bibitem[HX24]{HX} Ben Heuer and Daxin Xu: {\it $p$-adic non-abelian Hodge theory for curves via moduli stacks}, arXiv preprint arxiv:2402.01365, 2024.
  
  \bibitem[LZ17]{LZ} Ruochuan Liu and Xinwen Zhu: {\it Rigidity and a Riemann-Hilbert correspondence for $p$-adic local systems}, Invent. Math., 207(1):291–343, 2017.
   
  \bibitem[Min21]{Min21} Yu Min: {\it Integral $p$-adic Hodge theory of formal schemes in low ramification}, Algebra Number Theory, 15(4):1043–1076, 2021.
  
  \bibitem[MT20]{MT} Matthew Morrow and Takeshi Tsuji: {\it Generalised representations as $q$-connections in integral $p$-adic Hodge theory}, arXiv:2010.04059v2, 2020.
  
  \bibitem[MW22]{MW-JEMS} Yu Min and Yupeng Wang: {\it p-adic Simpson correpondence via prismatic crystals}, arXiv preprint arXiv:2201.08030v4, 2022. Accepted by J. Eur. Math. Soc..
  
  \bibitem[MW24]{MW-AIM} Yu Min and Yupeng Wang: {\it Integral $p$-adic non-abelian Hodge theory for small representations}, Adv. Math., 458:109950, 2024. %Available at \href{https://www.sciencedirect.com/science/article/pii/S0001870824004651}{https://www.sciencedirect.com/science/article/pii/S0001870824004651}.
  
  \bibitem[Ols05]{Ols} Martin C Olsson: {\it The logarithmic cotangent complex}, Mathematische Annalen, 333:859–931, 2005.
  
  \bibitem[Sch13]{Sch-Pi} Peter Scholze: {\it $p$-adic Hodge theory for rigid-analytic varieties}, Forum Math. Pi, 1:e1, 77, 2013.

  \bibitem[Sch17]{Sch-Diamond} Peter Scholze: {\it \'Etale cohomology of diamonds}, arXiv preprint arXiv:1709.07343, 2017.
  
  \bibitem[SS20]{SSt} Mao Sheng and Junchao Shentu: {\it On $\rE_1$-degeneration for the special fiber of a semistable family}, Commun. Number Theory Phys., 14(3):555–584, 2020.
  
  \bibitem[Sta24]{Sta} The Stacks project authors: {\it The stacks project}, \href{https://stacks.math.columbia.edu}{https://stacks.math.columbia.edu}, 2024.
  
  \bibitem[SZ18]{SZ} Tam\'as Szameuly and Gergely Z\'abr\'adi: {\it A $p$-adic Hodge decomposition according to Beilinson}, Algebraic geometry: Salt Lake City 2015, 75:495–572, Proc. Sympos. Pure Math., 97.2, Amer. Math. Soc., Providence, RI,, 2018.
  
  \bibitem[Tia23]{Tian} Yichao Tian: {\it Finiteness and duality for the cohomology of prismatic crystals}, J. Reine Angew. Math., 800:217–257, 2023.
  
  \bibitem[Tsu18]{Tsu} Takeshi Tsuji: {\it Notes on the local $p$-adic Simpson correspondence}, Mathematische Annalen, 371(1-2):795–881, 2018.
  
  \bibitem[Wan23]{Wan23} Yupeng Wang: {\it A $p$-adic Simpson correspondence for rigid analytic varieties}, Algebra Number Theory, 17(8):1453–1499, 2023.
    
\end{thebibliography}
\end{document}